\documentclass[15pt]{amsart}
\usepackage{amsmath}
\usepackage{mathtools}
\usepackage{}
\usepackage{graphicx}
\usepackage[colorlinks=true, allcolors=blue]{hyperref}
\usepackage{amsfonts}
\usepackage{amsthm}
\usepackage{newlfont}
\usepackage{amscd}
\usepackage{amsgen}
\usepackage{amssymb}
\usepackage{mathrsfs}	
\usepackage{longtable}
\usepackage{listings}
\usepackage{extarrows}
\usepackage{tikz}
\usepackage{tikz-cd}
\usepackage{verbatim}
\numberwithin{equation}{section}
\usepackage[all]{xy}
\usepackage{color}
\usepackage{amssymb}
\usepackage[left=2cm,right=2cm]{geometry}
\usepackage{tikz-cd}
\usepackage{mathtools}
\usepackage{dynkin-diagrams}
\usetikzlibrary{matrix,shapes,arrows,decorations.pathmorphing}
\usepackage{calligra}
\usepackage{mathrsfs}
\usepackage{enumitem} % funny enumerate environments
\usepackage{tgpagella} %font
\usepackage[parfill]{parskip} %noindent everywhere
\usepackage{float}%override random table positioning
\restylefloat{table}

\let\blb\mathbb
\def\CC{{\blb C}}

\def \PP{{\blb P}}

\def \ZZ{{\blb Z}}

\def \RR{{\blb R}}

\def \LL{{\blb L}}

\let\cal\mathcal

\def\Cc{{\cal C}}
\def\Dc{{\cal D}}
\def\Ec{{\cal E}}
\def\Fc{{\cal F}}
\def\Gc{{\cal G}}
\def\Hc{{\cal H}}

\def\Lc{{\cal L}}
\def\Mc{{\cal M}}
\def\Nc{{\cal N}}
\def\Oc{{\cal O}}
\def\Pc{{\cal P}}
\def\Qc{{\cal Q}}
\def\Rc{{\cal R}}
\def\Sc{{\cal S}}
\def\Tc{{\cal T}}
\def\Uc{{\cal U}}
\def\Vc{{\cal V}}
\def\Wc{{\cal W}}
\def\Xc{{\cal X}}

\def\Zc{{\cal Z}}

\def\Am{{\mathfrak A}}
\def\Bm{{\mathfrak B}}
\def\Cm{{\mathfrak C}}

\def\Fm{{\mathfrak F}}

\def\Mm{{\mathfrak M}}

\def\Vm{{\mathfrak V}}

\def\Xm{{\mathfrak X}}

\newtheorem{lemma}{Lemma}[section]
\newtheorem{proposition}[lemma]{Proposition}
\newtheorem{theorem}[lemma]{Theorem}
\newtheorem{corollary}[lemma]{Corollary}

\newtheorem{definition}[lemma]{Definition}
\newtheorem{conjecture}[lemma]{Conjecture}

\newtheorem{assumption}[lemma]{Assumption}

\theoremstyle{remark}

\newtheorem{remark}[lemma]{Remark}
\newtheorem{notation}[lemma]{Notation}

\newtheorem{innercustomthm}{\emph{\textbf{Condition}}}
\newenvironment{customthm}[1]
  {\renewcommand\theinnercustomthm{#1}\innercustomthm}
  {\endinnercustomthm}

\def\dbcoh{\Dc^b}
\def\wt{\widetilde}

\def\arw{\longrightarrow}
\def\Hom{\operatorname{Hom}}

\def\Ext{\operatorname{Ext}}
\def\rk{\operatorname{rk}}

\def\git{/\hspace{-3pt}/}

\DeclareMathOperator{\sHom}{\mathscr{H}\text{\kern -3pt {\calligra\large om}}\,}

\newcommand\quotient[2]{
        \mathchoice
            {% \displaystyle
                \text{\raise1ex\hbox{$#1$}\Big/\lower1ex\hbox{$#2$}}%
            }
            {% \textstyle
                #1\,/\,#2
            }
            {% \scriptstyle
                #1\,/\,#2
            }
            {% \scriptscriptstyle  
                #1\,/\,#2
            }
    }

\newcommand{\adj}[4]{#1\negmedspace: #2\rightleftarrows #3:\negmedspace #4}

\makeatletter
\def\namedlabel#1#2{\begingroup
    #2%
    \def\@currentlabel{#2}%
    \phantomsection\label{#1}\endgroup
}
\makeatother

\title{Calabi--Yau fibrations, simple $K$-equivalence and mutations}
\author{Marco Rampazzo}
\address{Alma Mater Studiorum Università di Bologna\\ Dipartimento di Matematica \\ Piazza di Porta San Donato 5\\ 40126 Bologna.}
\email[M.~ Rampazzo]{marco.rampazzo3@unibo.it}
\begin{document}

\maketitle

\begin{abstract}
    A homogeneous roof is a rational homogeneous variety of Picard rank 2 and index $r$ equipped with two different $\PP^{r-1}$-bundle structures. We consider bundles of homogeneous roofs over a smooth projective variety, formulating a relative version of the duality of Calabi--Yau pairs associated to roofs of projective bundles. We discuss how derived equivalence of such pairs can lift to Calabi--Yau fibrations, extending a result of Bridgeland and Maciocia to higher-dimensional cases. We formulate an approach to prove that the $DK$-conjecture holds for a class of simple $K$-equivalent maps arising from bundles of roofs. As an example, we propose a pair of eight-dimensional Calabi--Yau varieties fibered in dual Calabi--Yau threefolds, related by a GLSM phase transition, and we prove derived equivalence with the methods above.
\end{abstract}

\section{Introduction}
\noindent Dualities of Calabi--Yau varieties have been a popular research subject in the last decades. In fact, from superstring theory and gauged linear sigma models to the many conjectures on their derived and birational geometry, Calabi--Yau varieties lie in the intersection of several diverse fields. In light of Bondal--Orlov's reconstruction theorem \cite{bondalorlovreconstruction}, they occupy a special place among algebraic varieties, namely it is possible to construct pairs of non isomorphic (or non birational) Calabi--Yau varieties which are derived equivalent. A first example has been given in terms of the Pfaffian--Grassmannian pair \cite{borisovcaldararupg}. This example has a clear link with the idea of a phase transition in a non abelian gauged linear sigma model \cite{addingtondonovansegal}. Constructions of this kind, in contrast with their abelian counterpart, are quite rare, and proving derived equivalence for these pairs has very often relied on ad hoc arguments. In fact, while some constructions like the Pfaffian-Grassmannian above and the intersection of two translates of $G(2, 5)$ \cite{ottemrennemo, borisovcaldararuperry} can be now explained by the homological projective duality and categorical joins programs \cite{kuznetsovhpd, kuznetsovperrypj}, there exists a class of conjecturally derived equivalent pairs of Calabi--Yau varieties \cite[Conjecture 2.6]{kr2} for which a general argument is missing. All these pairs share a common description in terms of the notion of \emph{roof}, which has been introduced by Kanemitsu in the context of $K$-equivalence to define special Fano manifolds which admit two projective bundle structures with the same Grothendieck line bundle \cite{kanemitsu}. Examples of Calabi--Yau pairs associated to roofs can be found in the works \cite{mukaiduality, imou, kuznetsovimou, kr, kr2}. Furthermore, while for constructions related to homological projective duality a link with the physics of gauged linear sigma models has been provided \cite{rennemosegal}, for the case of roofs a simple GLSM interpretation in terms of non abelian phase transition is missing, even if the underlying equivalence of matrix factorization categories \cite{kr2} suggests its existence.

In this paper we introduce the notion of \emph{roof bundles}, which are locally trivial fibrations on a smooth projective base $B$, with fibers isomorphic to a homogeneous roof (i.e. a roof which is a rational homogeneous variety). We prove that each roof bundle $\Zc$ has itself two projective bundle structures $p_i:\Zc\arw\Zc_i$ with the same Grothendieck line bundle $\Lc$ (defined up to twists by pullbacks of line bundles on the base). $\Lc$ restricts on each fiber over $B$ to the Grothendieck line bundle $L$ of both the projective bundle structures of the fiber. The main motivation for this construction arises in light of the $DK$-conjecture \cite{bondalorlovdk, kawamatadk}. In fact, let us consider a simple $K$-equivalence, i.e. a birational morphism $\mu:\Xc_1\dashrightarrow\Xc_2$ between smooth projective varieties resolved by blowups $g_i:\Xc_0\dashrightarrow\Xc_i$ with smooth centers such that $g_1^*\omega_{\Xc_1}\simeq g_2^*\omega{\Xc_2}$: Li proved that $g_1$ and $g_2$ have the same exceptional locus $\Zc$ \cite{li}, and Kanemitsu proved that $\Zc$ is isomorphic to a family of roofs \cite{kanemitsu}. If we assume $\Zc$ to be a roof bundle, we construct fully faithful embeddings of $\dbcoh(\Xc_1)$ and $\dbcoh(\Xc_2)$ in $\dbcoh(\Xc_0)$ and we find a derived equivalence between $\Xc_1$ and $\Xc_2$ for some classes of these birational pairs. This provides evidence for the $DK$-conjecture in the form of the following theorem:

\begin{theorem}\label{thm_intro_3}(Corollary \ref{cor_knowncasesKeq})
    Let $\mu:\Xc_1\dashrightarrow\Xc_2$ be a simple $K$-equivalence such that its exceptional locus $\Zc$ is a roof bundle of type $G/P$ over a smooth projective base $B$, where $G/P$ is a roof of type $A^M_n$, $A_n\times A_n$, $A^G_4$, $C_2$ or $G_2$ according to the list \cite[Section 5.2.1]{kanemitsu}. %Denote by $L$ the Grothendieck line bundle of both the projective bundle structures of $G/P$. Suppose $\Oc_\Zc(-\Zc)$ is basepoint-free and that the restriction of global sections to every fiber over $B$ is surjective.
    Then $\Xc_1$ and $\Xc_2$ are derived equivalent.
\end{theorem}

Furthermore, we formulate a relative version of the Calabi--Yau duality arising from a roof: with the data of a hyperplane section of a roof bundle $\Zc$ over a base $B$, we define a pair of fibrations with Calabi--Yau general fibers which are pairwise connected by the aforementioned duality. To address the problem of derived equivalence we construct semiorthogonal decompositions for the hyperplane and develop an approach based on mutations of exceptional objects: in particular, we reduce the problem to finding a sequence of mutations in a suitable semiorthogonal decomposition of a hyperplane section of the associated roof (i.e. a general hyperplane section in the fiber of $\Zc\arw B$ over a sufficiently general point). In the following theorems let $G/P$ be a homogeneous roof, hence one of the roofs appearing in the list of Kanemitsu \cite[Section 5.2.1]{kanemitsu}, and denote by $L$ the Grothendieck line bundle of both the projective bundle structures of $G/P$ (which exists by Proposition \ref{prop_linebundle}).

\begin{theorem}\label{thm_intro_1}(Corollary \ref{cor_knowncases})
    Let $\Zc$ be a roof bundle of type $G/P$ on a smooth projective base $B$, where $G/P$ is a roof of type $A_k\times A_k$, $A^M_n$, $A^G_4$, $G_2$ or $C_2$ and let $S\in H^0(\Zc, \Lc)$ be general, where the line bundle $\Lc$ is defined according to Equation \ref{eq_Lc}. Suppose that $\Lc$ is basepoint-free and that the restriction map $H^0(\Zc, \Lc)\arw H^0(\pi^{-1}(b), L)$ is surjective for every $b\in B$. Then $X_1=Z(p_{1*}S)$ and $X_2=Z(p_{2*}S)$ are derived equivalent. Moreover, if $G/P$ is not of type $A_k\times A_k$, there are fibrations $f_i:X_i\arw B$ such that for the general $b\in B$ one has $f_i^{-1}(b)\simeq Y_i$, where $X_i = Z(p_{i*}S)$ and $S\in H^0(\Zc, \Lc)$ is a general section and $(Y_1, Y_2)$ is a Calabi--Yau pair associated to the roof $G/P$ in the sense of Definition \ref{def_CYpairs}.
\end{theorem}

As an example in Section \ref{sec_cy8folds} we propose a pair of eight-dimensional Calabi--Yau fibrations over $\PP^5$ such that for every point in $\PP^5$ the fibers are non birational Calabi--Yau threefolds, extending a construction by Bridgeland and Maciocia \cite{bridgelandmaciocia} of derived equivalent elliptic and $K3$ fibrations to higher-dimensional examples.\\
\\
Finally, we give a gauged linear sigma model describing the fibered Calabi--Yau eightfolds introduced in Section \ref{sec_cy8folds} as geometric phases. The model is strictly related to the construction given in \cite{kr} and, as the latter, it admits a very simple description of the phase transition. We summarize all results about this pair of Calabi--Yau fibrations in the following theorem:

\begin{theorem}\label{thm_intro_4}(Theorem \ref{thm_main2_body})
There exists a pair of derived equivalent Calabi--Yau eightfolds $X_1$, $X_2$ of Picard number two, and fibrations $f_1:X_1\arw \PP^5$ and $f_2:X_1\arw \PP^5$ such that for the general $b\in \PP^5$, $Y_1:= f_1^{-1}(b)$ and $Y_2:= f_1^{-1}(b)$ are a pair of non birational, derived equivalent Calabi--Yau threefolds associated to the roof $F(2,3,5)$ (in the sense of Definition \ref{def_CYpairs}). Moreover, $X_1$ and $X_2$ are isomorphic to the vacuum manifolds of two phases of a non abelian gauged linear sigma model.
\end{theorem}

\subsection{Organization of the paper}
In Section \ref{sec_construction} we recall some definitions about roofs and the associated Calabi--Yau pairs. Then we introduce roof bundles, fixing the notation which will be used in the reminder of this paper. Furthermore, in Section \ref{sec_cy8folds}, we discuss the main example of this construction: a pair of Calabi--Yau eightfolds fibered in Calabi--Yau threefolds over $\PP^5$. In Section \ref{sec_deriverdequivalenceCYfibrations} we review an approach for solving the problem of derived equivlence of a Calabi--Yau pair associated to a given roof, which led to the derived equivlences of \cite{kuznetsovimou, kr} (and, by means of a similar construction, to the ones of \cite{morimura, uedaflop}). Then we relativize the picture, providing a strategy for the problem of derived equivalence of Calabi--Yau fibrations, based on derived equivalence of the fibers. In Section \ref{sec_simplekequivalence} we establish a link between derived equivalence of a Calabi--Yau pair related to a given roof and derived equivalence of a class of simply $K$-equivalent pair of smooth projective varieties such that the exceptional locus of the associated blowups is a roof bundle whose fiber is isomorphic to the roof above. Finally, in Section \ref{sec_glsm}, we give a GLSM interpretation of the fibered duality of the Calabi--Yau eightfolds introduced in Section \ref{sec_cy8folds}.

\subsection*{Acknowledgments}
I would like to thank my PhD advisor Micha\l\ Kapustka for the constant support and encouragement throughout this work. I am also very grateful to Alexander Kuznetsov for reading early drafts of this paper and providing valuable feedback, which led to important corrections, and to Enrico Fatighenti and Riccardo Moschetti for clarifying discussions. This project has been supported by the PhD program of the University of Stavanger.

\section{Construction}\label{sec_construction}
\subsection{Notation and conventions}
    We shall work over the field of complex numbers. We call $\PP^{r-1}$ bundle the projectivization $\PP(E)$ of a rank $r$ vector bundle $E$ over a smooth projective base $B$, the corresponding projection $\PP(E)\arw B$ will be called projective bundle structure. More generally, a surjective morphism of smooth varieties with fibers isomorphic to $\PP^{r-1}$ will be denoted $\PP^{r-1}$-fibration.
    In the following, a Calabi--Yau variety is defined as an algebraic variety $X$ such that $\omega_X\simeq\Oc_X$ and $H^m(X, \Oc_X)=0$ for $0<m<\dim(X)$. We call Calabi--Yau fibration a surjective morphism $X\arw B$ such that the general fiber is a Calabi--Yau variety. Given a vector bundle $\Ec$ on a variety $X$ and a section $\sigma\in H^0(X, \Ec)$, by $Z(\sigma)\subset X$ we denote the zero locus of $\sigma$. Furthermore, given a vector space $V$ and $k\in\ZZ$, we call $V[-k]$ the complex of vector spaces which is identically zero in every \linebreak degree except for $k$, where it is equal to $V$. By $\dbcoh(X)$ we denote the bounded derived category of coherent sheaves on a variety $X$.
    
\subsection{A note on the Borel--Weil--Bott theorem}
    Throughout this work we will make extensive use of the Borel--Weil--Bott theorem to compute the cohomology of homogeneous, irreducible vector bundles over homogeneous varieties. See \cite{weyman} for a thorough introduction on the subject. The practical computations are combinatoric and they can be easily automatized, see for example the script \cite{pythonscript} which has been used in this paper.
    
\subsection{Homogeneous roofs}

    A Mukai pair \emph{\cite{mukaipairs}} is the data $(Z, E)$ of a Fano variety $Z$ and an ample vector bundle $E$ with $c_1(E) = c_1(Z)$. Among such objects, one distinguishes the class of simple Mukai pairs:
    
    \begin{definition}\label{def_simplemukaipair}\emph{\cite[Definition 0.1]{kanemitsu}}
        A \emph{simple Mukai pair} is a Mukai pair $(Z_1, E_1)$ such that $\PP(E_1)\simeq
        \PP(E_2)$ where $E_2$ is a vector bundle over a projective variety $Z_2$, satisfying $\rk(E_1)=\rk(E_2)$.
    \end{definition}

    \begin{definition}\label{def_roof}\emph{\cite[Definition 0.1]{kanemitsu}}
        A \emph{roof} of rank $r$, or roof of $\PP^{r-1}$-bundles, is a Fano variety $X$ which is isomorphic to the projectivization of a rank $r$ vector bundle $E$ over a Fano variety $Z$, where $(Z, E)$ is a simple Mukai pair.
    \end{definition}
    
    \begin{remark}
        Equivalently, a roof can be defined as a Fano variety of Picard rank 2 and index $r$ equipped with two different $\PP^{r-1}$-bundle structures \cite[page 2]{kanemitsu}.
    \end{remark}
    
    The following provides a useful characterization of roofs:
    
    \begin{proposition}\label{prop_linebundle}\cite[Proposition 1.5]{kanemitsu}
        Let $X$ be a smooth projective Fano variety of Picard number two. Assume that for $i\in\{1;2\}$ the extremal contraction $h_i: X \arw  Z_i$  is a smooth $\PP^{r - 1}$-fibration. Then the following conditions are equivalent:
        \begin{enumerate}
            \item[$\circ$] $X$ is a roof
            \item[$\circ$] The index of $X$ is $r$
            \item[$\circ$] There exists a line bundle $L$ on $X$ which restricts to $\Oc(1)$ on every fiber of $h_1$ and $h_2$.
        \end{enumerate}
    \end{proposition}

    \begin{comment}
        Given a roof $X$, the following picture emerges:
        \begin{equation}\label{eq_rooffirst}
            \begin{tikzcd}[row sep = huge, column sep = small]
            &&\PP(\Ec)\ar[equal]{r}\ar[swap]{dll}{h_1} & X\ar[equal]{r} & \PP(\Fc)\ar{drr}{h_2}&&\\
            Z_1  &&&&&& Z_2
            \end{tikzcd}
        \end{equation}
    \end{comment}
    
    All known examples of roofs, except for one, are rational homogeneous varieties. This motivates the following definition:
    
    \begin{definition}
        A \emph{homogeneous roof} is a roof which is isomorphic to a homogeneous variety $G/P$ of Picard number two, where $G$ is a semisimple Lie group and $P$ is a parabolic subgroup.
    \end{definition}
    
    By \cite[Proposition 2.6]{campanapeternellsurvey}, the only possible projective bundle structures are defined by the natural surjections to $G$-Grassmannians $G/P_1$ and $G/P_2$ associated to parabolic subgroups $P_1, P_2\subset P$. Thus, the data of a homogeneous roof defines the following diagram:
    
    \begin{equation}\label{eq_roof}
        \begin{tikzcd}[row sep = huge]
                & G/P\ar[swap]{dl}{h_1}\ar{dr}{h_2} & \\
                G/P_1 & & G/P_2
        \end{tikzcd}
    \end{equation}
    
    By \cite[Proposition 1.5]{kanemitsu}, there exists a (unique) $L$ which is the Grothendieck line bundle of both the projective bundle structures of a roof. In the case of homogeneous roofs, the ample line bundle $h_1^*\Oc(1)\otimes h_2^*\Oc(1)$ satisfies such requirements, hence we set $L:=h_1^*\Oc(1)\otimes h_2^*\Oc(1)$.\\
    \\
    A complete list of homogeneous roofs has been given in \cite[Section 5.2.1]{kanemitsu}. Let us summarize its content in Table \ref{tab_rooflist}. In the column ``type'' we refer to the nomenclature introduced by Kanemitsu, which will also be adopted in the reminder of this work. Given a simply connected and semisimple Lie group $G$, the expression $G/P^{n_1,\dots,n_k}$ denotes the variety associated to the Dynkin diagram of $G$ crossed in the nodes $n_1, \dots, n_k$, numbered from the left to the right and from the top to the bottom. This is a standard notation which is explained, for instance, in \cite[Section 2]{imouexceptional}.

    \begin{table}[H]
    \centering
    \small
    \begin{tabular}{ c|c|c|c|c|c}
         & $G$ & type& $G/P$ & $G/P_1$ & $G/P_2$\\ 
      \hline
         &  &   &   &   &   \\
         &  $SL(k+1)\times SL(k+1)$  & $A_k\times A_k$ & $\PP^k\times\PP^k$ & $\PP^k$ & $\PP^k$\\ 
         &  &   &   &   &   \\
         &    $SL(k+1)$  & $A_k^M$ & $F(1, k, k+1)$ & $\PP^{k}$ & $\PP^{k}$\\
         &  &   &   &   &   \\
         &    $SL(2k+1)$  & $A_{2k}^G$ & $F(k, k+1, 2k+1)$ & $G(k, 2k+1)$ & $G(k+1, 2k+1)$\\
         &  &   &   &   &   \\
         &    $Sp(3k-2)$ ($k$ even)   & $C_{3k/2-1}$& $IF(k-1, k, 3k-2)$ & $IG(k-1, 3k-2)$ & $IG(k, 3k-2)$\\
         &  &   &   &   &   \\
         &    $Spin(2k)$  & $D_k$ & $OG(k-1, 2k)$ & $OG(k, 2k)^+$ & $OG(k, 2k)^-$\\
         &  &   &   &   &   \\
         &    $F_4$ & $F_4$ &$F_4/P^{2,3}$ & $F_4/P^2$ & $F_4/P^3$\\
         &  &   &   &   &   \\
         &    $G_2$ & $G_2$ & $G_2/P^{1,2}$ & $G_2/P^1$ & $G_2/P^2$\\
         &  &   &   &   &   \\
    \end{tabular}
    \caption{Homogeneous roofs}\label{tab_rooflist}
    \end{table}
    
    One can associate to a roof a pair of Calabi--Yau varieties in the following way:
    \begin{definition}\emph{(cfr. \cite[Definition 2.5]{kr2})}\label{def_CYpairs}
        Let $X$ be a roof with projective bundle structures $h_i:X\arw Z_i$. We say $(Y_1,Y_2)$ is a \emph{Calabi--Yau pair} associated to the roof $X$ if $Y_i\simeq Z(h_{i*}\sigma)$ is a Calabi--Yau variety for $i\in\{1;2\}$, where $\sigma\in H^0(X, L)$ is a general section.
    \end{definition}
    
    \begin{remark}
        There is only one known example of roof which is not homogeneous, it is the projectivization of an Ottaviani bundle on a smooth quadric of dimension five. In fact, such variety has a second projective bundle structure on a smooth five-dimensional quadric \cite{kanemitsuottaviani}. The associated Calabi--Yau pair is a pair of $K3$ surfaces of degree 12 which has been studied in \cite{imouk3, kr2}
    \end{remark}
    
    The following lemma shows that when $X$ is a homogeneous roof the Calabi--Yau condition is always verified for pairs $(Y_1, Y_2)$ such as in Definition \ref{def_CYpairs}.
    \begin{lemma}\label{lem_CY}
        Let $G/P$ be a homogeneous roof with projective bundle structures $h_i:G/P\arw G/P_i$ and consider a general section $\sigma\in H^0(G/P, L)$. Call $ E_i:= h_{i*}L$. Then $Y_i=Z(h_{i*}\sigma)\subset G/P_i$ is either empty or a Calabi--Yau variety of codimension $r$.
    \end{lemma}
    \begin{proof}
        %Call $H=H^0(X, \Oc(1))$. Since $\Oc(1)$ is an ample line bundle, $\Ec$ is an ample vector bundle. $\Oc(1)$ is an homogeneous ample line bundle, hence it is globally generated. We have the following sequence:
        %\begin{equation}\label{eq_globallygen}
        %    0\arw K\arw H\otimes\Oc\arw\Oc(1)\arw 0    
        %\end{equation}
        %By applying the derived pushforward functor, which is left exact, to the sequence \ref{eq_globallygen} we conclude that $\Ec$ is globally generated if $R^1 h_* K = 0$. But this is true because $H^0(X,K)\simeq H^1(X, K)$ and since $K$ is homogeneous $H^0(X,K) = H^1(X, K) = 0$, hence $R^1 h_* K$ vanishes on every stalk. Thus $Y$ is of expected codimension by generality of $\sigma$, in fact, $h_*\sigma$ is general if $\sigma$ is general. 
        Note that $L$ is basepoint-free. Therefore, by generality of $\sigma$, it follows that $Z(\sigma)$ is smooth of expected codimension: this allows to conclude by \cite[Lemma 3.2]{debarrekuznetsovjacobians} that also $Y_i$ is smooth of expected codimension. If $\dim(G/P)-\rk(E_i)\leq 0$ there is nothing more to prove. Otherwise, let us proceed in the following way: since $(G/P_i,  E_i)$ is a Mukai pair, by adjunction formula $Y_i$ has vanishing first Chern class. By \cite[Example 7.1.5]{lazarsfeld1}, since $ E_i$ is ample, the restriction map $H^q(G/P_i, \Omega^p_{G/P_i})\arw H^q(Y_i, \Omega^p_{Y_i})$ is an isomorphism for $p+q<\dim(Y_i)$, in particular $H^q(G/P_i, \Oc_{G/P_i})\simeq H^q(Y_i, \Oc_{Y_i})$ for $q<\dim(Y_i)$. But since $G/P_i$ is rational homogeneous one has $H^\bullet(G/P_i, \Oc_{G/P_i})\simeq\CC[0]$ and this concludes the proof.
    \end{proof}
    
    \begin{remark}\label{rem_AnxAnemptyloci}
        Observe that in Lemma \ref{lem_CY} the case $Y_i=\emptyset$ can appear only with roofs of type $A_k\times A_k$. For these roofs, the projective bundle structures are given by projectivizations of vector bundles of rank $k+1$ on $\PP^k$, hence the zero loci of pushforwards of a general section $\sigma\in H^0(\PP^k\times\PP^k, L)$ are empty.
    \end{remark}

    \subsection{Homogeneous roof bundles}
    
    While the problem of describing and classifying fibrations of roofs over a smooth projective variety has been addressed in \cite{kanemitsu, romanoetal}, we focus on a special class of such objects, which we call homogeneous roof bundles: they provide a natural relativization of homogeneous roofs, keeping many of the properties of the latter objects in a relative setting. To this purpose, we recall the following definition (see, for instance, the notes \cite[Section 3]{mitchellnotes}):
    
    \begin{definition}
        Given a linear reductive group $G$, a parabolic subgroup $P$, a vector space $V$ and a representation\linebreak $\Gamma: G\arw\operatorname{End}(V)$, we call \emph{balanced product} of $G$ and $V$ with respect to $P$ the space $G\times^P V:=(G\times V)/P$, where the $P$-action on $G\times P$ is given by:
        \begin{equation}
            \begin{tikzcd}[row sep = tiny, column sep = large, /tikz/column 1/.append style={anchor=base east} ,/tikz/column 2/.append style={anchor=base west}]
                P\times G\times V \ar{r} & G\times V \\
                (p, g, v)\ar[maps to]{r} &\left(g p^{-1}, \Gamma_p(v)\right).
            \end{tikzcd}
        \end{equation} 
    \end{definition}
    
    \begin{definition}\label{def_roofbundles}
        Let $G$ be a semisimple Lie group and $P$ a parabolic subgroup such that $G/P$ is a homogeneous roof. Let $\Vc$ be a principal $G$-bundle over a smooth projective variety $B$. We define a \emph{homogeneous roof bundle} over $B$ the variety $\Vc\times^G G/P$.
    \end{definition}
    Let us call $\pi:\Zc\arw B$ the map induced by the structure map $\Vc\arw B$. Then, for every $b\in B$ we have $\pi^{-1}(b)\simeq G/P$. In the same way, once we fix $\Zc_i:=\Vc\times^G G/P_i$, we define maps $r_i:\Zc_i\arw B$ with all fibers isomorphic to $G/P_i$.\\
    \\
Recall that, given a principal $H$-bundle $\Wc$ over a variety $X$, there is the following exact functor from the category of $H$-modules to the category of vector bundles over $X$ (see \cite[Section 2.2]{nori}, or the survey \cite[Page 8]{balajinotes}):
    \begin{equation}
        \begin{tikzcd}[column sep = huge]
            H-\operatorname{Mod}\ar{r}{\Wc\times^H (-)} &  \operatorname{Vect}(X)
        \end{tikzcd}
    \end{equation}
    which sends the $H$-module $R$ to the vector bundle $\Wc\times^H R$. Observe now that $\Vc\times^G G$ is a principal $P$-bundle over $\Zc$, because $\Vc\arw\Vc/P$ is a principal $P$-bundle and $\Vc/P\simeq \Vc\times^G G/P\simeq\Zc$ (see, for instance, \cite[Proposition 3.5]{mitchellnotes}). This allows to construct an exact functor:
    \begin{equation}
        \begin{tikzcd}[column sep = huge]
            P-\operatorname{Mod}\ar{r}{\Vc\times^G G\times^P (-)} &  \operatorname{Vect}(\Zc).
        \end{tikzcd}
    \end{equation}
    Moreover, it is well-known that there exists an equivalence of categories
    \begin{equation}
        \begin{tikzcd}[column sep = huge]
            P-\operatorname{Mod}\ar{r}{\Vm_{G,P}} & \operatorname{Vect}^P(G/P)
        \end{tikzcd}
    \end{equation}
    which sends a $P$-module $H$ to the $P$-homogeneous vector bundle $\Vm_{G,P}(H) = G\times^P H$. In particular, $\Vm_{G,P}^{-1}$ is an exact functor. Summing all up we can construct an exact functor $F$ sending homogeneous vector bundles over $G/P$ to vector bundles over $\Zc$:
    \begin{equation}\label{eq_relativizationfunctor}
        \begin{tikzcd}[column sep = huge, row sep = large]
            \operatorname{Vect}^P(G/P) \ar{rr}{\Fm:=\Vc\times^G G\times^P (-) \circ \Vm_{G,P}^{-1}}\ar[swap]{dr}{\Vm_{G,P}^{-1}} & & \operatorname{Vect}(\Zc)\\
            & P-\operatorname{Mod}\ar[swap]{ur}{\Vc\times^G G\times^P (-)} &
        \end{tikzcd}
    \end{equation}
    For each homogeneous roof of the list \cite[Section 5.2.1]{kanemitsu}, the vector bundles $E_1$ and $E_2$ such that $\PP(E_1)\simeq \PP(E_2)\simeq G/P$ are homogeneous: for $i=1,2$, they have the form
        \begin{equation}
            E_i = G\times^{P_i} V_{E_i}
        \end{equation}
    for given representation spaces $V_{E_i}$. This allows to define the following vector bundle on $\Zc_i$:
    \begin{equation}\label{eq_Ec_i}
        \Ec_i:= \Fm(E_i)
    \end{equation}
    Similarly, consider $L$ on $G/P$ as introduced in Proposition \ref{prop_linebundle}: for any line bundle $T$ on $B$, the line bundle
        \begin{equation}\label{eq_Lc}
            \Lc_T:= \Fm(L)\otimes\pi^* T.
        \end{equation}
    restricts to $L$ on every fiber $\pi^{-1}(B)$.
    \begin{lemma}\label{lem_roofbundlecommutativediagram}
        Let $G/P$ be a homogeneous roof with projective bundle structures $h_i:G/P\simeq \PP(E_i)\arw G/P_i$ for $i\in\{1;2\}$. Let $\Vc\arw B$ be a principal $G$-bundle over a smooth projective variety $B$. Then the following holds:
        \begin{enumerate}
            \item $\Zc:=\Vc\times^G  G/P$ has projective bundle structures $p_i:\Zc\simeq\PP(\Ec_i)\arw\Zc_i$. Moreover, $\Lc_\Oc = \Fm(L)$ restricts to $\Oc(1)$ on the fibers of both $p_1$ and $p_2$, and $p_{i*}\Lc_\Oc\simeq\Ec_i$
            \item There are smooth extremal contractions $r_i:\Zc_i\arw B$
            \item The following diagram is commutative:
         \begin{equation}\label{eq_smalldiagramfibrations}
            \begin{tikzcd}[row sep = large]
                & \Zc\ar[swap]{dl}{p_1}\ar{dr}{p_2} &  \\
                \Zc_1\ar[swap]{dr}{r_1} & & \Zc_2\ar{dl}{r_2} \\
                &B& 
            \end{tikzcd}
        \end{equation}
        \end{enumerate}
    \end{lemma}
    \begin{proof}
        To prove (1) consider the following surjection, which is a consequence of the fact that $G/P = \PP(E_i)$:
        \begin{equation}
            h_i^*E_i\arw L\arw 0.
        \end{equation}
        Applying $\Fm$ we get:
        
        \begin{equation}
            p_i^*\Ec_i\arw\Lc_\Oc\arw 0
        \end{equation}
        
        which is compatible with pullback by \cite[Proposition 3.6]{mitchellnotes}. This determines an isomorphism $\Zc\simeq\PP(\Ec_i)$ by \cite[Ch. II.7, Proposition 7.12]{hartshorne}, and thus $p_{i*}\Lc_\Oc\simeq\Ec$ by \cite[Ch. II.7, Proposition 7.11]{hartshorne}.\\
        \\
        Let us now prove (2). First, observe that $r_i$ is locally projective (hence proper) and every fiber is isomorphic to a Fano variety of Picard number one, which means that $-K_{\Zc_i}$ is $r_i$-ample because $-K_{\Zc_i}|_{r_i^{-1}(b)}$ is ample for every $b\in B$ \cite[Theorem 1.7.8]{lazarsfeld1}. This proves that $r_i$ is a Fano--Mori contraction \cite[Definition I.2.2]{occhettathesis}. To prove that $r_i$ is a contraction of extremal ray (or elementary Fano--Mori contraction) we just need to show that $\operatorname{Pic}(\Zc_i)/\operatorname{Pic}(B)\simeq\ZZ$ \cite[Definition I.2.2]{occhettathesis}. By \cite[Proposition 2.3]{fossumiversen}, since $\Zc_i$ is a locally trivial $G/P_i$-fibration we have an exact sequence:
        \begin{equation}
            H^0(G/P_i, \Oc^*)/\CC^*\arw\operatorname{Pic}(B)\arw\operatorname{Pic}(\Zc_i)\arw\operatorname{Pic}(G/P_i)\arw 0
        \end{equation}
        and the first term vanishes because of the long cohomology sequence associated to the exponential sequence for $G/P_i$:
        \begin{equation}
            \begin{split}
                0\arw H^0(G/P_i,\ZZ)\arw H^0(G/P_i, \Oc)\arw H^0(G/P_i, \Oc^*)\arw H^1(G/P_i,\ZZ)\arw\cdots
            \end{split}
        \end{equation}
        where we observe that $H^1(G/P_i,\ZZ)$ vanishes since the integral cohomology of rational homogeneous varieties is generated by their Bruhat decompositon, and it is nonzero only in even degree. Our claim is proven once we recall that  by Definition \ref{def_roof} one has $\operatorname{Pic}(G/P_i)\simeq\ZZ$.\\
        \\
        In order to prove Claim (3), consider a point $z\in\Zc$. Then, since $\pi$ is locally trivial, there exists a unique fiber $\pi^{-1}(b)\simeq G/P$ containing $z$, and $p_i|_\pi^{-1}(b)$. Similarly, $p_i$ is locally trivial, hence $z$ is contained in a unique $p^{-1}(z_i)$ for $i\in\{1;2\}$. Finally, by local triviality of $r_i$, we see that there is a unique $b_i$ such that $z_i\in r_i^{-1}(b_i)$. But then $b_1 = b_2 = b$, because otherwise $z$ would be contained in two different fibers of $\pi$.
    \end{proof}
    \begin{remark}
        In the lemma above, $\Lc_\Oc$ can be replaced by $\Lc_T$ for any line bundle $T$ on $B$. In fact, by projection formula and the lemma above, one has $p_{i*}\Lc_T = r_i^*T \otimes p_{i*}\Lc_\Oc = r_i^*T \otimes \Ec_i$. However, twisting $\Lc_\Oc$ by a line bundle $\pi^*T$  affects properties like its first Chern class, being basepoint-free or having nonzero global sections. Therefore, in the remainder of this paper, whenever we will fix a roof bundle $\Zc$, we will also choose $\Lc:=\Lc_T$ (by fixing the right $T$) satisfying the required hypotheses.
    \end{remark}

\subsection{Calabi--Yau fibrations}

Our main interest in Sections \ref{sec_cy8folds} and \ref{sec_deriverdequivalenceCYfibrations} is to investigate the zero loci of pushforwards of a general section $S\in H^0(\Zc , \Lc)$, hence relativizing the setting of Definition \ref{def_CYpairs}. %Let us make this clearer by the following lemma.

% GENERALITY STUFF

\begin{lemma}\label{lem_CYfibered}
    Let $\Zc$ be a homogeneous roof bundle of type $G/P\not\simeq\PP^n\times\PP^n$ over a smooth projective variety $B$ and fix \linebreak $h_i:G/P\simeq\PP(E_i)\arw G/P_i$ for $i\in\{1;2\}$.\\
    Assume that $\Lc$ is basepoint-free and that the restriction map $\rho_b:H^0(\Zc, \Lc)\arw H^0(\pi^{-1}(b), L)$ is surjective. Given a general section $S\in H^0(\Zc, \Lc)$, let us call $X_i:= Z(p_{i*}S)$. Then there exist fibrations $f_i:X_i\arw B$ such that, for a general $b\in B$, $(f_1^{-1}(b), f_2^{-1}(b))$ is a Calabi--Yau pair associated to $G/P$ in the sense of Definition $\ref{def_CYpairs}$.
\end{lemma}
\begin{proof}
    Let us call $f_i:=r_i|_{X_i}$. By the condition $\Ec_i|_{r_i^{-1}(b)}\simeq E_i$ and $r_i^{-1}(b)\simeq G/P_i$ it follows that $(r_i^{-1}(b), \Ec_i|_{r_i^{-1}(b)})$ is a Mukai pair. By assumption, $\Lc$ is basepoint-free, hence $Z(S)$ is smooth of expected codimension if $S$ is chosen away from a closed subset $\Wc\subset H^0(\Zc, \Lc)$. Since $\rho_b$ is surjective for every $b$, if we choose $b$ sufficiently general one has that the preimage $\rho_b^{-1}(\sigma)$ of a general $\sigma\in H^0(G/P, L)$ intersects $\Wc$ in a positive-codimensional subset: hence, for general $b$ and $S$, both $Z(S)\subset\Zc$ and $Z(S|_{\pi^{-1}(b)})\subset\pi^{-1}(b)$ are smooth of expected codimension. It follows that  $f_i^{-1}(b)=Z(p_{i*}S|_{r_i^{-1}(b)})\subset r_i^{-1}(b)$ is Calabi--Yau by Lemma \ref{lem_CY}. Moreover, $E_i\simeq h_{i*}L$ and the varieties $f_1^{-1}(b)$ and $f_2^{-1}(b)$ are the zero loci of the pushforwards of the same section $S_{\pi^{-1}(b)}$, thus they are a Calabi--Yau pair associated to the roof of type $G/P$ as in Definition \ref{def_CYpairs}.
\end{proof}

\section{A pair of Calabi--Yau eightfolds}\label{sec_cy8folds}
{\subsection{Roof of type $A^G _4$}}
Let us briefly describe  the roof of type $A^G _4$ and its related dual Calabi--Yau threefolds. Fix a vector space $V_5$ of dimension five. We call $G(2,V_5)$ and $G(3,V_5)$ the Grassmannians of respectively  linear 2-spaces and linear 3-spaces in $V_5$. On these Grassmannians, there are the following tautological short exact sequences:
\begin{equation}\label{eq_universalsequenceG25}
   0\arw\Uc_{G(2, V_5)}\arw V_5\otimes\Oc\arw\Qc_{G(2, V_5)}\arw 0 
\end{equation}
\begin{equation}\label{eq_universalsequenceG35}
   0\arw\Uc_{G(3, V_5)}\arw V_5\otimes\Oc\arw\Qc_{G(3, V_5)}\arw 0 
\end{equation}
where $\det\Uc_{G(2, V_5)}^\vee\simeq\det\Qc_{G(2, V_5)}\simeq\Oc_{G(2, V_5)}(1)$ and $\det\Uc_{G(3, V_5)}^\vee\simeq\det\Qc_{G(3, V_5)}\simeq\Oc_{G(3, V_5)}(1)$.
The roof of type $A^G_4$ is the flag variety $F(2,3,V_5)$: in fact this flag is isomorphic to the projectivization of $\Qc_{G(2, V_5)}^\vee(2)$ and $\Uc_{G(3, V_5)}(2)$ and these isomorphisms define projections to the Grassmannians. We get the diagram:

\begin{equation}\label{roofA4}
    \begin{tikzcd}[row sep = large]
     & F(2,3,V_5)\ar{dl}{h_1}\ar[swap]{dr}{h_2} & \\
    G(2, V_5)   &    & G(3, V_5)
    \end{tikzcd}
\end{equation}

Observe that the line bundle $\Oc(1,1):=h_1^*\Oc(1)\otimes h_2^*\Oc(1)$ satisfies $h_{1*}\Oc(1,1)\simeq \Qc_{G(2,V_5)}^\vee(2)$ and $h_{2*}\Oc(1,1) \simeq \Uc_{G(3,V_5)}(2)$. General sections of such bundles are Calabi--Yau threefolds \cite{michalconifold, inoueitomiura, kr}. Moreover, for a general $S\in H^0(F(2,3,V_5), \Oc(1,1))$, the pushforwards $h_{1*}S$ and $h_{2*}S$ are non birational \cite[Theorem 5.7]{kr}.

{\subsection{The homogeneous roof bundle of type $A^G_{4}$ over $\PP^5$}\label{sec_CYgeometry}}
We will discuss here the geometry of the roof bundle of type $A^G_4$ over $\PP^5$, which is a locally trivial fibration with fibers isomorphic to the flag variety $SL(5)/P^{2,3}\simeq F(2,3,5)$. Let us fix a vector space $V_6\simeq \CC^6$ and the (twisted) tangent bundle $T(-1)$ of $\PP(V_6)\simeq\PP^5$. We can define a roof bundle of type $SL(5)/P^{2,3}$ over $\PP^5$ by considering the $SL(5)$-bundle $\Vc = \operatorname{Iso}(\CC^5, T(-1))$ and taking $\Vc \times^{SL(5)} (SL(5)/P^{2,3}) \simeq \Fc l(2,3, T(-1))$. In this setting, Diagram \ref{eq_smalldiagramfibrations} becomes:
\begin{equation}\label{eq_familyA4}
    \begin{tikzcd}[row sep=huge, column sep = tiny]
        & \Fc l(2,3, T(-1))\ar[swap]{dl}{p_1}\ar{dr}{p_2} &  \\
        \hspace{-30pt}\Gc r(2, T(-1))\ar[swap]{dr}{r_1} & & \Gc r(3, T(-1)) \hspace{-30pt}\ar{dl}{r_2} \\
        &\PP^5& 
    \end{tikzcd}
\end{equation}
where $\Gc r$ and $\Fc l$ denote respectively Grassmann bundles and flag bundles. Note that all varieties here are rational homogeneous: in fact one has $\Fc l(2,3, T(-1))\simeq F(1,3,4,V_6)$, $\Gc r(2, T(-1))\simeq F(1,3,V_6)$ and $\Gc r(3, T(-1))\simeq F(1,4,V_6)$. Moreover, there are also a $\PP^2$-bundle structure $\rho:F(1,3,V_6)\arw G(3, V_6)$ and a $\PP^3$-bundle structure $\tau:F(1,4,V_6)\arw G(4, V_6)$. All this can be summarized by the following diagram:
\begin{equation}\label{eq_familyA4_explicit}
    \begin{tikzcd}[row sep=huge, column sep = tiny]
        &&F(1,3,4,V_6)\ar{dl}{p_1}\ar[swap]{dr}{p_2}\ar{dd}{\pi} &&  \\
        &F(1,3,V_6)\ar[swap]{dr}{r_1}\ar{dl}{\rho} & & F(1,4,V_6)\ar{dl}{r_2}\ar[swap]{dr}{\tau}& \\
       G(3, V_6) &&G(1, V_6)&& G(4, V_6)
    \end{tikzcd}
\end{equation}
where $\pi$ is just the composition of either $r_1$ and $p_1$ or $r_2$ and $p_2$.\\
\\
In the following, given a weight $\omega$, we will call $\Ec_\omega$ the associated vector bundle. Given a dominant weight $\omega$, we will call $V_\omega$ the associated representation space. On $F(1,3,4,V_6)$ we use the following notations for line bundles: $\Oc(a,b,c):=\pi^*\Oc(a)\otimes p_1^*\rho^*\Oc(b)\otimes p_2^*\tau^*\Oc(c)$. Fix the standard basis $\{\omega_1, \dots, \omega_5\}$ of fundamental weights for $A_5$. Observe that $\Oc(1,1,1)= \Ec_{\omega_1+\omega_3+\omega_4}$ on $F(1,3,4,V_6)$ has pushforwards to the Picard rank 2 flag varieties given by $p_{1*}\Oc(1,1,1)= \rho^*\Qc_{G(3, V_6)}^\vee(1,2)$ and $p_{2*}\Oc(1,1,1)=\Pc(1,2)$ where $\Pc$ is the rank 3 vector bundle defined by the following short exact sequence on $F(1,4,V_6)$:
\begin{equation}
    0\arw \Oc(-1, 0)\arw\tau^*\Uc_{G(4, V_6)}\arw\Pc\arw 0.
\end{equation}
%Also, we observe for later convenience that $\Pc=\Ec_{\omega_3-\omega_4}$.
\subsection{A pair of Calabi--Yau eightfolds}\label{sec_CY}
\begin{proposition}\label{prop_picardnumberCY}
Let $S\in H:=H^0(F(1,3,4,V_6), \Oc(1,1,1))$ be a general section. Then $X_1 = Z(p_{1*}S)$ and $X_2 = Z(p_{2*}S)$ are Calabi--Yau eightfolds of Picard number $2$, and $H^1(X_i, T_{X_i})\simeq H/(\CC\oplus V_{\omega_1+\omega_5})\simeq\CC^{1014}$.  Moreover, there exist fibrations $f_i:X_i\arw B$ such that for the general $b\in B$ the pair $(f_1^{-1}(b), f_2^{-1}(b))$ is a Calabi--Yau pair associated to the roof of type $A^G_4$. 
\end{proposition}
\begin{proof}
One has $\det(\Qc_{G(3, V_6)}^\vee(1,2)) = \Oc(3,5)$ on $F(1,3,V_6)$ and $\det(\Pc(1,2))\simeq \Oc(4,5)$ on $F(1,4,V_6)$.

On the other hand, combining the relative tangent bundle with the relative Euler sequence of $\rho$ one finds \linebreak $\omega_{F(1,3,V_6)}\simeq\Oc(-3, -5)$, and similarly $\omega_{F(1,4,V_6)}\simeq\Oc(-4, -5)$. Hence sections of $\Ec_i=p_{i*}\Oc(1,1,1)$ define eight-dimensional varieties with vanishing first Chern class for $i\in\{1;2\}$. Since the Grothendieck line bundle of $\PP(\Ec_i)$ is ample, $\Ec_i$ is also ample and we can use again \cite[Example 7.1.5]{lazarsfeld1}: the restriction maps
\begin{equation}\label{eq_cohomologyrestriction}
    \begin{split}
         H^q(F(1,3,V_6), \Omega^p_{F(1,3,V_6)}) &\arw H^q(X_1, \Omega^p_{X_1})\\
         H^q(F(1,4,V_6), \Omega^p_{F(1,4,V_6)}) &\arw H^q(X_2, \Omega^p_{X_2})
    \end{split}
\end{equation}
are isomorphisms for $p+q<\dim(X_1)$, and since $F(1,3,V_6)$ and $F(1,4,V_6)$ are homogeneous varieties, their structure sheaves have cohomology of dimension one concentrated in degree zero. The Calabi--Yau condition follows from setting $p=0$ in \ref{eq_cohomologyrestriction}.

The cohomology of the tangent bundle $T_{X_1}$ can be computed in terms of the normal bundle sequence:
\begin{equation}
    0\arw T_{X_1}\arw T_{F(1,3,V_6)}|_{X_1}\arw \rho^*\Qc_{G(3, V_6)}^\vee(1,2)\arw 0
\end{equation}
The last term is the pullback of a homogeneous irreducible vector bundle on $G(3, V_6)$, hence its cohomology can be computed by the Borel--Weil--Bott theorem. On the other hand, $T_{F(1,3,V_6)}$ fits in the following exact sequence
\begin{equation}\label{eq_relativetangenteuler}
    0\arw\Oc\arw\rho^*\Uc(1, -1)\arw T_{F(1,3,V_6)}\arw\rho^*T_{G(3, V_6)}\arw 0
\end{equation}
which follows by the relative tangent bundle sequence of $F(1,3,V_6)\arw G(3,V_6)$ and the relative Euler sequence of the projective bundle structure $F(1,3,V_6)\simeq\PP(r^*\Uc(1, -1))$. In order to restrict it to $X_1$ we take the tensor product of all terms with the Koszul resolution of $X$. Each of the resulting terms in the resolutions can be written, using the Littlewood--Richardson formula if necessary, as a direct sum of homogeneous irreducible vector bundles on the flag, and therefore Borel--Weil--Bott applies.\\
\\
We get this result:
\begin{equation}
   H^m(X, T_X) \simeq \left\{\begin{array}{lc}
   V_{\omega_1+\omega_3+\omega_4} / (\CC\oplus V_{\omega_1+\omega_5}) & m=1\\
    & \\
   \CC^2 & m=7
   \end{array}\right.
\end{equation}
which allows us to prove our claim. In fact, since $Y$ is Calabi--Yau, by Serre duality we have:
\begin{equation}
    H^7(Y, T_Y)\simeq H^1(Y, \Omega^1_Y)= H^{(1,1)}(Y)
\end{equation}
and we conclude that the Picard number of $Y$ is two by the long exact sequence of cohomology of the exponential sequence. The case of $X_2$ is identical.%: in fact, the sequence of Equation \ref{eq_relativetangenteuler} involves only bundles on $F(1,3,V_6)$, and the weights of the bundles involved in the corresponding sequence on $F(1,4,V_6)$ are obtained by reversing the order of the fundamental weights on the crossed Dynkin diagram of the flag variety. Therefore, the result is identical by the symmetry of the Dynkin diagram of type $A_5$.\\
\\
Let us now prove that 
$$
\rho_b:H^0(F(1,3,4,6), \Oc(1,1,1))\arw H^0(F(2,3,5), \Oc(1,1))
$$
is surjective: a point in $F(2,3,V_5)$ can be described as a pair of totally decomposable forms $(v\wedge w, v\wedge w\wedge z)$ where $v,w,z\in V_5\simeq\CC^5$. 

\begin{comment}
Then, for every $u \in V_5 / \operatorname{Span}(v,w,z)$ there exists an embedding
\begin{equation}\label{eq_psi}
        \begin{tikzcd}[row sep = tiny, column sep = large, /tikz/column 1/.append style={anchor=base east} ,/tikz/column 2/.append style={anchor=base west}]
            F(2,3,V_5) \ar{r} &F(1,3,4,V_6)\\
            (v\wedge w, v\wedge w\wedge z)\ar[maps to]{r} &(u\wedge v\wedge w, u\wedge v\wedge w\wedge z).
        \end{tikzcd}
\end{equation} 
\end{comment}

Let us now consider a section $\sigma\in H^0(F(2,3,V_5), \Oc(1,1))$. It can be represented as:
\begin{equation}\label{eq_explicitsection}
    \sigma: (v\wedge w, v\wedge w\wedge z)\arw \sigma^{ijklm}v_{[i}w_{j]}v_{[k}w_l z_{m]}
\end{equation}
where indices in square bracket are anti-symmetrized and repeated indices are summed up for each of their possible values (i.e. $v^i w_i := v^1w_1+\dots+v^5w_5$ for $v\in V_5^\vee$ and $w\in V_5$). Similarly, one can describe a section $S\in H^0(F(1,3,4,V_6), \Oc(1,1,1))$ by:
\begin{equation}
    \sigma: (u\wedge v\wedge w, u\wedge v\wedge w\wedge z)\arw S^{ijklmnp}u_{[i}v_j w_{k]}u_{[l}v_m w_n z_{p]}.
\end{equation}
Consider now a point $[u] \in G(1,V_6)$ with $u\in V_6$ falling in its equivalence class. Then, $\rho_{[u]}$ acts on $S$ by contraction with $u$ (this is independent on the choice of the representative $u$ up to an overall nonzero factor) :
\begin{equation}
    \rho_{[u]} (S): (v\wedge w, v\wedge w\wedge z)\arw S^{ijklmnp}u_{[i}v_j w_{k]}u_{[l}v_m w_n z_{p]}
\end{equation}
and we see that this map is clearly surjective for every $[u]$. 
The proof is concluded by observing that $\Oc(1,1,1)$ is ample, and hence basepoint-free, allowing us to apply Lemma \ref{lem_CYfibered}.  
\end{proof}

%Observe that, since $\Oc(1,1,1)$ is basepoint-free an the restriction map $H^0(\Zc, \Oc(1,1,1))\arw H^0(F(2,3,5), \Oc(1,1))$ is surjective, by Lemma \ref{lem_CYfibered} $X_1$ and $X_2$ are Calabi--Yau fibrations.

\section{Derived equivalence of Calabi--Yau fibrations}\label{sec_deriverdequivalenceCYfibrations}

In the following we will make extensive use of mutations of semiorthogonal decompositions. A thorough introduction on the topic can be found, for example, in the works \cite{bondal, bondalorlovdk, kuznetsovhpd}. We will mostly focus on left mutations, since it is well known that given an admissible subcategory $\Am\subset\Cm$ of a triangulated category $\Cm$, the functors $\LL_\Am:^\perp\Am\arw\Am^\perp$ and $\RR_\Am:\Am^\perp\arw^\perp\Am$ are mutual inverses (see, for example, \cite[Lemma 2.7]{kuznetsovcubic4folds} and the source therein).

\subsection{Preparatory lemmas}\label{subsec_preplemmas}

The following lemma, which is a simple consequence of the definition of adjoint functors, will be necessary for the proof of Lemma \ref{lem_Lsemiorthogonalcommutes}:

\begin{lemma}\label{lem_categoricallemma}
    Let $\Bm, \Mm, \Xm$ be categories. Consider the adjoint pairs ${\adj{L_1}{\Xm}{\Mm}{R_1}}$ and $\adj{L_2}{\Bm}{\Xm}{R_2}$. Then, for every object $x\in\Xm$ one has the following identity:
    \begin{equation}
        L_1(\epsilon_{2,x}) = \epsilon_{12,\,\, L_1 x}\circ L_1L_2R_2(\eta_{1,x}).
    \end{equation}
    where by $\epsilon_{12, L_1 x}$ we denote the counit morphism of the adjoint pair ${\adj{L_1L_2}{\Bm}{\Mm}{R_2R_1}}$ on the object $L_1 x\in\Mm$.
\end{lemma}

\begin{proof}
By the following chains of morphisms
\begin{equation}
    \begin{tikzcd}[column sep = huge]
        x\ar[maps to]{r}{\eta_{1,x}} & R_1L_1 x \ar[maps to]{r}{L_2 R_2}\ar[bend left=20, maps to]{rr}{Id_{R_1 L_1 x}} & L_2 R_2 R_1 L_1 x \ar[maps to]{r}{\epsilon_{2, R_1 L_1 x}} & R_1 L_1 x \\
        x\ar[maps to]{r}{L_2 R_2}\ar[bend right=20, maps to]{rr}{Id_{x}} & L_2 R_2 x \ar[maps to]{r}{\epsilon_{2,x}} & x \ar[maps to]{r}{\eta_{1, x}} & R_1 L_1 x
    \end{tikzcd}
\end{equation}

we deduce that this diagram commutes:

\begin{equation}
    \begin{tikzcd}[row sep = huge, column sep = huge]
        L_1 L_2 R_2 x \ar{r}{L_1(\epsilon_{2,x})}\ar[swap]{d}{L_1 L_2 R_2(\eta_{1,x})} & L_1 x \ar{d}{L_1(\eta_{1,x})} \\
        L_1 L_2 R_2 R_1 L_1 x\ar{r}{L_1(\epsilon_{2, R_1 L_1 x})} & L_1 R_1 L_1 x.
    \end{tikzcd}
\end{equation}

This, in turn, implies:

\begin{equation}\label{eq_adjointstuff}
    \epsilon_{1, L_1 x}\circ L_1(\epsilon_{2, R_1 L_1 x})\circ L_1 L_2 R_2(\eta_{1,x}) = \epsilon_{1, L_1 x}\circ L_1(\eta_{1,x})\circ L_1(\epsilon_{2,x}).
\end{equation}

Note that $\epsilon_{1, L_1 x}\circ L_1(\eta_{1,x}) = Id_{L_1 x}$, and that $\epsilon_{1, L_1 x}\circ L_1(\epsilon_{2, R_1 L_1 x}) = \eta_{12, L_1 x}$. These last considerations, applied to Equation \ref{eq_adjointstuff}, complete the proof.

\end{proof}

\begin{definition}\label{def_commutingmutations}
    Let $\pi:\Zc\arw B$ be a flat proper morphism of smooth projective varieties and let $\Mc$ be a smooth divisor embedded in $\Zc$ by  $\iota:\Mc\xhookrightarrow{\hspace{10pt}}\Zc$. Consider two relatively exceptional objects $\Ec, \Fc\in\dbcoh(\Zc)$ (in the sense of \cite[Section 3]{kuznetsovlevico}) and suppose there exist strong semiorthogonal decompositions
    \begin{equation}
        \begin{split}
            \dbcoh(\Zc) &= \langle\Cc, \Ec\otimes\pi^*\dbcoh(B), \Fc\otimes\pi^*\dbcoh(B)\rangle\\
        \dbcoh(\Mc) &= \langle\Dc, \iota^*\Ec\otimes\iota^*\pi^*\dbcoh(B), \iota^*\Fc\otimes\iota^*\pi^*\dbcoh(B)\rangle\\
        \end{split}
    \end{equation}
    for some admissible subcategories $\Cc, \Dc$. We say that the left mutation $\LL_{\langle\Ec\otimes\pi^*\dbcoh(B)\rangle}(\Fc\otimes\pi^*\dbcoh(B))$ \emph{commutes} with $\iota^*$ if the following equivalence holds:
    \begin{equation}
        \begin{split}
            \dbcoh(\Mc) &= \langle\Dc, \LL_{\langle\iota^*\Ec\otimes\iota^*\pi^*\dbcoh(B)\rangle}(\iota^*\Fc\otimes\iota^*\pi^*\dbcoh(B)), \iota^*\Ec\otimes\iota^*\pi^*\dbcoh(B)\rangle\\
            & = \langle\Dc, \iota^*\LL_{\langle\Ec\otimes\pi^*\dbcoh(B)\rangle}(\Fc\otimes\pi^*\dbcoh(B)), \iota^*\Ec\otimes\iota^*\pi^*\dbcoh(B)\rangle\\
        \end{split}
    \end{equation}
\end{definition}

\begin{definition}\label{def_Lsemiorthogonality}
    Let $\Zc\arw B$ be a flat and proper morphism of smooth projective varieties, let $\Lc$ be a line bundle on $\Zc$. Consider two objects $\Ec, \Fc\in\dbcoh(\Zc)$. We say that $\Ec$ is $\Lc$-\emph{semiorthogonal} to $\Fc$ if the following condition is fulfilled:
    \begin{equation}
        \pi_*R\Hc om_{\Zc}(\Ec, \Fc\otimes\Lc^\vee) = 0
    \end{equation}
\end{definition}

\begin{lemma}
    In the setting of Definition \ref{def_Lsemiorthogonality}, let $B$ be a point. Then $\Ec$ is $\Lc$-semiorthogonal to $\Fc$ if and only if\linebreak $\Ext^\bullet_{\Zc}(\Ec, \Fc\otimes\Lc^\vee)=0$.
\end{lemma}

\begin{proof}
    Fix $B\simeq \{pt\}$. One has:
    \begin{equation}
        \begin{split}
            \pi_*R\Hc om_{G/P}(\Ec, \Fc\otimes \Lc^\vee) &\simeq H^\bullet(G/P, R\Hc om_{G/P}(\Ec, \Fc\otimes \Lc^\vee))\\
            &\simeq \Ext^\bullet_{G/P}(\Ec, \Fc\otimes \Lc^\vee)
        \end{split}
    \end{equation}
    where the first isomorphism is a consequence of \cite[Page 50, Corollary 2]{mumfordabelian}.
\end{proof}

\begin{lemma}\label{lem_Lsemiorthogonalcommutes}
    In the setting of Definition \ref{def_commutingmutations}, let $\Mc$ be the zer locus of a section of a line bundle $\Lc$. Then, if $\Ec$ is $\Lc$-semiorthogonal to $\Fc$, $\LL_{\langle \Ec\otimes\pi^*\dbcoh(B)\rangle}\Fc\otimes\pi^*\dbcoh(B)$ commutes with $\iota^*$.
\end{lemma}
\begin{proof}
    %Note that both the subcategories $\iota^*\LL_{\langle \Ec\otimes\pi^*\dbcoh(B)\rangle}(\Fc\otimes\pi^*\dbcoh(B))$ and $\LL_{\langle \iota^*\Ec\otimes\iota^*\pi^*\dbcoh(B)\rangle}\iota^*(\Fc\otimes\pi^*\dbcoh(B))$ are $B$-linear as a consequence of \cite[Lemma 2.36]{kuznetsovhyperplane}, so we just need to prove that $\iota^*\LL_{\langle \Ec\otimes\pi^*\dbcoh(B)\rangle}(\Fc\otimes\pi^*\Gc)\simeq\LL_{\langle \iota^*\Ec\otimes\iota^*\pi^*\dbcoh(B)\rangle}\iota^*(\Fc\otimes\pi^*\Gc)$ for every $\Gc\in\dbcoh(B)$. In fact, this would imply the existence of a $B$-linear isomorphism of categories, acting on objects in the following way:
    %\begin{equation}
    %    \begin{tikzcd}[row sep = tiny]
    %        \iota^*\LL_{\langle \Ec\otimes\pi^*\dbcoh(B)\rangle}(\Fc\otimes\pi^*\dbcoh(B))\ar{r}{\sim} %&\LL_{\langle \iota^*\Ec\otimes\iota^*\pi^*\dbcoh(B)\rangle}\iota^*(\Fc\otimes\pi^*\dbcoh(B))\\
    %        \iota^*\LL_{\langle \Ec\otimes\pi^*\dbcoh(B)\rangle}\Fc\otimes\iota^*\pi^*\Gc\ar[maps to]{r} & %\LL_{\langle \iota^*\Ec\otimes\iota^*\pi^*\dbcoh(B)\rangle}\iota^*\Fc\otimes\iota^*\pi^*\Gc.
    %    \end{tikzcd}
    %\end{equation}
    
    In order to describe the left mutation of $\iota^*\Fc\otimes\pi^*\Gc$ through $\langle\iota^*\Ec\otimes\iota^*\pi^*\dbcoh(B)\rangle$ inside $\dbcoh(\Mc)$, we consider the following functors (and their right adjoints):
    \begin{equation}\label{eq_psi}
        \begin{tikzcd}[row sep = tiny, column sep = large, /tikz/column 1/.append style={anchor=base east} ,/tikz/column 2/.append style={anchor=base west}]
            \Psi_{\langle\Ec\otimes\pi^*\dbcoh(B)\rangle}:\dbcoh(B)\ar{r} &\dbcoh(\Zc)\\
            \Gc\ar[maps to]{r} &\pi^*\Gc\otimes\Ec
        \end{tikzcd}
    \end{equation} 
    \begin{equation}\label{eq_psi!}
        \hspace{35pt}\begin{tikzcd}[row sep = tiny, column sep = large, /tikz/column 1/.append style={anchor=base east} ,/tikz/column 2/.append style={anchor=base west}]
            \Psi_{\langle\Ec\otimes\pi^*\dbcoh(B)\rangle}^!:\dbcoh(\Zc)\ar{r} & \dbcoh(B)\\
            \Rc\ar[maps to]{r} & \pi_*R\Hc om_{\Zc}(\Ec, \Rc)
        \end{tikzcd}
    \end{equation}
    
    Introduce the following notation:
    \begin{equation}
        \begin{split}
            \alpha &= \iota^*\Psi_{\langle\Ec\otimes\pi^*\dbcoh(B)\rangle}\Psi_{\langle\Ec\otimes\pi^*\dbcoh(B)\rangle}^!(\eta_{1, \Fc\otimes\pi^*\Gc})\\
            \epsilon_\Zc &= \epsilon_{2, \Fc\otimes\pi^*\Gc}\\
            \epsilon_\Mc &= \epsilon_{12, \iota^*(\Fc\otimes\pi^*\Gc)}.
        \end{split}
    \end{equation} 

    We can apply Lemma \ref{lem_categoricallemma} to the data $L_1 = \iota^*$, $R_1 = \iota_*$, $L_2 = \Psi_{\langle\Ec\otimes\pi^*\dbcoh(B)\rangle}$, $R_2 = \Psi_{\langle\Ec\otimes\pi^*\dbcoh(B)\rangle}^!$, $\Xm = \dbcoh(\Zc)$, $\Bm = \dbcoh(B)$, $\Mm = \dbcoh(\Mc)$. As a result, the following diagram commutes:
    
    \begin{equation}\label{eq_trianglediagramuniversalitygeneral}
        \begin{tikzcd}[row sep = huge, column sep = huge]
            \iota^*\Psi_{\langle\Ec\otimes\pi^*\dbcoh(B)\rangle}\Psi_{\langle\Ec\otimes\pi^*\dbcoh(B)\rangle}^!(\Fc\otimes\pi^*\Gc)\ar{d}{\alpha}\ar{r}{\iota^*\epsilon_\Zc} & \iota^*(\Fc\otimes\pi^*\Gc)\ar[equals]{d}\\
            \iota^*\Psi_{\langle\Ec\otimes\pi^*\dbcoh(B)\rangle}\Psi_{\langle\Ec\otimes\pi^*\dbcoh(B)\rangle}^!\iota_*\iota^*(\Fc\otimes\pi^*\Gc)\ar[swap]{r}{\epsilon_\Mc}  & \iota^*(\Fc\otimes\pi^*\Gc)
        \end{tikzcd}
    \end{equation}

    Let us now prove the following claim:\\
    \textbf{Claim.} \emph{$\alpha$ is an isomorphism if $\Ec$ is $\Lc$-semiorthogonal to $\Fc$}.
    To this purpose, let us focus on the following term:
    \begin{equation}
        \begin{split}
            \Psi_{\langle\Ec\otimes\pi^*\dbcoh(B)\rangle}^!\iota_*\iota^*(\Fc\otimes\pi^*\Gc)
            = \pi_*R\Hc om_\Zc(\Ec, \iota_*\iota^*(\Fc\otimes\pi^*\Gc))
        \end{split}
    \end{equation}
    One has
    \begin{equation}
        \begin{split}
            \pi_*R\Hc om_\Zc(\Ec, \iota_*\iota^*(\Fc\otimes\pi^*
            \Gc))\simeq \pi_*(\Ec^\vee\otimes\iota_*\iota^*(\Fc\otimes\pi^*
            \Gc)).
        \end{split}
    \end{equation}
    Moreover, $\iota_*\iota^*(\Fc\otimes\pi^*\Gc)$ has a resolution given by the tensor product of $\Fc$ with the Koszul resolution of $\iota_*\iota^*\Oc$:
    \begin{equation}
        0\arw\Fc\otimes\pi^*\Gc\otimes\Lc ^\vee\arw\Fc\otimes\pi^*\Gc\arw\iota_*\iota^*(\Fc\otimes\pi^*\Gc)\arw 0
    \end{equation}
    By left-exactness of the derived pushforward we get the following long exact sequence:
        \begin{equation}
            \begin{split}
                0\arw R^0\pi_*(\Ec^\vee\otimes\Fc\otimes\pi^*\Gc\otimes\Lc^\vee)\arw R^0\pi_*(\Ec^\vee\otimes\Fc\otimes\pi^*\Gc)\arw \\
                \arw R^0\pi_*(\Ec^\vee\otimes\iota_*\iota^*(\Fc\otimes\pi^*\Gc))\arw R^1\pi_*(\Ec^\vee\otimes\Fc\otimes\pi^*\Gc\otimes\Lc^\vee)\arw\\
                \arw R^1\pi_*(\Ec^\vee\otimes\Fc\otimes\pi^*\Gc)\arw R^1\pi_*(\Ec^\vee\otimes\iota_*\iota^*(\Fc\otimes\pi^*\Gc))\arw\cdots
            \end{split}
        \end{equation}
    Hence, proving the claim reduces to show that 
    \begin{equation}
        R^k\pi_*(\Ec^\vee\otimes\Fc\otimes\pi^*\Gc\otimes\Lc^\vee)=0
    \end{equation}
    for every $k$. By the (derived) projection formula one has:
    \begin{equation}
        \begin{split}
            R^k\pi_*(\Ec^\vee\otimes\Fc\otimes\pi^*\Gc\otimes\Lc^\vee) &\simeq R^k\pi_*(\Ec^\vee\otimes\Fc\otimes\Lc^\vee)\otimes\Gc\\
            &\simeq R^k\pi_*R\Hc om_\Zc(\Ec,\Fc\otimes\Lc^\vee)\otimes\Gc=0
        \end{split}
    \end{equation}
    where the last isomorphism follows by the fact that $\Ec$ is $\Lc$-semiorthogonal to $\Fc$. This proves that $\alpha$ is an isomorphism for every $\Gc\in \dbcoh(B)$.
    
    In order to prove that $\iota^*\LL_{\langle \Ec\otimes\pi^*\dbcoh(B)\rangle}(\Fc\otimes\pi^*\Gc)$ and $\LL_{\langle \iota^*\Ec\otimes\iota^*\pi^*\dbcoh(B)\rangle}\iota^*(\Fc\otimes\pi^*\Gc)$ are isomorphic for every $\Gc\in\dbcoh(B)$, note that these objects are defined as the cones of respectively $\epsilon_\Mc$ and  $\iota^*\epsilon_\Zc$, in the following distinguished triangles:
    \begin{equation}\label{eq_trianglesfibration}
            \begin{split}
                \Theta_{\langle\iota^*\Ec\otimes\iota^*\pi^*\dbcoh(B)\rangle}\Theta_{\langle\iota^*\Ec\otimes\iota^*\pi^*\dbcoh(B)\rangle}^!\iota^*\Fc\otimes\iota^*\pi^*\Gc\xrightarrow{\epsilon_\Mc}\\
                \arw\iota^*\Fc\otimes\iota^*\pi^*\Gc\arw\LL_{\langle\iota^*\Ec\otimes\iota^*\pi^*\dbcoh(B)\rangle}\iota^*\Fc\otimes\iota^*\pi^*\Gc;\\
                \\
                \iota^*\Psi_{\langle\Ec\otimes\pi^*\dbcoh(B)\rangle}\Psi_{\langle\Ec\otimes\pi^*\dbcoh(B)\rangle}^!\Fc\otimes\pi^*\Gc\xrightarrow{\iota^*\epsilon_\Zc}\\
                \arw\iota^*(\Fc\otimes\pi^*\Gc)\arw\iota^* \LL_{\langle\Ec\otimes\pi^*\dbcoh(B)\rangle}\Fc\otimes\pi^*\Gc.
            \end{split}
        \end{equation}
%One has the following maps:
%\begin{equation}
%    \begin{split}
%        \iota^*\Psi_{\langle\Ec\otimes\pi^*\dbcoh(B)\rangle}\Psi_{\langle\Ec\otimes\pi^*\dbcoh(B)\rangle}^!\Fc\otimes\pi^*\Gc\xrightarrow{\iota^*\operatorname{ev}_\Zc}
%        \iota^*(\Fc\otimes\pi^*\Gc)\\
%        \iota^*\Psi_{\langle\Ec\otimes\pi^*\dbcoh(B)\rangle}\Psi_{\langle\Ec\otimes\pi^*\dbcoh(B)\rangle}^!\iota_*\iota^*\Fc\otimes\pi^*\Gc\xrightarrow{\operatorname{ev}_\Mc}
%        \iota^*(\Fc\otimes\pi^*\Gc)
%    \end{split}
%\end{equation}

    where $\Theta_{\langle\iota^*\Ec\otimes\iota^*\pi^*\dbcoh(B)\rangle} = \iota^*\Psi_{\langle\Ec\otimes\pi^*\dbcoh(B)\rangle}$. Since $\alpha$ is an isomorphism, the proof is concluded by \cite[page 232, Corollary 4]{gelfandmanin} applied to Diagram \ref{eq_trianglediagramuniversalitygeneral}.

\end{proof}

%With the following lemma, in the setting of a roof bundle $\pi:\Zc\arw B$, we show that $\Lc$-semiorthogonality of vector bundles constructed by a representation of $P$ can be checked on the fibers of $\pi$.

\begin{lemma}\label{lem_Lsemiorthogonal}
    Let $\pi:\Zc\arw B$ be a roof bundle of type $G/P$ where $B$ is a smooth projective variety and $\Vc$ a principal $G$-bundle on $B$. Fix $\Lc:=\Lc_T$ for some $T\in B$, as in Equation \ref{eq_Lc}.
    %\textcolor{blue}{this is not precise and I think can even be not enough? are all fibers \textbf{canonically} isomorphic to G/P? I guess not when the fibration is nontrivial. But maybe $L$ is somehow unique with respect to a chosen component of Aut(G/P). Anyway this merits a comment somewhere.}
    Let $\Ec, \Fc$ be vector bundles on $\Zc$ such that for every $b\in B$ one has $\Ec|_{\pi^{-1}(b)}\simeq E$ and $\Fc|_{\pi^{-1}(b)}\simeq F$, where $E$ is $L$-semiorthogonal to $F$. Then, for every $\Gc\in\dbcoh(B)$, $\Ec$ is $\Lc$-semiorthogonal to $\Fc\otimes\pi^*\Gc$.
\end{lemma}
\begin{proof}
    Since $\pi_*R\Hc om_{\Zc}(\Ec, \Fc\otimes\pi^*\Gc\otimes\Lc^\vee)\simeq R\pi_*(\Ec^\vee\otimes \Fc\otimes\pi^*\Gc\otimes\Lc^\vee)$, the claim follows by proving that \linebreak  $R^k\pi_*(\Ec^\vee\otimes\Fc\otimes\pi^*\Gc\otimes\Lc^\vee)=0$ for every $k$. By the derived projection formula one has:
    \begin{equation}
       R^k\pi_*(\Ec^\vee\otimes\Fc\otimes\pi^*\Gc\otimes\Lc^\vee)= R^k\pi_*(\Ec^\vee\otimes\Fc\otimes\Lc^\vee)\otimes\Gc. 
    \end{equation}
    The following step is to prove that the stalk $R^k\pi_*(\Ec^\vee\otimes\Fc\otimes\Lc^\vee)_b$ vanishes for every $b\in B$. Once we fix $b$, we observe that:
    \begin{enumerate}
        \item[$\circ$] $\Ec^\vee\otimes\Fc\otimes\Lc^\vee$ is flat over $B$ \cite[Proposition III.9.2]{hartshorne}
        \item[$\circ$] The map $b\longmapsto \dim H^k(\pi^{-1}(b), \Ec^\vee\otimes\Fc\otimes\Lc^\vee|_{\pi^{-1}(b)}) =$\\ $\dim H^k(G/P, E^\vee\otimes F\otimes L^\vee)$ is constant for every $k$
    \end{enumerate}
    Then we apply \cite[Page 50, Corollary 2]{mumfordabelian} and we find:
    \begin{equation}
        \begin{split}
            R^k\pi_*(\Ec^\vee\otimes\Fc\otimes\Lc^\vee)_b &\simeq H^k(G/P, E^\vee\otimes F\otimes L^\vee)\\
            &\simeq \Ext^k_{G/P}(E, F\otimes L^\vee) = 0
        \end{split}
    \end{equation}
    where the last equality holds because $E$ is $L$-semiorthogonal to $F$ by assumption. This proves that \linebreak $R^k\pi_*(\Ec^\vee\otimes\Fc\otimes\Lc^\vee)=0$, hence concluding the proof.
\end{proof}

\begin{comment}

\begin{definition}\label{def_squarebrackets}
    Let $\Zc = \Vc\times^G G/P\xrightarrow{\,\,\pi\,\,}B$ be a roof bundle of type $G/P$ over a smooth projective base $B$, where $\Vc$ is a principal $G$-bundle. Consider a representation $\Gamma$ acting on a vector space $V_\Gamma$ and the associated homogeneous vector bundle $E=G\times^P V_{\Gamma}$. Then, we introduce the following notation for the associated vector bundle on $\Zc$:
    \begin{equation}
        [E]:= \Vc\times^G G\times^P V_\Gamma.
    \end{equation}
\end{definition}

\begin{lemma}\label{lem_squarebrackets}
    Let the notation be as in Definition \ref{def_squarebrackets} and consider a homogeneous vector bundle $F$ on $G/P$. Then, one has $[E]\otimes[F] = [E\otimes F]$.
\end{lemma}

\begin{proof}
    Let $\Gamma, \Delta$ be representations of $P$ acting on the vector spaces $V_\Gamma$ and $V_\Delta$, such that $E = G\times^P V_\Gamma$ and $F = G\times^P V_\Delta$. Then, one has $E\otimes F = G\times^P V_\Gamma\otimes V_\Delta$ and hence $[E\otimes F] = \Vc\times^G G\times^P V_\Gamma\otimes V_\Delta$. On the other hand, $[E]\otimes[F]$
\end{proof}

\end{comment}

\subsection{Roof bundles and mutations}

\begin{notation}\label{notation}
    Hereafter we establish the notation for the main results of Sections \ref{sec_deriverdequivalenceCYfibrations} and \ref{sec_simplekequivalence}.
    \begin{enumerate}
        \item[$\circ$] Consider a homogeneous roof bundle $\Zc\xrightarrow{\hspace{3pt}\pi\hspace{3pt}} B$ of type $G/P$ on a smooth projective base $B$, with projective bundle structures $p_i:\Zc\arw \Zc_i$ for $i\in\{1;2\}$. Call $\Vc$ a principal $G$-bundle over $B$ such that $\Zc = \Vc\times^G G/P$. We call $h_i:G/P\arw G/P_i$ the projective bundle structures of $G/P$.
        \item[$\circ$] Let $L$ be the Grothendieck line bundle of both the projective bundle structures of $G/P$. For some line bundle $T$ on $B$, fix $\Lc:=\Lc_T$ as in Equation \ref{eq_Lc}. 
        \item[$\circ$] Choose a section $S\in H^0(\Zc, \Lc)$ with zero locus $\Mc$, a general section $\sigma\in H^0(G/P, L)$, a point $b\in B$ such that $M := Z(\sigma) = Z(S|_{\pi^{-1}(b)})$ is smooth of expected codimension, and the following diagram:
        \begin{equation}
            \begin{tikzcd}[row sep = huge]
            M = \Mc\times_B\{b\}\ar{r}\ar{d} & \Mc \ar{d} \\
            \{b\}\ar{r} & B
            \end{tikzcd}
        \end{equation}
        where the generality conditions on $b$ and $S$ are the ones of Lemma \ref{lem_CYfibered}. Call $\iota:\Mc\xhookrightarrow{\hspace{10pt}}\Zc$ and \linebreak $l:M\xhookrightarrow{\hspace{10pt}}G/P$ the respective embeddings.
        
        \item[$\circ$] Consider a Calabi--Yau pair $(Y_1, Y_2)$ associated to the roof $G/P$. Similarly, we consider the pair $(X_1, X_2)$ of Calabi--Yau fibrations defined as $X_i = Z(p_{i*}S)$. 
        
        \item[$\circ$] Fix two full exceptional collections consisting of vector bundles:
        \begin{equation}\label{eq_collectionsGPi}
            \begin{split}
                \dbcoh(G/P_1) = \langle J_1, \dots, J_m\rangle, \hspace{10pt}
                \dbcoh(G/P_2) = \langle K_1, \dots, K_m\rangle
            \end{split}
        \end{equation}
        which by \cite[Corollary 2.7]{orlovblowup} induce:
        
        \begin{equation}\label{eq_smallcollectionsshortGP}
            \begin{split}
                \dbcoh(G/P) &= \langle E_1, \dots, E_N\rangle = \langle F_1, \dots, F_N\rangle
             \end{split}
        \end{equation}
        
        where the bundles $E_i$ have the form $E_i = h_1^*J_j\otimes L^{\otimes k}$ and $F_i = h_2^*K_j\otimes L^{\otimes k}$ for some integers $j,k$. Moreover, by Theorem (Orl03 Proposition 2.10), one gets two semiorthogonal decompositions for $\dbcoh(M)$:
        
        \begin{equation}\label{eq_smallcollectionsshort}
            \begin{split}
                \dbcoh(M) &= \langle \theta_1\dbcoh(Y_1), l^*E_1, \dots, l^*E_n\rangle = \langle \theta_2\dbcoh(Y_2), l^*F_1, \dots, l^*F_n \rangle \\
            \end{split}
        \end{equation}
        
        where $\theta_i$ is a fully faithful functor defined as the composition of the pullback to the projectivization of the normal bundle of $Y_i$ with its embedding in $M$. Note that the numbers of exceptional objects in Equations \ref{eq_smallcollectionsshortGP} and \ref{eq_smallcollectionsshort} are different, but all exceptional objects of \ref{eq_smallcollectionsshort} are pullbacks of objects from \ref{eq_smallcollectionsshortGP} (up to twisting the whole collection by $L^\vee$).
        We introduce the following ordered lists of exceptional objects: 
        
        \begin{equation}\label{eq_listsofbundles}
            \begin{split}
                \boldsymbol{E} &= \left( E_1, \dots, E_n\right) \\
                \boldsymbol{F} &= \left( F_1, \dots, F_n\right) \\
            \end{split}
        \end{equation}
        
        Observe that $\boldsymbol{E}$ and $\boldsymbol{F}$ are subsets of the generators of $\dbcoh(G/P)$ appearing in the exceptional collections \ref{eq_smallcollectionsshortGP}, and they comprise precisely the object whose pullbacks generate $\theta_i\dbcoh(Y_i)^\perp$.
        \item[$\circ$] Note that $G/P$ every exceptional object is homogeneous \cite[Proposition 2.1.4]{boehning}, hence every $J_i$ is a homogeneous vector bundle. By \cite[Theorem 3.1]{samokhin}, the subcategory $\Fm(J_i)\otimes r_1^*\dbcoh(B)$ is admissible in $\dbcoh(\Zc_1)$, where $\Fm$ is the same of Equation \ref{eq_relativizationfunctor}. Similarly we can define admissible subcategories $F(K_i)\otimes r_2^*\dbcoh(B)\subset\dbcoh(\Zc_2)$. By applying \cite[Theorem 3.1]{samokhin} and \cite[Corollary 2.7]{orlovblowup} to the collections \ref{eq_smallcollectionsshort} we find:
        
        \begin{equation}\label{eq_bigcollectionsshortGP}
            \begin{split}
                \dbcoh(\Zc) &= \langle [E_1], \dots, [E_N]\rangle = \langle [F_1], \dots, [F_N]\rangle \\
            \end{split}
        \end{equation}
        where we set $[E_i] := p_1^*\Fm(J_i)\otimes\Lc^{\otimes k}\otimes\pi^*\dbcoh(B)$, and the same kind of notation for $[F_i]$. 
        More generally, given a homogeneous vector bundle $G\times^{P} V_\Gamma$ on $G/P$, we will use the notation $[G\times^{P} V_\Gamma]$ to denote the subcategory $(\Vc\times^G G\times^{P} V_{\Gamma})\otimes\pi^*\dbcoh(B)\subset\dbcoh(\Zc)$.
        Since $L$ and $J_i$ are both homogeneous, one has $[E_i]\simeq [h_1^*J_i\otimes L^{\otimes k}]$ and $[F_i] = [h_2^*K_j\otimes L^{\otimes k}]$. Furthermore, by \cite[Proposition 2.10]{cayleytrick}, there are semiorthogonal decompositions:
        \begin{equation}\label{eq_bigcollectionsshort}
            \begin{split}
                \dbcoh(\Mc) &= \langle \phi_1\dbcoh(X_1), \iota^*[E_1], \dots, \iota^*[E_n]\rangle = \langle \phi_2\dbcoh(X_2), \iota^*[F_1], \dots, \iota^*[F_n]\rangle \\
            \end{split}
        \end{equation}
        where $\phi_1$ and $\phi_2$ are fully faithful functors given again by composing the pullback to the projectivization of the normal bundle of $X_i$ with its embedding in $\Mc$.

        \item[$\circ$] Consider a sequence of pairs $(\boldsymbol E^{(\lambda)}, \psi^{(\lambda)})$ where, for each $\lambda$, $\boldsymbol E^{(\lambda)} =  \left(E_1^{(\lambda)},\dots,E_n^{(\lambda)}\right)$ is an ordered list of exceptional objects of $\dbcoh(G/P)$ and $\psi^{(\lambda)}:\dbcoh(Y_1)\arw\dbcoh(M)$ is a fully faithful functor. Consider also a sequence of operations:
        \begin{equation}
            \xi^{(\lambda)}:(\boldsymbol E^{(\lambda)}, \psi^{(\lambda)})\longmapsto (\boldsymbol E^{(\lambda+1)}, \psi^{(\lambda+1)})
        \end{equation}
        such that each $\xi^{(\lambda)}$ falls in one of the following classes (type \ref{itm:op_o1}, \ref{itm:op_o2} or \ref{itm:op_o3}):
        \begin{description}%[style=multiline, labelwidth=1.5cm]
        \item[\namedlabel{itm:op_o1}{O1}] For $1\leq i\leq n-1$ exchanging the order of $E_{i+1}^{(\lambda)}$ with $E_{i}^{(\lambda)}$ and substituting the latter with $\LL_{E_i}E_{i+1}^{(\lambda)}$, while leaving $\psi^{(\lambda)}$ unchanged:
            \begin{equation}
                \begin{split}
                    \left(E_1^{(\lambda+1)},\dots, E_n^{(\lambda+1)}\right) &=\left(E_1^{(\lambda)},\dots,E_{i-1}^{(\lambda)}, \LL_{E_i^{(\lambda)}}E_{i+1}^{(\lambda)}, E_i^{(\lambda)}, E_{i+2}^{(\lambda)},\dots, E_n^{(\lambda)}\right)\\
                    \psi^{(\lambda+1)} &=\psi^{(\lambda)}
                \end{split}
            \end{equation}
            where $\LL_{E_i^{(\lambda)}}E_{i+1}^{(\lambda)}$ is defined by the following distinguished triangle in $\dbcoh(G/P)$:
            \begin{equation}
                E_i^{(\lambda)}\otimes\Ext_{G/P}^\bullet(E_i^{(\lambda)},E_{i+1}^{(\lambda)})\arw E_{i+1}^{(\lambda)}\arw\LL_{E_i^{(\lambda)}}E_{i+1}^{(\lambda)}.
            \end{equation}
            Similarly, we define the operation of exchanging the order of $E_{i}^{(\lambda)}$ with $E_{i+1}^{(\lambda)}$ and substituting the former with $\RR_{E_{i+1}^{(\lambda)}}E_{i}^{(\lambda)}$, while leaving $\psi^{(\lambda)}$ unchanged:
            \begin{equation}
                \begin{split}
                    \left(E_1^{(\lambda+1)},\dots, E_n^{(\lambda+1)}\right) &=\left(E_1^{(\lambda)},\dots,E_{i-1}^{(\lambda)}, E_{i+1}^{(\lambda)}, \RR_{E_{i+1}^{(\lambda)}}E_{i}^{(\lambda)}, E_{i+2}^{(\lambda)},\dots, E_n^{(\lambda)}\right)\\
                    \psi^{(\lambda+1)} &=\psi^{(\lambda)}
                \end{split}
            \end{equation}
             where $\RR_{E_{i+1}^{(\lambda)}}E_{i}^{(\lambda)}$ is defined by the following distinguished triangle in $\dbcoh(G/P)$:
            \begin{equation}
                \RR_{E_{i+1}^{(\lambda)}}E_{i}^{(\lambda)}\arw E_{i}^{(\lambda)}\arw E_{i+1}^{(\lambda)}\otimes\Ext_{G/P}^\bullet(E_i^{(\lambda)},E_{i+1}^{(\lambda)})^\vee.
            \end{equation}
            \item[\namedlabel{itm:op_o2}{O2}] Sending $E_1^{(\lambda)}$ to the end, twisting it by $L^{\otimes(r-1)}$ and substituting $\psi^{(\lambda)}$ with $\RR_{l^*E_1}\psi^{(\lambda)}$:
            \begin{equation}
                \begin{split}
                    \left(E_1^{(\lambda+1)},\dots, E_n^{(\lambda+1)}\right) &=\left(E_2^{(\lambda)},\dots,E_{n}^{(\lambda)}, E_1^{(\lambda)},\otimes L^{\otimes(r-1)}\right)\\
                    \psi^{(\lambda+1)} &=\RR_{l^*E_1}\psi^{(\lambda)}
                \end{split}
            \end{equation}
            or sending $E_n^{(\lambda)}$ to the beginning, tensoring it with $L^{\otimes(-r+1)}$ and substituting $\psi^{(\lambda)}$ with\\
            $\LL_{l^*(E_n\otimes L^{\otimes(-r+1)})}\psi^{(\lambda)}$:
            \begin{equation}
                \begin{split}
                    \left(E_1^{(\lambda+1)},\dots, E_n^{(\lambda+1)}\right) &\arw\left(E_1^{(\lambda)}\otimes L^{\otimes(-r+1)}, E_2^{(\lambda)},\dots,E_{n-1}^{(\lambda)}\right)\\
                    \psi^{(\lambda+1)} &\arw\LL_{l^*(E_n\otimes L^{\otimes(-r+1)})}\psi^{(\lambda)}
                \end{split}
            \end{equation}
            \item[\namedlabel{itm:op_o3}{O3}] For any $k\in\ZZ$, replacing $E_i^{(\lambda)}$ with $E_i^{(\lambda)}\otimes L^{\otimes k}$ for every $i$, and substituting $\psi^{(\lambda)}$ with $\psi^{(\lambda)}(-)\otimes L^{\otimes k}$:
            \begin{equation}
                \begin{split}
                    \left(E_1^{(\lambda+1)},\dots, E_n^{(\lambda+1)}\right) &=\left(E_1^{(\lambda)}\otimes L^{\otimes k},\dots, E_n^{(\lambda)}\otimes L^{\otimes k}\right)\\
                    \psi^{(\lambda+1)} &=\psi^{(\lambda)}(-)\otimes L^{\otimes k}
                \end{split}
            \end{equation}
        \end{description}
        
        \item[$\circ$] Consider a sequence of mutations and twists on the semiorthogonal decompositions \ref{eq_smallcollectionsshort} which acts in the following way:
        \begin{equation}\label{eq_mutationsinassumption}
            \begin{split}
                \langle \theta_1\dbcoh(Y_1), l^*E_1, \dots, l^*E_n\rangle \arw \langle \psi\dbcoh(Y_1), l^*F_1, \dots, l^*F_n\rangle 
            \end{split}
        \end{equation}
    \end{enumerate}
\end{notation}

\begin{comment}
\begin{lemma}
    In the language of Notation \ref{notation}, one has the following semiorthogonal decompositions: 
        \begin{equation}\label{eq_bigcollectionsshort}
            \begin{split}
                \dbcoh(\Mc) &= \langle \phi_1\dbcoh(X_1), \iota^*[E_1], \dots, \iota^*[E_n]\rangle \\
                            &= \langle \phi_2\dbcoh(X_2), \iota^*[F_1], \dots, \iota^*[F_n]\rangle \\
            \end{split}
        \end{equation}
    where $\phi_1$ and $\phi_2$ are the same as in Equation \ref{eq_collectionhyperplanefibration}.
\end{lemma}
\begin{proof}
    Observe that $J_i$ is an exceptional object on $G/P$, hence it is homogeneous. Moreover, $L$ is homogeneous as well, since it is a line bundle on $G/P$. Recall that, by definition of $\Lc$, one has $\Lc = \pi^* T\otimes(\Vc\times^G G\times^P V_\Gamma)$ where $T$ is a line bundle on $B$, and the $V_\Gamma$ is a representation space such that $L=G\times^P V_\Gamma$. Therefore, one has $[E_i]\simeq [h_1^*J_i\otimes L^{\otimes k}]$. Similarly, $[F_i] = [h_2^*K_j\otimes L^{\otimes k}]$. Furthermore, by (Orl03 Proposition 2.10):
\end{proof}
\end{comment}

\begin{lemma}\label{lem_inducedmutationsinZ}
    Let $\dbcoh(G/P) = \langle W_1, \dots, W_{n},\rangle$ be a full exceptional collection. Then, in the setting of Notation \ref{notation}, for $1\leq i\leq N-1$ one has:
    \begin{equation}\label{eq_mutatedblockZ}
        \small
        \begin{split}
            \dbcoh(\Zc) &= \langle [W_1],\dots, [W_{i-1}], [\LL_{W_i}W_{i+1}], [W_i], [W_{i+2}],\dots, [W_N]\rangle
        \end{split}
    \end{equation}
    Moreover, the following holds: 
    \begin{equation}\label{eq_homogeneousvb}
        \begin{split}
            \LL_{[W_i]}[W_{i+1}] \simeq [\LL_{W_i}W_{i+1}]
        \end{split}
    \end{equation}
\end{lemma}

\begin{proof}
    By \cite[Theorem 3.1]{samokhin} one has a semiorthogonal decomposition $\dbcoh(\Zc)$ $=\langle [W_1],\dots, [W_N]\rangle$, then applying \cite[Corollary 2.9]{kuznetsovcubic4folds} one finds:
    \begin{equation}%\label{eq_mutatedblockZ}
        \small
        \begin{split}
            \dbcoh(\Zc) &= \langle [W_1],\dots, [W_{i-1}], \LL_{[W_i]}[W_{i+1}], [W_i], [W_{i+2}],\dots, [W_N]\rangle
        \end{split}
    \end{equation}
    On the other hand, for $1\leq i\leq N-1$ one has:
    \begin{equation}\label{eq_generalcollectionGPmutated}
        \dbcoh(G/P) = \langle W_1, \dots, W_{i-1}, \LL_{W_i} W_{i+1}, W_i, W_{i+2},\dots, W_m\rangle
    \end{equation}
    and by applying again \cite[Theorem 3.1]{samokhin} to \ref{eq_generalcollectionGPmutated} one finds:
    \begin{equation}
        \small
        \begin{split}
            \dbcoh(\Zc) &= \langle [W_1],\dots, [W_{i-1}], [\LL_{W_i}W_{i+1}], [W_i], [W_{i+2}],\dots, [W_N]\rangle
        \end{split}
    \end{equation}
    By comparison with Equation \ref{eq_mutatedblockZ} we see that both $[\LL_{W_i}W_{i+1}]$ and $\LL_{[W_i]}[W_{i+1}]$ are equivalent to the subcategory $^\perp\langle [W_1],\dots, [W_{i-1}]\rangle\hspace{5pt}\cap\hspace{5pt}\langle[W_i], [W_{i+2}],\dots, [W_N]\rangle^\perp$, hence they are equivalent.
\end{proof}
    
\begin{lemma}\label{lem_inducedmutationsinM}
    In the language of Notation \ref{notation}, consider a semiorthogonal decomposition $\dbcoh(M) = \langle\phi_1\dbcoh(Y_1), l^*W_1,\dots, l^*W_n\rangle$ where, for $1\leq j\leq n$, $W_j$ is a homogeneous vector bundle on $G/P$. Assume that $W_i$ is $L$-semiorthogonal to $W_{i+1}$ for some positive $i < n$. Then the following holds:\begin{enumerate}
        \item[$\circ$] $\LL_{W_i}W_{i+1}$ commutes with $l^*$
        \item[$\circ$] $\LL_{[W_i]}[W_{i+1}]$ commutes with $\iota^*$
        \item[$\circ$] There is a semiorthogonal decomposition:
        \begin{equation}\label{eq_mutatedblockMcomm}
            \begin{split}
                \dbcoh(\Mc) =& \langle \phi_1\dbcoh(X_1), \iota^*[W_1], \dots, \iota^*[W_{i-1}],\hspace{5pt} \iota^*[\LL_{W_i}W_{i+1}], \iota^*[W_i], \iota^*[W_{i+2}],\dots,\iota^*[W_n]\rangle
            \end{split}
        \end{equation}
    \end{enumerate}
\end{lemma}

\begin{proof}
    Let us first recall that by \cite[Theorem 3.1]{samokhin} and \cite[Proposition 2.10]{cayleytrick} one has
    
    \begin{equation}
        \dbcoh(\Mc) = \langle \phi_1\dbcoh(Y_1), \iota^*[W_1], \dots, \iota^*[W_{n}]\rangle.
    \end{equation}
    
    We now prove commutativity. By Lemma \ref{lem_inducedmutationsinZ}, one has $[\LL_{W_i}W_{i+1}] = \LL_{[W_i]}[W_{i+1}] $. By Lemma \ref{lem_Lsemiorthogonal}, since $W_i$ is $L$-semiorthogonal to $W_{i+1}$ it follows that $\Wc_i$ is $\Lc$-semiorthogonal to $\Wc_{i+1}$. Then, by Lemma \ref{lem_Lsemiorthogonalcommutes} applied to \linebreak $G/P\arw\{pt\}$ one finds that $\LL_{W_i}W_{i+1}$ commutes with $l^*$, while by applying the same lemma to $\Zc\arw B$ it follows that $\LL_{[W_i]}[W_{i+1}]$ commutes with $\iota^*$. By the latter we get:
    \begin{equation}\label{eq_commutativityinlemma}
        \begin{split}
            \LL_{\iota^*[W_i]}\iota^*[W_{i+1}]&\simeq \iota^*\LL_{[W_i]}[W_{i+1}] = \iota^*[\LL_{W_i}W_{i+1}]
        \end{split}
    \end{equation}
    By \cite[Corollary 2.9]{kuznetsovcubic4folds} one has:
    \begin{equation}\label{eq_mutatedblockM}
            \begin{split}
                \dbcoh(\Mc) =& \langle \phi_1\dbcoh(X_1), \iota^*[W_1], \dots, \iota^*[W_{i-1}], \hspace{5pt}\LL_{\iota^*[W_i]}\iota^*[W_{i+1}], \iota^*[W_i], \iota^*[W_{i+2}],\dots,\iota^*[W_n]\rangle
            \end{split}
        \end{equation}
    and substituting \ref{eq_commutativityinlemma} in \ref{eq_mutatedblockMcomm} completes the proof.

    %$\LL_{\langle \iota^*\Ec_i\otimes\iota^*\pi^*\dbcoh(B)\rangle}\iota^*\Ec_{i+1} &= \iota^* (\Vc\times^G G\times^{P} V_{\Gamma})$.
\end{proof}

\begin{proposition}\label{prop_serrefunctor}
        In the language of Notation \ref{notation}, assume there is a semiorthogonal decomposition
        \begin{equation}
            \dbcoh(M) = \langle\phi_1\dbcoh(Y_1), l^*W_1,\dots, l^*W_n\rangle
        \end{equation}
        where every $W_j$ is a homogeneous vector bundle on $G/P$. Then one has:
    \begin{equation}
    %\small
        \begin{split}
            \dbcoh(\Mc) &= \langle \LL_{\iota^*[W_n\otimes L^{\otimes(-r+1)}]}\phi_1\dbcoh(X_1), \iota^*[W_n\otimes L^{\otimes(-r+1)}], \iota^*[W_1], \dots, \iota^*[W_{n-1}] \rangle.
        \end{split}
    \end{equation}
\end{proposition}

\begin{proof}
    Since $\Mc$ is a smooth projective variety, by the Serre functor, there is the following semiorthogonal decomposition:
    \begin{equation}\label{eq_decompositionwithserre}
        %\small
        \begin{split}
            \dbcoh(\Mc) &= \langle \LL_{\iota^*[W_n]\otimes 
            \omega_\Mc}\phi_1\dbcoh(X_1), \iota^*[W_n]\otimes\omega_\Mc, \iota^*[W_1], \dots, \iota^*[W_{n-1}] \rangle.
        \end{split}
    \end{equation}
    One has \cite[Example 3.2.11]{fulton}:
    \begin{equation}\label{eq_omegaZbymukaipair}
        \omega_\Zc\simeq p_i^*\omega_{\Zc_i}\otimes p_i^*\det\Ec_i\otimes\Lc^{\otimes(-r)}
    \end{equation}
    but since $\Zc$ is a roof bundle, $(\Zc|_{r_i^{-1}(b)}, \Ec_i|_{r_i^{-1}(b)})$ is a Mukai pair for every $b\in B$, which implies that \linebreak $\omega_{\Zc_i}\otimes \det\Ec_i\simeq r_i^*T$, where $T$ is a line bundle on $B$. Then, by plugging this into \ref{eq_omegaZbymukaipair} we get $\omega_\Zc\simeq\Lc^{\otimes(-r)}\otimes \pi^*T$. Due to the following normal bundle sequence:
    \begin{equation}
        0\arw T_\Mc\arw \iota^*T_\Zc\arw\Lc\arw 0
    \end{equation}
    one has $\omega_\Mc\simeq \iota^*\omega_\Zc\otimes\iota^*\Lc^\vee\simeq \iota^*\Lc^{\otimes(-r+1)}\otimes\iota^*\pi^* T$.
    Then, the proof is completed by the following computation:
    \begin{equation}
        \iota^*[W_n]\otimes\omega_\Mc = \iota^*[W_n]\otimes\iota^*\Lc^{\otimes(-r+1)} = \iota^*[W_n\otimes L^{\otimes(-r+1)}].
    \end{equation}
    and by using it in the decomposition \ref{eq_decompositionwithserre}.
\end{proof}

Let us gather here the assumptions for the main theorem of this chapter.
\begin{assumption}\label{ass_roofbundle}
    The data of Notation \ref{notation} fulfill the following requirements:
    \begin{description}%[style=multiline, labelwidth=1.5cm]
        \item[\namedlabel{itm:ass_a1}{A1}] $\Lc$ is basepoint-free and the restriction map
        \begin{equation}
            H^0(\Zc, \Lc)\arw H^0(\pi^{-1}(b), \Lc|_{\pi^{-1}(b)})
        \end{equation}
        is surjective for every $b\in B$
        \item[\namedlabel{itm:ass_a2}{A2}] The sequence of mutations described in Notation \ref{notation} (Equation \ref{eq_mutationsinassumption}) acts by a composition of the following operations:
        \begin{enumerate}
            \item[$\circ$] Mutations of pairs of exceptional objects $\LL_{l^*E}l^*F$ where $E$, $F$ satisfy the semiorthogonality condition\linebreak $\Ext^\bullet_{G/P}(F, E)=0$ and $E$ is $L$-semiorthogonal to $F$ (in short, $\LL_{l^*E}l^*F$ satisfies Condition \ref{cond_mutationsroof}  as it will be defined in Appendix \ref{appendix})
            \item[$\circ$] Overall twists by a power of $L$
            \item[$\circ$] Applying the Serre functor of $\Sc_M$ sending the last exceptional object to the beginning of the semiorthogonal decompositions, or applying the inverse functor $\Sc_M^{-1}$
            \item[$\circ$] Applying $\LL_{l^*E}$ or $\RR_{l^*E}$ to the subcategory $\theta_i\dbcoh(Y_i)$, where $l^*E$ is an exceptional object in the right (respectively left) semiorthogonal complement of $\theta_i\dbcoh(Y_i)$.
        \end{enumerate}
        \item[\namedlabel{itm:ass_a3}{A3}] The sequence of operations $\xi = \xi^{(R)}\dots\xi^{(1)}$ acts on $(\boldsymbol E, \theta_1)$ in the following way:
        \begin{equation}
            \xi:(\boldsymbol E, \theta_1)\longmapsto (\boldsymbol F, \psi)
        \end{equation}
        where $(\boldsymbol E^{(1)}, \psi^{(1)}) = (\boldsymbol E, \theta_1)$ and $(\boldsymbol E^{(R+1)}, \psi^{(R+1)}) = (\boldsymbol F, \psi)$.
        %and for every $\lambda$ the following semiorthogonality condition is satisfied:
        %\begin{equation}
        %    \Ext^\bullet_{G/P}(E_i^{(\lambda)}, E_j^{(\lambda)})=0\hspace{10pt} \text{for } i>j.
        %\end{equation}
    \end{description}
    Note that, while \ref{itm:ass_a1} is a condition on $\Lc$, \ref{itm:ass_a2} and \ref{itm:ass_a3} depend only on $G/P$: if they are satisfied by one roof bundle of type $G/P$, then they are fulfilled by every roof bundle of the same type. We say that a pair $(\Zc, \Lc)$ of a roof bundle $\Zc$ of type $G/P$ together with the line bundle $\Lc$ satisfies Assumption \ref{ass_roofbundle} if $\Lc$ satisfies Assumption \ref{itm:ass_a1} and there exist two full exceptional collections $G/P = \langle E_1, \dots E_N\rangle = \langle F_1\dots F_N\rangle$ and a section $\sigma\in H^0(G/P, L)$ as required in  Notation \ref{notation}, which are compatible with Assumptions \ref{itm:ass_a2} and \ref{itm:ass_a3}.
\end{assumption}

The condition \ref{itm:ass_a1} is needed to ensure smoothness of the general sections of $\Lc$, and the fact that zero loci of pushforwards of these sections have the property of being Calabi--Yau fibrations. On the other hand, assumption \ref{itm:ass_a2} is needed to construct the mutations in $\dbcoh(\Mc)$. The last assumption \ref{itm:ass_a3} is needed to ensure that such mutations really yield an equivalence $\dbcoh(X_1)\simeq\dbcoh(X_2)$.

\begin{remark}\label{rem_collections}
    Before stating the main results, let us remind that this approach requires a full exceptional collection for the fibers of $r_1$ and $r_2$. The problem of finding full exceptional collections for homogeneous varieties is still open, but there are many cases where a solution has been found. Let $G/P$ be a roof with projective bundle structures $h_i:G/P\arw G/P_i$ for $i\in\{1;2\}$. Let us review the cases where a full exceptional collection is known for both $G/P_1$ and $G/P_2$. Note that in each of the known cases the objects are all vector bundles.
    \begin{enumerate}
        \item[$\circ$] Type $A_n\times A_n$, $A^M_n$ and $A^G_{2n}$: here $G/P_i$ is a $SL(V)$-Grassmannian for some vector space $V$. Full exceptional collections for these varieties have been constructed in \cite{kapranov}.
        \item[$\circ$] Type $C_{3n/2-1}$: in this case $G/P_i$ is a symplectic Grassmannian. The only case where a full exceptional collection is known for both $G/P_1$ and $G/P_2$ is the roof of type $C_2$. The collections have been established in \cite{beilinson, kapranov}.
        \item[$\circ$] Type $D_n$: the only two cases where both $G/P_i$ have known full exceptional collections are $D_4$ and $D_5$. In the former, by triality, $G/P_i$ is a six-dimensional quadric, for which a full exceptional collection has been found in \cite{kapranov}. In the latter, the varieties $G/P_i$ are spinor tenfolds, a full exceptional collection for them is given in \cite{kuznetsovhyperplane}.
        \item[$\circ$] Type $G_2$: there are known full exceptional collections for both $G/P_1$ and $G/P_2$ \cite{kapranov, kuznetsovhyperplane}.
        \item[$\circ$] Type $F_4$: To the best of the author's knowledge, no full exceptional collection is known for $F_4/P^2$ and $F_4/P^3$.
    \end{enumerate}
\end{remark}
    
\begin{theorem}\label{thm_derivedequivalencefibrations} 
    Let $(\Zc, \Lc)$ satisfy Assumption \ref{ass_roofbundle}. Then a general section of $\Lc$ induces a derived equivalence of Calabi--Yau fibrations $\Phi:\dbcoh(X_1)\arw\dbcoh(X_2)$.
\end{theorem}

\begin{proof}
    Consider two full exceptional collections $G/P = \langle E_1, \dots E_N\rangle = \langle F_1\dots F_N\rangle$ and a general section $\sigma\in H^0(G/P, L)$ with zero locus $M$, compatible with Assumptions \ref{itm:ass_a2} and \ref{itm:ass_a3}.
    Then, by Assumption \ref{itm:ass_a2} one has a sequence of mutations providing the following identification:
    \begin{equation}\label{eq_initialcollection1}
        \begin{split}
            \dbcoh(M) &= \langle  \theta_1 \dbcoh(Y_1), l^*E_1, \dots, l^*E_n\rangle
        \end{split}
    \end{equation}
    \begin{equation}\label{eq_aftermutations}
        \begin{split}
            \hspace{30pt} = \langle \psi\dbcoh(Y_1), l^*F_1, \dots, l^*F_n\rangle
        \end{split}
    \end{equation}
    which consists in applying a sequence of operations to the first semiorthogonal decomposition, which only include mutations of $L$-semiorthogonal exceptional pairs, the Serre functor of $M$, overall twists by a line bundle and mutations of $\theta_1\dbcoh(Y_1)$ through the admissible subcategory generated by an exceptional object.\\
    \\
    Consider the sequence of operations $\xi = \xi^{(R)}\cdots\xi^{(1)}$ defined in Notation \ref{notation}. By Assumption \ref{itm:ass_a3} one has:
    \begin{equation}
        \xi:(\boldsymbol E, \theta_1)\longmapsto (\boldsymbol F, \psi)
        \end{equation}
        These operations, by \ref{itm:ass_a2}, are in one-to-one correspondence with the mutations used to transform the decomposition \ref{eq_initialcollection1} into \ref{eq_aftermutations} (the case of mutations of pairs is treated in Lemma \ref{lem_inducedmutationsinM}). In fact, for $1\leq\lambda\leq R+1$ one has a semiorthogonal decomposition:
        \begin{equation}\label{eq_collectionlambda}
            \begin{split}
                \dbcoh(M) = \langle \psi^{(\lambda)}\dbcoh(Y_1), l^*E_1^{(\lambda)},\dots,l^*E_n^{(\lambda)}\rangle.
            \end{split}
        \end{equation}
        Every operation $\xi^{(\lambda)}$ commutes also with the mapping $W\longmapsto[W]$: for the cases of Operations \ref{itm:op_o2} and \ref{itm:op_o3} this follows by the fact that the tensor product of homogeneous vector bundles associated to representations $\Gamma, \Gamma'$ is the homogeneous vector bundle associated to the representation $\Gamma\otimes\Gamma'$, while the case of type \ref{itm:op_o1} follows from Lemma \ref{lem_inducedmutationsinZ}.
        
        Therefore, we can define $[\boldsymbol E] = \{[E_1],\dots,[E_n]\}$, $[\boldsymbol F] = \{[F_1],\dots,[F_n]\}$ and construct a sequence of operations $[\xi] =[\xi^{(R)}]\cdots[\xi^{(1)}]$ as follows:
        \begin{equation}
            [\xi^{(\lambda)}]:([\boldsymbol E^{(\lambda)}], \Phi^{(\lambda)})\longmapsto ([\boldsymbol E^{(\lambda+1)}], \Phi^{(\lambda+1)})
        \end{equation}
        where $([\boldsymbol E^{(1)}], \Phi^{(1)})=([\boldsymbol E], \phi_1$) and  $([\boldsymbol E^{(R+1)}], \Phi^{(R+1)}) =([\boldsymbol F], \Phi)$, where the functor $\Phi^{(\lambda)}$ is yet to be defined.\\
        \\
        The next step is to show that we can associate to the sequence of mutations on the collection \ref{eq_initialcollection1} a sequence of mutations on the decomposition $\dbcoh(\Mc) = \langle \phi_1\dbcoh(X_1), \iota^*[E_1], \dots, \iota^*[E_n]\rangle$ defined through the operations $[\xi^{(\lambda)}]$, thus obtaining for every $\lambda$:
        \begin{equation}\label{eq_blockwiseinitialcollection1}
            \begin{split}
                \dbcoh(\Mc) = \langle \Phi^{(\lambda)}\dbcoh(X_1), \iota^*[E_1^{(\lambda)}],\dots,\iota^*[E_n^{(\lambda)}]\rangle.
            \end{split}
        \end{equation}
        To prove our claim, let us consider each of the allowed kinds of mutations on \ref{eq_initialcollection1}, and describe the associated mutation on \ref{eq_blockwiseinitialcollection1}, alongside with the definition of the appropriate $\Phi^{(\lambda)}$.
        \begin{enumerate}
            \item[$\circ$] Every time a left mutation of pairs 
            \begin{equation}
                \begin{split}
                    \langle\dots, l^*E_i^{(\lambda)}, l^*E_{i+1}^{(\lambda)},\dots\rangle & =\langle\dots,\LL_{l^*E_i^{(\lambda)}} l^*E_{i+1}^{(\lambda)}, l^*E_i^{(\lambda)},\dots\rangle
                \end{split}
            \end{equation}
            is performed in \ref{eq_collectionlambda}, we apply the following operation in \ref{eq_blockwiseinitialcollection1}:
            \begin{equation}\label{eq_mutationlambdaintheorem1}
                \begin{split}
                    \langle\Phi^{(\lambda)}\dbcoh(X_1),\dots, \iota^*[E_i^{(\lambda)}], \iota^*[E_{i+1}^{(\lambda)}],\dots\rangle
                    & =\langle\Phi^{(\lambda)}\dbcoh(X_1),\dots,\iota^*[\LL_{E_i^{(\lambda)}}E_{i+1}^{(\lambda)}], [E_i^{(\lambda)}],\dots\rangle\\
                    & =:\langle\Phi^{(\lambda+1)}\dbcoh(X_1),\iota^*[E_{1}^{(\lambda+1)}],
                       \dots,\iota^*[E_n^{(\lambda+1)}]\rangle
                \end{split}
            \end{equation}
            
            where $\Phi^{(\lambda+1)} := \Phi^{(\lambda)}$. We obtain a semiorthogonal decomposition because of the following argument: by Assumption \ref{itm:ass_a2},  $E_i^{(\lambda)}$ is $L$-semiorthogonal to $E_{i+1}^{(\lambda)}$, thus by Lemma \ref{lem_inducedmutationsinM}, $\LL_{E_i^{(\lambda)}}E_{i+1}^{(\lambda)}$ commutes with $l^*$ and $\LL_{[E_i^{(\lambda)}]}[E_{i+1}^{(\lambda)}]$ commutes with $\iota^*$. Finally, since $\LL_{[E_i^{(\lambda)}]}[E_{i+1}^{(\lambda)}] = [\LL_{E_i^{(\lambda)}}E_{i+1}^{(\lambda)}]$ by Lemma \ref{lem_inducedmutationsinZ}, we see that the operation described in Equation \ref{eq_mutationlambdaintheorem1} is simply the left mutation of $\iota^*[E_{i+1}^{(\lambda)}]$ through $\iota^*[E_{i}^{(\lambda)}]$. An analogous argument works for right mutations.
    
            \item[$\circ$] Every time the Serre functor is applied to Equation \ref{eq_collectionlambda}:
            
            \begin{equation}
                \begin{split}
                    & \langle \psi^{(\lambda)}\dbcoh(Y_1), l^*E_1^{(\lambda)}, \dots, l^*E_n^{(\lambda)}, \rangle\\
                    & = \langle \LL_{l^*E_n^{(\lambda)}\otimes L^{-r+1}}\psi^{(\lambda)}\dbcoh(Y_1), l^*E_n^{(\lambda)}\otimes L^{-r+1}, l^*E_1^{(\lambda)}, \dots, l^*E_{n-1}^{(\lambda)}\rangle
                \end{split}
            \end{equation}
            
            we perform the following operation on Equation \ref{eq_blockwiseinitialcollection1}:
            
            \begin{equation}
                \begin{split}
                    & \langle \Phi^{(\lambda)}\dbcoh(X_1), \iota^*[E_1^{(\lambda)}], \dots, \iota^*[E_n^{(\lambda)}], \rangle =\\
                    & = \langle \LL_{[\iota^*E_n^{(\lambda)}\otimes L^{-r+1}]}\Phi^{(\lambda)}\dbcoh(X_1), [\iota^*E_n^{(\lambda)}\otimes L^{-r+1}], \iota^*[E_1^{(\lambda)}], \dots, \iota^*[E_{n-1}^{(\lambda)}]\rangle\\
                    &\hspace{12pt}=:\langle\Phi^{(\lambda+1)}\dbcoh(X_1),\iota^*[E_{1}^{(\lambda+1)}],
                       \dots,\iota^*[E_n^{(\lambda+1)}]\rangle
                \end{split}
            \end{equation}
            where $\Phi^{(\lambda+1)}:= \LL_{[\iota^*E_n^{(\lambda)}\otimes L^{-r+1}]}\Phi^{(\lambda)}$. In fact, by Proposition \ref{prop_serrefunctor} the resulting collection above is the one obtained by applying the Serre functor of $\Mc$ to $\iota^*[E_1^{(\lambda)}]$ and sending the subcategory equivalent to $\dbcoh(X_1)$ to the beginning of the collection. The same holds for the inverse Serre functor.
            
            \item[$\circ$] Whenever Equation \ref{eq_collectionlambda} is twisted by $L^{\otimes k}$ for some $k\in\ZZ$, perform the following operation on Equation \ref{eq_blockwiseinitialcollection1}:
            
            \begin{equation}
                \begin{split}
                    \langle \Phi^{(\lambda)}\dbcoh(X_1), \iota^*[E_1^{(\lambda)}], \dots, \iota^*[E_n^{(\lambda)}], \rangle & = \langle\Phi^{(\lambda)}\dbcoh(X_1)\otimes\Lc^{\otimes k}, \iota^*[E_1^{(\lambda)}\otimes L^{\otimes k}], \dots, \iota^*[E_n^{(\lambda)}\otimes L^{\otimes k}], \rangle\\
                    &=:\langle\Phi^{(\lambda+1)}\dbcoh(X_1),\iota^*[E_{1}^{(\lambda+1)}],
                       \dots,\iota^*[E_n^{(\lambda+1)}]\rangle
                \end{split}
            \end{equation}
            
            where $\Phi^{(\lambda+1)}:= \Lc^{\otimes k}\otimes\Phi^{(\lambda)}$.
        \end{enumerate}
        
        In this way, we showed that for every $\lambda$ there is the following semiorthogonal decomposition:
        
        \begin{equation}
            \begin{split}
                \dbcoh(\Mc) = \langle \Phi^{(\lambda)}\dbcoh(X_1), \iota^*[E_1^{(\lambda)}],\dots,\iota^*[E_n^{(\lambda)}]\rangle.
            \end{split}
        \end{equation}
        
        In  particular, for $\lambda = R+1$ we obtain:
        
        \begin{equation}\label{eq_blockwiseaftermutations}
            \begin{split}
                \dbcoh(\Mc) &= \langle \Phi\dbcoh(X_1), \iota^*[F_1], \dots, \iota^*[F_n]\rangle
            \end{split}
        \end{equation}
        
        On the other hand, starting from the full exceptional collection $G/P_2 = \langle K_1, \dots, K_m\rangle$, by (Sam06, Theorem 3.1) and (Orl03, Proposition 2.10) one has
        
        \begin{equation}\label{eq_blockwiseinitialcollection2intheorem}
            \begin{split}
                \dbcoh(\Mc) &= \langle \phi_2\dbcoh(X_2), \iota^*[F_1], \dots, \iota^*[F_n]\rangle
            \end{split}
        \end{equation}
        
        The proof is completed by comparing \ref{eq_blockwiseaftermutations} with \ref{eq_blockwiseinitialcollection2intheorem}.
    \end{proof}

    \begin{remark}
        Note that Theorem \ref{thm_derivedequivalencefibrations} holds for any nonzero $S\in H^0(\Zc, \Lc)$ with smooth zero locus, if the pair $(\Zc, \Lc)$ satisfies Assumption \ref{ass_roofbundle}. In fact, by (Orl03 Proposition 2.10), for every nonzero section $S$ its zero locus admits the semiorthogonal decompositions \ref{eq_bigcollectionsshort}.
    \end{remark}
    
    As we will show in Lemma \ref{lem_cohomologyconditions}, Theorem \ref{thm_derivedequivalencefibrations} can be immediately applied to all cases of roofs where a sequence of mutations realizing a derived equivalence of a Calabi--Yau pair is known, provided that $\Lc$ satisfies Assumption \ref{itm:ass_a1}.

\begin{lemma}\label{lem_cohomologyconditions}
    The roofs of type $A^M_n$, $A_n\times A_n$, $A^G_{4}$, $C_2$ and $G_2$ satisfy Assumptions \ref{itm:ass_a2} and \ref{itm:ass_a3}.
\end{lemma}

\begin{proof}
    Let $G/P$ be a roof of the types listed above. Then, in the language of Notation \ref{notation}, there is a sequence of mutations of suitable semiorthogonal decomposition for the zero locus of a general section of $L$, providing a derived equivalence of the associated Calabi--Yau pair: these mutations are either applications of the Serre functor, overall twists by a power of $L$, mutations of a subcategory equivalent to $\dbcoh(Y_1)$ through an exceptional object, or mutations of pairs of exceptional objects. Hence, the proof reduces to ensure that all mutations of pairs $\LL_{l^*E} l^*F$ or $\RR_{l^*F} l^*E$ satisfy the requirement of Assumption \ref{itm:ass_a2}. This property will be checked case by case in Appendix \ref{appendix}, with Lemma \ref{lem_mukaiderivedequivalencefibers} for $A^M_n$, Lemma \ref{lem_derivedequivalenceAnxAn} for $A_n\times A_n$, Lemma \ref{lem_derivedequivalenceC2} for $C_2$ and Lemma \ref{prop_derivedeqAG4} for $A^G_4$, while the case of $G_2$ is treated in Corollary \ref{cor_derivedequivalenceG2}. %The relevant mutations have been computed in \cite{kuznetsovimou}: to fulfill Condition \ref{cond_mutationsroof}  it is enough to note that the vanishings of \cite[Corollary 2]{kuznetsovimou} of vector bundles on $M$ hold identically on $G/P$, and that ... ... ...
\end{proof}

Even if a general argument is lacking, for all roof bundles where a proof of derived equivalence based on mutations of the associated Calabi--Yau pair is known, such mutations satisfy Assumptions \ref{itm:ass_a2} and \ref{itm:ass_a3}. Therefore, in light of \cite[Conjecture 2.6]{kr2}, we formulate the following:
\begin{conjecture}\label{conj_derivedequivalenceeverywhere}
Let $G/P$ be a homogeneous roof, and $\Zc$ a homogeneous roof bundle of type $G/P$ with projective bundle structures $p_i:\Zc\arw\Zc_i$ for $i\in\{1;2\}$. Given a general section $S\in H^0(\Zc, \Lc)$, the Calabi--Yau fibrations $X_i:=Z(p_{i*}S)$ are derived equivalent.
\end{conjecture}

We summarize all the evidence we have in support of Conjecture \ref{conj_derivedequivalenceeverywhere} in the form of a corollary for Theorem \ref{thm_derivedequivalencefibrations} and Lemma \ref{lem_cohomologyconditions}.

\begin{corollary}\label{cor_knowncases} (Theorem \ref{thm_intro_1})
    Let $\Zc$ be a roof bundle of type $A_n\times A_n$, $A^M_n$, $A^G_4$, $C_2$ or $G_2$ and let $\Lc$ satisfy Assumption \ref{itm:ass_a1}. Then, given a general section $S$ of $\Lc$, the associated pair of Calabi--Yau fibrations is derived equivalent.
\end{corollary}

\begin{proof}
    By Lemma \ref{lem_cohomologyconditions}, Assumptions $\ref{itm:ass_a2}$ and $\ref{itm:ass_a3}$ are satisfied. Thus, by requiring that $\Lc$ satisfies Assumption $\ref{itm:ass_a1}$, the claim is proven by Theorem \ref{thm_derivedequivalencefibrations}.
\end{proof}
    
{\section{Simple {K-equivalence} and roof bundles}\label{sec_simplekequivalence}}

{\subsection{Setup and notation}}
\noindent Let $\Xc_1$, $\Xc_2$ be smooth projective varieties. We call $K$-equivalence a birational morphism $\mu:\Xc_1\dashrightarrow \Xc_2$ such that there is the following diagram:
\begin{equation}\label{eq_keqdiagramsmall}
    \begin{tikzcd}[row sep=large, column sep = normal]
    & \Xc_0\ar[swap]{ld}{g_1}\ar{rd}{g_2} &\\
    \Xc_1\ar[swap, dashed]{rr}{\mu} & & \Xc_2
    \end{tikzcd}
\end{equation}
where $\Xc_0$ is a smooth projective variety and $g_1$ and $g_2$ are birational maps fulfilling $g_1^*K_{\Xc_1}\simeq g_2^*K_{\Xc_2}$. By the $DK$-conjecture \cite{bondalorlovdk, kawamatadk}, two $K$-equivalent varieties are expected to be derived equivalent. We can provide some evidence to this conjecture, and establish a method to verify it for the class of simple $K$-equivalent maps, under some assumption on the resolution $\Xc_0$.\\
\\
A simple $K$-equivalence, following the notation of Diagram \ref{eq_keqdiagramsmall}, is a $K$-equivalence $\mu$ such that $g_1$ and $g_2$ are blowups in smooth centers: by \cite[Lemma 2.1]{li} $g_1$ and $g_2$ have the same exceptional divisor $\Zc$. Then, by the structure theorem for simple $K$-equivalence \cite[Thm. 0.2]{kanemitsu}, $\Zc$ is a family of roofs over a smooth projective variety $B$. In short, we will say that $\mu$ has exceptional divisor $\Zc$ when $\Zc$ is the exceptional divisor of $g_1$ and $g_2$ (cfr. \cite[ Definition 1.2]{kanemitsu}). It is natural to consider the case where $\Zc$ is a homogeneous roof bundle in the sense of Definition \ref{def_roofbundles}:

\begin{definition}
We say that a simple $K$-equivalence $\mu$ is \emph{homogeneous} of type $G/P$ if its exceptional divisor is a homogeneous roof bundle of type $G/P$ over a smooth projective variety $B$.
\end{definition}

For every homogeneous simple $K$-equivalence $\mu$ of type $G/P$ there exists the following diagram:
\begin{equation}\label{eq_keqdiagram}
    \begin{tikzcd}[row sep=large, column sep = scriptsize]
    &   & \Zc\ar[hook]{d}{f}\ar[swap]{lldd}{p_1}\ar{rrdd}{p_2} &  &\\  
    &   & \Xc_0\ar[swap]{ld}{g_1}\ar{rd}{g_2} & &\\
    \Zc_1\ar[hook]{r}\ar[swap]{rrdd}{r_1}&   \Xc_1\ar[dashed]{rr}{\mu} & & \Xc_2\ar[hookleftarrow]{r} &\Zc_2\ar{lldd}{r_2}\\
    &&\color{white}{AAAA}&&\\
    &&B&&
    \end{tikzcd}
\end{equation}
which is a simple adaptation of \cite[Diagram 0.2.1]{kanemitsu} to our setting.

\section{Simple K-equivalence and mutations}
In the remainder of this chapter we will extensively use the language established in Notation \ref{notation}. Let us perform the following computation for later convenience.
\begin{lemma}\label{lem_canonicalbundlekequivalence}
    In the language of Notation \ref{notation}, one has $\omega_{\Xc_0}|_\Zc\simeq\Lc^{\otimes(-r+1)}\otimes\pi^* T $ for some line bundle $ T \in\dbcoh(B)$.
\end{lemma}
\begin{proof}
    We recall that by \cite[Example 3.2.11]{fulton} one has $\omega_\Zc\simeq p_i^*\omega_{\Zc_i}\otimes p_i^*\det\Ec_i\otimes\Lc^{\otimes(-r)}$, and since $(r_i^{-1}(b), \Ec_i|_{r_i^{-1}(b)})$ is a Mukai pair for every $b\in B$, we get $\omega_\Zc \simeq\Lc^{\otimes (-r)}$ up to twists by pullbacks of line bundles from $B$. By the normal bundle sequence:
    \begin{equation}
        0\arw T_\Zc\arw f^*T_{\Xc_0}\arw\Nc_{\Zc|\Xc_0}\arw 0
    \end{equation}
    one finds $f^*\omega_{\Xc_0} = \omega_\Zc\otimes\det\Nc_{\Zc|\Xc_0}^\vee = \Lc^{\otimes (-r)}\otimes\det\Nc_{\Zc|\Xc_0}^\vee$. The proof is completed by the fact that, since $\Zc$ is the exceptional divisor of the blowup $g_i:\Xc_0\arw\Xc_i$, by \cite[Proposition 1.4]{kanemitsu} one has $\Nc_{\Zc|\Xc_0}\simeq f^*\Oc(\Zc) \simeq \Lc^\vee$.
\end{proof}

\begin{lemma}\label{lem_twistkequivalence}
    In the language of Notation \ref{notation}, assume that for either $i=1$ or $i=2$ there is a semiorthogonal decomposition $\dbcoh(\Xc_0) = \langle \sigma\dbcoh(\Xc_i), f_*[W_1],\dots, f_*[W_n]\rangle$ where every $W_j$ is a homogeneous vector bundle on $G/P$ and $\sigma$ is a fully faithful functor. Then one has:
    \begin{equation}
        \begin{split}
            \dbcoh(\Xc_0) &= \langle \sigma\dbcoh(\Xc_i)\otimes\Oc(-t\Zc), f_*[W_1\otimes L^{\otimes t}], \dots,  f_*[W_n\otimes L^{\otimes t}]\rangle
        \end{split}
    \end{equation}
    for every $t\in\ZZ$.
\end{lemma}
\begin{proof}
    Let us start with the decomposition 
    \begin{equation}
        \begin{split}
            \dbcoh(\Xc_0) &= \langle \sigma\dbcoh(\Xc_i), f_*[W_1], \dots,  f_*[W_n]\rangle.
        \end{split}
    \end{equation}
    Note that, since $\Zc$ is the exceptional divisor of the blowup $\Xc_0\arw\Xc_i$ and $\Lc$ is the Grothendieck line bundle of $\Zc$ as a projective bundle over $\Zc_i$, by \cite[Proposition 1.4]{kanemitsu} one has $f^*\Oc(-\Zc)\simeq \Lc$. Hence, for any integer $t$, let us twist the collection above by $\Oc(-t\Zc)$, obtaining:
    \begin{equation}
        \begin{split}
            \dbcoh(\Xc_0) = \langle \sigma\dbcoh(\Xc_i)\otimes\Oc(-t\Zc),\hspace{122pt}\\ f_*[W_1]\otimes\Oc(-t\Zc), \dots,  f_*[W_n]\otimes\Oc(-t\Zc)\rangle.
        \end{split}
    \end{equation}
    We conclude the proof by showing that for every $j$ one has $f_*[W_j]\otimes\Oc(-t\Zc)\simeq f_*[W_j\otimes L^{\otimes t}]$. This assertion follows simply by projection formula. In fact, for every $\Gc\in\dbcoh(B)$ we have
    \begin{equation}
        \begin{split}
            f_*(\Wc_j\otimes\pi^*\Gc)\otimes \Oc(-t\Zc) & \simeq f_*(\Wc_j\otimes\pi^*\Gc\otimes f^*\Oc(-t\Zc)) \\
            & \simeq f_*(\Wc_j\otimes\pi^*\Gc\otimes \Lc^{\otimes t}).
        \end{split}
    \end{equation}
\end{proof}

\begin{lemma}\label{lem_sodforkequivalence}
    In the language of Notation \ref{notation} there are the following semiorthogonal decompositions:
    \begin{equation}
        \begin{split}
            \dbcoh(\Xc_0) = \langle \wt g_1^*\dbcoh(\Xc_1), f_*[E_1], \dots, f_*[E_n] \rangle & = \langle \wt g_2^*\dbcoh(\Xc_2), f_*[F_1], \dots, f_*[F_n] \rangle.
        \end{split}
    \end{equation}
    where $\wt g_i^*:= g_i^*(-)\otimes\Oc(-\Zc)$.
\end{lemma}
\begin{proof}
    Let us start by applying Orlov's blowup decomposition. For $-r+1\leq k\leq -1$ one has the following fully faithful functors:
    \begin{equation}
        \begin{tikzcd}[row sep = tiny, column sep = large, /tikz/column 1/.append style={anchor=base east} ,/tikz/column 2/.append style={anchor=base west}]
            \Theta_k: \dbcoh(\Zc_1) \ar{r} & \dbcoh(\Xc_0) \\
            \Fc \ar[maps to]{r} & f_*(p_1^*\Fc\otimes\Lc^{\otimes k})
        \end{tikzcd}
    \end{equation}
    and the semiorthogonal decomposition:
    \begin{equation}
        \begin{split}
            \dbcoh(\Xc_0) &= \langle \Theta_{-r+1}\dbcoh(\Zc_1), \dots, \Theta_{-1}\dbcoh(\Zc_1), g_1^*\dbcoh(\Xc_1) \rangle.
        \end{split}
    \end{equation}
    By applying \cite[Theorem 3.1]{samokhin} to each copy of $\dbcoh(\Zc_1)$ with respect to the collection $\dbcoh(\Zc_1) = \langle J_1, \dots, J_m\rangle$ we find:
    \begin{equation}\label{eq_kequivalenceorlovsamokhin}
        \begin{split}
            \dbcoh(\Xc_0) = \langle f_*(\Fm(J_1)\otimes\Lc^{\otimes(-r+1)}\otimes\pi^*\dbcoh(B)), \dots, f_*(\Fm(J_m)\otimes\Lc^{\otimes(-r+1)}\otimes\pi^*\dbcoh(B)),\\
            \vdots \hspace{150pt} \vdots \hspace{45pt}\\
            f_*(\Fm(J_1)\otimes\Lc^{\otimes(-1)}\otimes\pi^*\dbcoh(B)), \dots, f_*(\Fm(J_m)\otimes\Lc^{\otimes(-1)}\otimes\pi^*\dbcoh(B)),\\
            g_1^*\dbcoh(\Xc_1) \rangle
        \end{split}
    \end{equation}
    where $\Fc$ is the functor defined in Equation \ref{eq_relativizationfunctor}. Let us now apply the inverse Serre functor to the whole semiorthogonal complement of $g_1^*\dbcoh(\Xc_1)$. Then, for $1\leq i\leq m$ and $1\leq j\leq r-1$ the block $f_*(\Fm(J_i)\otimes\Lc^{\otimes(-j)}\otimes\pi^*\dbcoh(B))$ will become $f_*(\Fm(J_i)\otimes\Lc^{\otimes(-j)}\otimes\pi^*\dbcoh(B))\otimes\omega_{\Xc_0}^\vee$. By the projection formula we compute:
    \begin{equation}
        \begin{split}
            f_*(\Fm(J_i)\otimes\Lc^{\otimes(-j)}\otimes\pi^*\dbcoh(B))\otimes\omega_{\Xc_0}^\vee \\ =f_*(\Fm(J_i)\otimes\Lc^{\otimes(-j)}\otimes\pi^*\dbcoh(B))\otimes f^*\omega_{\Xc_0}^\vee) \\
            =f_*(\Fm(J_i)\otimes\otimes\pi^*\dbcoh(B))\otimes \Lc^{\otimes(r-j-1)})
        \end{split}
    \end{equation}
    where the second equality follows from Lemma \ref{lem_canonicalbundlekequivalence}. Substituting this in the decomposition \ref{eq_kequivalenceorlovsamokhin} we get:
    \begin{equation}
        \begin{split}
            \dbcoh(\Xc_0) = \langle & g_1^*\dbcoh(\Xc_1),\hspace{250pt} \\
            & f_*(\Fm(J_1)\otimes\pi^*\dbcoh(B)), \dots\dots\dots\dots\dots\dots\dots, f_*(\Fm(J_m)\otimes\pi^*\dbcoh(B)),\\
            &\hspace{30pt} \vdots \hspace{190pt} \vdots \hspace{45pt}\\
            &f_*(\Fm(J_1)\otimes\Lc^{\otimes(r-2)}\otimes\pi^*\dbcoh(B)), \dots, f_*(\Fm(J_m)\otimes\Lc^{\otimes(r-2)}\otimes\pi^*\dbcoh(B))\rangle.
        \end{split}
    \end{equation}
    Finally, we twist the whole collection by $\Oc(-\Zc)$. By Lemma \ref{lem_twistkequivalence}, what we obtain is:
    \begin{equation}
        \begin{split}
            \dbcoh(\Xc_0) = \langle & \wt g_1^*\dbcoh(\Xc_1),\hspace{250pt} \\
            & f_*(\Fm(J_1)\otimes\pi^*\dbcoh(B)\otimes\Lc), \dots\dots\dots\dots\dots\dots\dots, f_*(\Fm(J_m)\otimes\pi^*\dbcoh(B)\otimes\Lc),\\
            & \hspace{30pt} \vdots \hspace{190pt} \vdots \hspace{45pt}\\
            & f_*(\Fm(J_1)\otimes\Lc^{\otimes(r-1)}\otimes\pi^*\dbcoh(B)), \dots, f_*(\Fm(J_m)\otimes\Lc^{\otimes(r-1)}\otimes\pi^*\dbcoh(B))\rangle.
        \end{split}
    \end{equation}
    The result follows by comparing this decomposition with Equation \ref{eq_bigcollectionsshort}.
\end{proof}

The following lemma and its relevance in the further computations are analogous to Lemma \ref{lem_Lsemiorthogonalcommutes} and the role it played in Section \ref{sec_deriverdequivalenceCYfibrations}.

\begin{lemma}\label{lem_Lsemiorthogonalcommuteskequivalence}
    Let $\pi:\Zc\arw B$ be a flat and proper morphism of smooth projective varieties, consider a closed immersion $f:\Zc\arw\Xc$ of codimension one, where $\Xc$ is smooth and projective. Suppose there exist admissible subcategories $\Cc\subset \dbcoh(\Xc)$, $\Dc\subset\dbcoh(\Zc)$ and vector bundles $\Ec, \Fc\in\dbcoh(\Zc)$ relatively exceptional over $B$ such that one has the following strong, $B$-linear semiorthogonal decompositions:
    
    \begin{equation}
        \begin{split}
            \dbcoh(\Zc) &= \langle\Dc, \Ec\otimes\pi^*\dbcoh(B),                                    \Fc\otimes\pi^*\dbcoh(B)\rangle\\ 
            \dbcoh(\Xc) &= \langle\Cc, f_*(\Ec\otimes\pi^*\dbcoh(B)),                                    f_*(\Fc\otimes\pi^*\dbcoh(B))\rangle.
        \end{split}
    \end{equation}
    Then, if $\Ec$ is $\Lc$-semiorthogonal to $\Fc$, $\LL_{\langle \Ec\otimes\pi^*\dbcoh(B)\rangle}\Fc\otimes\pi^*\dbcoh(B)$ commutes with $f_*$.
\end{lemma}

\begin{proof}
    Let us recall the functors $\Psi_{\langle\Ec\otimes\pi^*\dbcoh(B)\rangle}$ and $\Psi_{\langle\Ec\otimes\pi^*\dbcoh(B)\rangle}^!$ defined by Equations \ref{eq_psi} and \ref{eq_psi!}. If we apply Lemma \ref{lem_categoricallemma} to the functors $L_1:= f_*$, $R_1:=f^!$, $L_2:= \Psi_{\langle\Ec\otimes\pi^*\dbcoh(B)\rangle}$, $R_2:=\Psi_{\langle\Ec\otimes\pi^*\dbcoh(B)\rangle}^!$, we find the commutative diagram:
    
    \begin{equation}\label{eq_diagramkequivalencecats}
        \begin{tikzcd}[row sep = huge]
            f_*\Psi_{\langle\Ec\otimes\pi^*\dbcoh(B)\rangle}\Psi_{\langle\Ec\otimes\pi^*\dbcoh(B)\rangle}^!\Fc\otimes\pi^*\Gc\ar{r}{f_*\epsilon_\Zc}\ar{d}{\beta} & f_*(\Fc\otimes\pi^*\Gc)\ar[equals]{d}\\
            \Xi_{\langle f_*(\Ec\otimes\pi^*\dbcoh(B))\rangle}\Xi_{\langle f_*(\Ec\otimes\pi^*\dbcoh(B))\rangle}^!f_*(\Fc\otimes\pi^*\Gc)\ar{r}{\epsilon_{\Xc}} & f_*(\Fc\otimes\pi^*\Gc)
        \end{tikzcd}
    \end{equation}
    
    where we defined $\beta:= f_*\Psi_{\langle\Ec\otimes\pi^*\dbcoh(B)\rangle}\Psi_{\langle\Ec\otimes\pi^*\dbcoh(B)\rangle}^!(\eta_{1, \Fc\otimes\pi^*\Gc})$, $\epsilon_\Zc:= \epsilon_{2,\Fc\otimes\pi^*\Gc}$, $\epsilon_\Xc := \epsilon_{12, f_*(\Fc\otimes\pi^*\Gc)}$ and \linebreak  $\Xi_{\langle f_*(\Ec\otimes\pi^*\dbcoh(B))\rangle}:= f_*\Psi_{\langle\Ec\otimes\pi^*\dbcoh(B)\rangle}$. Let us now prove the following:\\
    \textbf{Claim.} \emph{The map $\beta$ is an isomorphism if $\Ec$ is $\Lc$-semiorthogonal to $\Fc$.}\\
    Proving the claim is equivalent to show that under the requirement of $\Lc$-semiorthogonality one has:
    
    \begin{equation}
        \begin{split}
            f_*\Psi_{\langle\pi^*\dbcoh(B)\otimes\Ec\rangle}\Psi_{\langle\pi^*\dbcoh(B)\otimes\Ec\rangle}^! f^!f_*(\Fc\otimes\pi^*\Gc) =& f_*\Psi_{\langle\pi^*\dbcoh(B)\otimes\Ec\rangle}\Psi_{\langle\pi^*\dbcoh(B)\otimes\Ec\rangle}^! (\Fc\otimes\pi^*\Gc).
        \end{split}
    \end{equation}
    
    We start with the following chain of equalities:
    
    \begin{equation}\label{eq_kequivalencecomputation}
        \begin{split}
            f_*\Psi_{\langle\pi^*\dbcoh(B)\otimes\Ec\rangle}\Psi_{\langle\pi^*\dbcoh(B)\otimes\Ec\rangle}^! f^!f_*(\Fc\otimes\pi^*\Gc) =& f_*(\Ec\otimes \pi^*\pi_*R\Hc om_\Zc(\Ec, f^!f_*(\Fc\otimes\pi^*\Gc))) \\
            =& f_*(\Ec\otimes \pi^*\pi_*(\Ec^\vee\otimes f^!f_*(\Fc\otimes\pi^*\Gc)))\\
            =& f_*(\Ec\otimes\pi^*\pi_*(\Ec^\vee\otimes f^*f_*(\Fc\otimes\pi^*\Gc)\otimes\Lc^\vee[-1]))
        \end{split}
    \end{equation}
    
    where we used \cite[Corollary 3.35]{huybrechts}. By \cite[Corollary 11.4]{huybrechts} and Lemma \ref{lem_canonicalbundlekequivalence} one has a distinguished triangle
    
    \begin{equation}
        \begin{split}
            \Fc\otimes\pi^*\Gc\otimes\Lc[1]\arw f^*f_*(\Fc\otimes\pi^*\Gc)\arw 
            \Fc\otimes\pi^*\Gc\arw \Fc\otimes\pi^*\Gc\otimes\Lc[2]
        \end{split}
    \end{equation}
    
    on which we apply the tensor product by $\Ec^\vee\otimes\Lc^\vee[-1]$, to find:
    
    \begin{equation}
        \begin{split}
            \Ec^\vee\otimes\Fc\otimes\pi^*\Gc\arw \Ec^\vee\otimes f^*f_*(\Fc\otimes\pi^*\Gc)\otimes\Lc^\vee[-1]\arw
            \Ec^\vee\otimes\Fc\otimes\pi^*\Gc\otimes\Lc^\vee[-1]\arw \Ec^\vee\otimes\Fc\otimes\pi^*\Gc[1]
        \end{split}
    \end{equation}
    
    By applying the derived $\pi_*$ we get:
    
    \begin{equation}
        \begin{split}
            \dots\arw R^m\pi_*(\Ec^\vee\otimes\Fc\otimes\pi^*\Gc)\arw R^{m-1}\pi_*(\Ec^\vee\otimes f^*f_*(\Fc\otimes\pi^*\Gc)\otimes\Lc^\vee)\arw \\
            R^{m-1}\pi_*(\Ec^\vee\otimes\Fc\otimes\pi^*\Gc\otimes\Lc^\vee)\arw R^{m+1}(\pi_*\Ec^\vee\otimes\Fc\otimes\pi^*\Gc)\arw\cdots
        \end{split}
    \end{equation}
    
    By $\Lc$ -semiorthogonality of $\Ec$ and $\Fc$ one has $R^m\pi_*(\Ec^\vee\otimes\Fc\otimes\pi^*\Gc\otimes\Lc^\vee) = 0$ for every $m$, hence we get an isomorphism:
    
    \begin{equation}
        \pi_*(\Ec^\vee\otimes f^*f_*(\Fc\otimes\pi^*\Gc)\otimes\Lc^\vee[-1])\simeq \pi_*(\Ec^\vee\otimes \Fc\otimes\pi^*\Gc)
    \end{equation} 
    
    If we substituting this in Equation \ref{eq_kequivalencecomputation} we find:
    
    \begin{equation}
        \begin{split}
            f_*\Psi_{\langle\pi^*\dbcoh(B)\otimes\Ec\rangle}\Psi_{\langle\pi^*\dbcoh(B)\otimes\Ec\rangle}^! f^!f_*(\Fc\otimes\pi^*\Gc) =& f_*(\Ec\otimes\pi^*\pi_*(\Ec^\vee\otimes\Fc\otimes\pi^*\Gc)) \\
            =& f_*\Psi_{\langle\pi^*\dbcoh(B)\otimes\Ec\rangle}\Psi_{\langle\pi^*\dbcoh(B)\otimes\Ec\rangle}^! (\Fc\otimes\pi^*\Gc).
        \end{split}
    \end{equation}
    
    thus proving the claim.\\
    We are now ready to prove the isomorphism:
    \begin{equation}
        f_*\LL_{\langle \Ec\otimes\pi^*\dbcoh(B)\rangle}\Fc\otimes\pi^*\Gc\simeq \LL_{\langle f_*(\Ec\otimes\pi^*\dbcoh(B))\rangle}f_*(\Fc\otimes\pi^*\Gc).
    \end{equation}
    Let us recall the distinguished triangle in $\dbcoh(\Zc)$:
    \begin{equation}
        \begin{split}
            \Psi_{\langle\Ec\otimes\pi^*\dbcoh(B)\rangle}\Psi_{\langle\Ec\otimes\pi^*\dbcoh(B)\rangle}^!\Fc\otimes\pi^*\Gc\xrightarrow{\epsilon_\Zc}\\
                \arw\Fc\otimes\pi^*\Gc\arw \LL_{\langle\Ec\otimes\pi^*\dbcoh(B)\rangle}\Fc\otimes\pi^*\Gc.
        \end{split}
    \end{equation}
    If we apply the derived pushforward $f_*$ we find \cite[proof of Corollary 2.50]{huybrechts}:
    \begin{equation}
        \begin{split}
             R^0 f_*\Psi_{\langle\Ec\otimes\pi^*\dbcoh(B)\rangle}\Psi_{\langle\Ec\otimes\pi^*\dbcoh(B)\rangle}^!\Fc\otimes\pi^*\Gc\xrightarrow{f_*\epsilon_\Zc}\\
            \arw R^0 f_*(\Fc\otimes\pi^*\Gc)\arw R^0 f_* \LL_{\langle\Ec\otimes\pi^*\dbcoh(B)\rangle}\Fc\otimes\pi^*\Gc\\
            R^{1} f_*\Psi_{\langle\Ec\otimes\pi^*\dbcoh(B)\rangle}\Psi_{\langle\Ec\otimes\pi^*\dbcoh(B)\rangle}^!\Fc\otimes\pi^*\Gc\arw\cdots
        \end{split}
    \end{equation}
    %where one has the relation $R^{i+1} f_*\Psi_{\langle\Ec\otimes\pi^*\dbcoh(B)\rangle}\Psi_{\langle\Ec\otimes\pi^*\dbcoh(B)\rangle}^!\Fc\otimes\pi^*\Gc = $ $R^{i} f_*\Psi_{\langle\Ec\otimes\pi^*\dbcoh(B)\rangle}\Psi_{\langle\Ec\otimes\pi^*\dbcoh(B)\rangle}^!\Fc\otimes\pi^*\Gc[1]$.
    On the other hand, one has the distinguished triangle:
    \begin{equation}
        \begin{split}
            \Xi_{\langle f_*(\Ec\otimes\pi^*\dbcoh(B))\rangle}\Xi_{\langle f_*(\Ec\otimes\pi^*\dbcoh(B))\rangle}^!f_*(\Fc\otimes\pi^*\Gc)\xrightarrow{\epsilon_{\Xc}} \\
            f_*(\Fc\otimes\pi^*\Gc)\arw \LL_{\langle f_*(\Ec\otimes\pi^*\dbcoh(B))\rangle}f_*(\Fc\otimes\pi^*\Gc)\arw\cdots
        \end{split}
    \end{equation}
    and the proof is concluded by commutativity of Diagram \ref{eq_diagramkequivalencecats} and \cite[page 232, Corollary 4]{gelfandmanin}.
\end{proof}

\begin{lemma}\label{lem_kequivalenceinducedmutationsinX0}
    In the language of Diagram \ref{eq_keqdiagram}, consider a set $\{W_1, \dots, W_n\}$ of exceptional, semiorthogonal objects of $\dbcoh(G/P)$ such that for either $i=1$ or $i=2$ one has a semiorthogonal decomposition:
    \begin{equation}\label{eq_sodX0withWs}
        \dbcoh(\Xc_0) = \langle \sigma\dbcoh(\Xc_i),  f_*[W_1], \dots,  f_*[W_n]\rangle.
    \end{equation}
    where $\sigma$ is a fully faithful functor. Then, if $W_i$ is $L$-semiorthogonal to $W_{i+1}$ one also has:
    \begin{equation}
        \begin{split}
            \dbcoh(\Xc_0) = \langle & \sigma\dbcoh(\Xc_i),  f_*[W_1], \dots,  f_*[W_{i-1}], \\
                                    &  f_*[\LL_{W_i}W_{i+1}],  f_*[W_{i+1}], f_*[W_{i+2}], \dots,  f_*[W_n]\rangle\\
            \dbcoh(\Xc_0) = \langle & \sigma\dbcoh(\Xc_i),  f_*[W_1], \dots,  f_*[W_{i-1}], \\
                                    &  f_*[W_{i+1}],  f_*[\RR_{W_{i+1}}W_i], f_*[W_{i+2}], \dots,  f_*[W_n]\rangle.
        \end{split}
    \end{equation}
\end{lemma}
\begin{proof}
    We will only prove the statement about left mutations, the other being nearly identical. Starting from the decomposition \ref{eq_sodX0withWs} and applying \cite[Corollary 2.9]{kuznetsovcubic4folds}, one finds:
    \begin{equation}\label{eq_lemmastartingpointsod}
        \begin{split}
            \dbcoh(\Xc_0) = \langle & \sigma\dbcoh(\Xc_1),  f_*[W_1], \dots,  f_*[W_{i-1}], \\
                                    & \LL_{ f_*[W_i]} f_*[W_{i+1}],  f_*[W_{i}], f_*[W_{i+2}], \dots,  f_*[W_n]\rangle.
        \end{split}
    \end{equation}
    By Lemma \ref{lem_Lsemiorthogonal}, since $W_i$ is $L$-semiorthogonal to $W_{i+1}$ it follows that $\Wc_i$ is $\Lc$-semiorthogonal to $\Wc_{i+1}$ (the terminology is from Notation \ref{notation}), hence we can apply Lemma \ref{lem_Lsemiorthogonalcommuteskequivalence} to replace $\LL_{ f_*[W_i]} f_*[W_{i+1}]$ with $ f_*\LL_{[W_i]}[W_{i+1}]$ in the decomposition \ref{eq_lemmastartingpointsod}. By Lemma \ref{lem_inducedmutationsinZ} we have $\LL_{[W_i]}[W_{i+1}]\simeq [\LL_{W_i} W_{i+1}]$ and this concludes the proof.
\end{proof}

\begin{proposition}\label{prop_serrefunctorkequivalence}
        In the language of Notation \ref{notation}, assume that for either $i=1$ or $i=2$ there is a semiorthogonal decomposition $\dbcoh(\Xc_0) = \langle \sigma\dbcoh(\Xc_i), f_*[W_1],\dots, f_*[W_n]\rangle$ where every $W_j$ is a homogeneous vector bundle on $G/P$ and $\sigma$ is a fully faithful functor. Then one has:
    \begin{equation}
        \begin{split}
            \dbcoh(\Xc_0) &= \langle \RR_{f_*[W_1]}\sigma\dbcoh(\Xc_i), f_*[W_2], \dots,  f_*[W_n],  f_*[W_1\otimes L^{\otimes(r-1)}]\rangle
        \end{split}
    \end{equation}
\end{proposition}

\begin{proof}
    Let us start by mutating $\sigma\dbcoh(\Xc_i)$ one step to the right. We obtain:
    \begin{equation}
        \begin{split}
            \dbcoh(\Xc_0) &= \langle f_*[W_1], \RR_{f_*[W_1]}\sigma\dbcoh(\Xc_i), f_*[W_2], \dots,  f_*[W_n]\rangle.
        \end{split}
    \end{equation}
    By the inverse Serre functor  we find:
    \begin{equation}\label{eq_startingpointserrefunctorkequivalence}
        \dbcoh(\Xc_0) = \langle \RR_{f_*[W_1]}\sigma\dbcoh(\Xc_i), f_*[W_2],\dots, f_*[W_n], f_*[W_1]\otimes\omega^\vee_{\Xc_0}\rangle.
    \end{equation}
    For every $\Gc\in\dbcoh(B)$ one has:
    \begin{equation}
        \begin{split}
            f_*(\Wc_1\otimes\pi^*\Gc)\otimes\omega^\vee_{\Xc_0} &= f_*(\Wc_1\otimes\pi^*\Gc\otimes f^*\omega^\vee_{\Xc_0})\\
            & = f_*(\Wc_1\otimes\pi^*\Gc\otimes\Lc^{\otimes(r-1)}\otimes\pi^* T ).
        \end{split}
    \end{equation}
    where $ T $ is a line bundle on $B$. The first isomorphism is by projection formula and the second by Lemma \ref{lem_canonicalbundlekequivalence}. This shows that we can substitute $f_*[W_1]\otimes\omega^\vee_{\Xc_0}$ with $f_*[W_1\otimes L^{\otimes(r-1)}]$ in the decomposition \ref{eq_startingpointserrefunctorkequivalence}, hence proving our claim.
\end{proof}

We are now ready to prove the main theorem of this chapter:

\begin{theorem}\label{thm_derivedequivalencekequivalence}
    Let $\mu:\Xc_1\dashrightarrow\Xc_2$ be a homogeneous simple $K$-equivalence of type $G/P$, with exceptional divisor $\Zc$ satisfying Assumptions \ref{itm:ass_a2} and \ref{itm:ass_a3}. Then $\mu$ satisfies the DK conjecture, i.e. there is an equivalence of categories $\rho:\dbcoh(\Xc_1)\arw\dbcoh(\Xc_2)$.
\end{theorem}

\begin{proof}
    Consider the data of Diagram \ref{eq_keqdiagram} defining a homogeneous simpke $K$-equivalence $\mu$ of type $G/P$. Let us give a modification of the sequence of pairs $(\boldsymbol E^{(\lambda)}, \psi^{(\lambda)})$ introduced in Notation \ref{notation}: we consider the same ordered sequence $\boldsymbol E^{(\lambda)} = (E_1^{(\lambda)}, \dots, E_n^{(\lambda)})$, and redefine the functor $\rho^{(\lambda)}$ to be compatible with the present setting. In analogy with the operations \ref{itm:op_o1}, \ref{itm:op_o2}, \ref{itm:op_o3} of Notation \ref{notation}, we define $(\boldsymbol E^{(\lambda+1)}, \rho^{(\lambda+1)})$ to be obtained by one of the following operations:
    \begin{description}%[style=multiline, labelwidth=1.5cm]
        \item[\namedlabel{itm:kequivalencelistmutation}{O4}]  exchange $E_i^{(\lambda)}$ with $E_{i+1}^{(\lambda)}$ and replace either $E_{i+1}^{(\lambda)}$ with $\LL_{E_i^{(\lambda)}}E_{i+1}^{(\lambda)}$ or $E_i^{(\lambda)}$ with $\RR_{E_{i+1}^{(\lambda)}}E_i^{(\lambda)}$, where $E_i^{(\lambda)}$ is $L$-semiorthogonal to $E_{i+1}^{(\lambda)}$, leave $\rho^{(\lambda)}$ unchanged
        \item[\namedlabel{itm:kequivalencelistserre}{O5}] move $E_1^{(\lambda)}$ right after $E_n^{(\lambda)}$ and twist it by $L^{\otimes(r-1)}$, replace $\rho^{(\lambda)}$ with $\RR_{f_*[E_1^{(\lambda)}]}\rho^{(\lambda)}$, or, conversely, move $E_n^{(\lambda)}$ right before $E_1^{(\lambda)}$, twist it by $L^{\otimes(-r+1)}$ and replace $\rho^{(\lambda)}$ with $\LL_{f_*[E_n^{(\lambda)}]}\rho^{(\lambda)}$
        \item[\namedlabel{itm:kequivalencelisttwist}{O6}] for any $t\in\ZZ$, replace $E_i^{(\lambda)}$ with $E_i^{(\lambda)}\otimes L^{\otimes t}$ for all $i$ and replace $\rho^{(\lambda)}$ with $T_{-t} \rho^{(\lambda)}$, where $T_{-t}$ is the twist functor by $\Oc(-t\Zc)$
    \end{description}
    Moreover, we recall that $E_i^{(1)} = E_i$, $E^{(R+1)} = F_i$ for $1\leq i\leq n$ and we impose $\rho^{(1)} = \wt g_1^*$. In light of Lemma \ref{lem_sodforkequivalence}, we can prove our claim by showing that for $1\leq\lambda\leq R+1$ there is the following semiorthogonal decomposition:
    \begin{equation}\label{eq_sodlambdakequivalence}
        \begin{split}
            \dbcoh(\Xc_0) = \langle & \rho^{(\lambda)}\dbcoh(\Xc_1),  f_*[E_1^{(\lambda)}], \dots,  f_*[E_n^{(\lambda)}]\rangle
        \end{split}
    \end{equation}
    In fact, if this claim is true, for $\lambda = R+1$ we obtain a semiorthogonal decomposition 
    \begin{equation}\label{eq_sodlambdakequivalenceending}
        \begin{split}
            \dbcoh(\Xc_0) = \langle & \rho^{(R+1)}\dbcoh(\Xc_1),  f_*[F_1], \dots,  f_*[F_n]\rangle
        \end{split}
    \end{equation}
    and, once we compare it with the semiorthogonal decompositions found by Lemma \ref{lem_sodforkequivalence}, we conclude setting $\rho = \rho^{(R+1)}$.\\
    \\
    The existence of \ref{eq_sodlambdakequivalence} for every $\lambda$ can be proved by induction. For $\lambda = 1$ the decomposition 
    \ref{eq_sodlambdakequivalence} exists by Lemma \ref{lem_sodforkequivalence}. Let us now suppose that \ref{eq_sodlambdakequivalence} exists for $\lambda = \lambda_0$ as well, we will show the existence of such decomposition for $\lambda = \lambda_0+1$ by mutating the former accordingly. In fact, the pair $(\rho^{(\lambda_0+1)},\boldsymbol E^{(\lambda_0+1)})$ is obtained by $(\rho^{(\lambda_0)},\boldsymbol E^{(\lambda_0)})$ by one of the three operations \ref{itm:kequivalencelistmutation}, \ref{itm:kequivalencelistserre} and \ref{itm:kequivalencelisttwist} and all of them induce a mutation on the decomposition \ref{eq_sodlambdakequivalence} for $\lambda=\lambda_0$, hence they define a new semiorthogonal decomposition
    \begin{equation}\label{eq_sodlambda0+1kequivalence}
        \begin{split}
            \dbcoh(\Xc_0) = \langle & \rho^{(\lambda_0+1)}\dbcoh(\Xc_1),  f_*[E_1^{(\lambda_0+1)}], \dots,  f_*[E_n^{(\lambda_0+1)}]\rangle
        \end{split}
    \end{equation}
    as we proved by previous results. More precisely, Operation \ref{itm:kequivalencelistmutation} gives rise to a semiorthogonal decomposition by Lemma \ref{lem_Lsemiorthogonalcommuteskequivalence}, Operation \ref{itm:kequivalencelistserre} has been treated in Proposition \ref{prop_serrefunctorkequivalence} and Operation \ref{itm:kequivalencelisttwist} follows by Lemma \ref{lem_twistkequivalence}. Therefore the semiorthogonal decomposition \ref{eq_sodlambda0+1kequivalence} exists and this concludes the proof.
\end{proof}

The following corollary provides an extension to the results of \cite{bondalorlovflop, kawamatadk, namikawa} on derived equivalence for varieties related by $K$-equivalence of type $A_n\times A_n$ and $A^M_n$ which are respectively standard flops and Mukai flops.

\begin{corollary}\label{cor_knowncasesKeq}(Theorem \ref{thm_intro_3})
    Let $\mu:\Xc_1\dashrightarrow\Xc_2$ be a homogeneous simple $K$-equivalence of type $G/P$, where $G/P$ is a roof of type $A^M_n$, $A_n\times A_n$, $A^G_4$, $C_2$ or $G_2$. %Suppose that $\Lc$ satisfies Assumption \ref{itm:ass_a1}.
    Then $\Xc_1$ and $\Xc_2$ are derived equivalent.
\end{corollary}

\begin{proof}
    We observe that by Lemma \ref{lem_cohomologyconditions}, Assumptions \ref{itm:ass_a2} and \ref{itm:ass_a3} are satisfied by roof bundles of types $A^M_n$, $A_n\times A_n$, $A^G_4$, $C_2$ or $G_2$. Then, the proof follows directly from Theorem \ref{thm_derivedequivalencekequivalence}.
\end{proof}

{\section{GLSM and Calabi--Yau fibrations: the roof bundle of type \texorpdfstring{$A^G_{2k}$}{something}}\label{sec_glsm}}

Let us fix a roof bundle $\Zc$ of type $G/P=F(2,3,V_5)$. Hereafter we present a GLSM describing the zero loci $X_1$ and $X_2$ as vacuum manifolds associated to the critical loci of a superpotential $w$ related by a phase transition, i.e. the quotients of the critical loci by $GL(k+2)$. We will mainly focus our attention to the Calabi--Yau pair of Section \ref{sec_CY}, but we will keep the discussion slightly more general: we fix $B=\PP^{2k+1}$ and consequently $\Zc=F(1,k+1,k+2,2k+2)$.

\subsection{The roof bundle of type \texorpdfstring{$A^G_{2k}$}{something} over \texorpdfstring{$\PP^{2k+1}$}{something}}
The geometry for the case $k=2$ has been established in Section \ref{sec_CYgeometry}, hence we will be brief. First, let us consider the bundle $\Pc$ defined by the embedding of pullbacks of tautological bundles on $F(1, k+2, V_{2k+2})$:

\begin{equation}
    0\arw u^*\Uc_1\arw t^*\Uc_{k+2}\arw \Pc\arw 0
\end{equation}

where $u:F(1,k+2,2k+2)\arw G(1,2k+2)$ and $t: F(1,k+2,2k+2)\arw G(k+2,2k+2)$. It follows that $\Pc$ has rank $k+1$ and determinant $\det(\Pc)=\Oc(1, -1)$. Moreover, by the tautological sequences of both $ G(1,2k+2)$ and $G(k+2,2k+2)$ and the snake lemma, one finds also the following sequence:

\begin{equation}\label{eq_sequencewithquotientsGLSM}
    0\arw \Pc\arw u^*\Qc_1\arw t^*\Qc_{k+2}\arw  0.
\end{equation}

where $\Qc_1$ and $\Qc_{k+2}$ are the quotients of $V_{2k+2}\otimes\Oc$ by respectively $\Uc_1$ and $\Uc_{k+2}$. This sequence describes the natural embedding of $t^*\Uc_{k+2}/u^*\Uc_1$ inside $u^*\Qc_1$. Observe that restricting the sequence to a fiber $u^{-1}([v])$ we recover the pullback to $F(2,3,V_{2k+2}/\operatorname{Span}(v))$ of the tautological sequence of $G(2,V_{2k+2}/\operatorname{Span}(v))$.\\
\\
Let us now consider the following description of $F(1,k+2,V_{2k+2})$:

\begin{equation}\label{eq_GITF14}
   F(1,k+2, V_{2k+2})\simeq\quotient{\Hom(\CC^{k+2}, V_{2k+2})\setminus D_{k+1}}{G}
\end{equation}

where

\begin{equation}
    G = \left\{\left(
        \begin{array}{cc}
            \lambda & a\\
            0  & h
        \end{array}
        \right)\right\}\subset GL(k+2), \,\,\,\,\, \lambda\in\CC^*, \,\,\, a\in\CC,\,\,\,\,\,h\in GL(k+1).
\end{equation}

and $D_k:=\{C\in\Hom(\CC^{k+2}, V_{2k+2}) : \rk(C)<k+2\}$. The quotient is taken with respect to the right $G$-action defines as $C\sim C g^{-1}$. Given a $k+1$-dimensional vector space $V_{k+1}$, we can describe $\Pc(1,2)$ as a $G$-equivariant vector bundle over $F(1,k+2,V_{2k+2})$ in the following way:
\begin{equation}\label{eq_G}
    \begin{tikzcd}[row sep = huge]
    \Pc(1,2)\ar{d} =   \displaystyle\quotient{\Hom(\CC^{k+2}, V_{2k+2})\setminus D_{k+1}\oplus V_{k+1}}{G} \\
    F(1, k+2,V_{2k+2}) 
    \end{tikzcd}
\end{equation}
where the equivalence relation on $\Hom(\CC^{k+2}, V_{2k+2})\setminus D_{k+1}\times V_{k+1}$ is $(C, x)\sim (Cg^{-1}, \lambda^{-3}\det h^{-2}hx)$. In fact, since $\Oc(1, 0)=t^*\Uc^\vee$ and $\Oc(0, 1)=u^*\det\Uc^\vee$, the weight of $\Oc(0,1)$ under its associated one-dimensional representation is $\det g^{-1}=\lambda^{-1}\det h^{-1}$.\\
\\
Following the same approach adopted for the GLSM of a Calabi--Yau pair of type $A^G_{2k}$ in \cite{kr}, a section $s$ of $\Pc(1,2)$ is defined by an equivariant map $\hat s: \Hom(\CC^{k+2}, V_{2k+2})\arw\CC^{k+1}$ fulfilling the equivariancy condition $s([C])=[C, \hat s(C)]$. Therefore it must satisfy
\begin{equation}
    \hat s(C g^{-1}) = \lambda^{-1}\det g^{-2}h\hat s(C).
\end{equation}

In  order to characterize $\hat s$, given a point $[C]\in F(1, k+2, V_{2k+2})$, let us pick a representative $C$ and rename as $v$ the first column of $C$, call $B$ the rest of the matrix. We use the notation $C = (v|B)$ for juxtaposition. Then, observe that on the fiber over $(v|B)$ the first map of \ref{eq_sequencewithquotientsGLSM} embeds in $V_{2k+2}/\text{Span}(v)$ the image of $B$. In addition, the function $(v|B)\arw B\hat s((v|B))$ transforms like the fiber of $V_{2k+2}\otimes\Oc(1,2)$ under the $G$-action. Since its image lies in the image of $B$, by the maximal rank condition on $(v|B)$ it must lie in $V_{2k+2}/\text{Span}(v)$, which is the fiber of $t^*\Qc$ over $v$.

\subsection{Preliminary notions}
While a gauged linear sigma model is a precise and well-defined entity in theoretical physics, there is no unique mathematical definition of such object. We will use the following:

\begin{definition}\label{def_glsm}
    We call \emph{gauged linear sigma model} $(V, G, \CC^*_R, W)$ the following data:
    \begin{enumerate}
        \item A finite dimensional vector space $V$
        \item A linear reductive group $G$ with an action on $V$
        \item A \emph{$R$-symmetry}, which is an action of $\CC^*$ on $V$, traditionally denoted by $\CC^*_R$
        \item A polynomial $W:V\arw\CC$ called \emph{superpotential}.
    \end{enumerate}
    Moreover, we require the following conditions to hold:
    \begin{enumerate}
        \item The $G$-action and the $\CC^*_R$-action commute on $V$
        \item $W$ is $G$-invariant and $\CC^*_R$-homogeneous with positive weight
    \end{enumerate}
\end{definition}

\begin{theorem} \cite[Theorem 3.3]{halic}
    Let $X$ be a normal, affine $G$-variety. The GIT-equivalence classes in $\Hom(G, \CC^*)$ corresponding to the $G$-action on $X$ are the relative interiors of the cones of a rational, polyhedral fan $\Delta^G(X)$. Such regions will be denoted \emph{chambers}.
\end{theorem}

\begin{definition}\label{def_glsmphase}
    Let $(V, G, \CC^*_R, W)$ be a GLSM. We call \emph{phase} of the GLSM a chamber of the associated polyhedral fan.
\end{definition}

\begin{definition}
    Let $(V, G, \CC^*_R, W, I)$ be a GLSM phase, where $I$ is the associated chamber. We call \emph{critical locus} of the superpotential:
    \begin{equation}
        \operatorname{Crit}(W):=Z(dW)
    \end{equation}
    where $dW$ is the gradient of $W$. Moreover, we call \emph{vacuum manifold} the GIT quotient:
    \begin{equation}
        Y_I = \operatorname{Crit}(W)\git_\tau G. 
    \end{equation}
    for any $\rho_\tau\in I$.
\end{definition}

\subsection{The model}
We are now ready to present the GLSM. Let us call $V$ the vector space
\begin{equation}
    V = \Hom(\CC, V_{2k+2})\oplus\Hom(\CC^{k+1}, V_{2k+2})\oplus V_{k+1}^\vee
\end{equation}

endowed with the following $G$-action:

\begin{equation}\label{eq_dualaction}
    \begin{tikzcd}[row sep = tiny, column sep = huge, /tikz/column 1/.append style={anchor=base east} ,/tikz/column 2/.append style={anchor=base west}] 
    G\times V\ar{r} & V \\
    g, (v, B, x)\ar{r} & (v\lambda^{-1}, Bh^{-1}, \lambda^3\det h^2 x h^{-1})
    \end{tikzcd}
\end{equation}

where $g$ decomposes as in Equation \ref{eq_G}. Given a smooth section $s\in H^0(F, 1,k+2,V_{2k+2}), \Pc(1,2))$ we fix our superpotential as the following $G$-invariant function:

\begin{equation}\label{eq_superpotential}
    \begin{tikzcd}[row sep = tiny, column sep = large, /tikz/column 1/.append style={anchor=base east} ,/tikz/column 2/.append style={anchor=base west}]
    V\ar{r}{w} & \CC \\
    (v, B, x)\ar[maps to]{r} & x\cdot \hat s(v, B)
    \end{tikzcd}
\end{equation}

where the dot is the usual contraction $V_{k+1}^\vee\times V_{k+1}\arw\CC$.\\
We define a family of characters

\begin{equation}
\begin{tikzcd}[row sep = tiny, column sep = huge, /tikz/column 1/.append style={anchor=base east} ,/tikz/column 2/.append style={anchor=base west}]
    G \ar{r}{\rho_\tau} &\CC^* \\
    g\ar[maps to]{r}&\lambda^{-\tau}\det h^{-\tau}
\end{tikzcd}
\end{equation}

which describes a line inside the character group. We consider the variation of GIT related to the crossing of the vertex between the loci $\tau>0$ and $\tau<0$. In order to do this we observe that \cite[Lemma 2.4]{king} and \cite[Theorem 3.3]{halic} hold, hence, fixed one of the two chambers, we investigate the locus $Z_\pm\in V$ of triples $(v,B,x)$ such that there exists a one-parameter subgroup $\{g_t\}\subset G$ with $\rho_\pm^{-1}(g_t)\arw 0$ and $g_t(v,B,x)$ has a limit in $V$ for $t\arw 0$. Then, the corresponding semistable locus is $V^{ss}_\pm=V\setminus Z_\pm$.\\
\\
Let us fix a one-parameter subgroup in $G$ depending on $k+2$ parameters $\alpha_0,\dots, \alpha_{k+1}$ whose elements are
\begin{equation}
    g_t = \left(
    \begin{array}{ccc}
        t^{\alpha_0} &  & \\
         & \ddots & \\
         & & t^{\alpha_{k+1}}
    \end{array}
    \right)
\end{equation}
    
\subsubsection{The chamber $\tau>0$}

Here the condition $\rho_+^{-1}(g_t)\arw 0$ translates to $\sum_{i=0}^{k+1} \alpha_i > 0$. Then $(v,B,x)\in Z_+$ if and only if there exist a tuple $\alpha_0, \dots \alpha_{k+1}$ satisfying a set of inequalities which for the general $(v,B,x)$ are:
\begin{equation}\label{eq_inequalitiesfibrationone}
    \left\{
    \begin{array}{rl}
        \sum_{i=0}^{k+1} \alpha_i & >0  \\
         -\alpha_j& \geq 0 \\
         3 \alpha_0 + 2\sum_{i=1}^{k+1}\alpha_i - \alpha_j& \geq 0 \\
    \end{array}
    \right.
\end{equation}
Such inequalities admit no common solution, which translates to the fact the general point $(v, B, x)$ is semistable. One observes that the locus $\Delta_h:=\{(v, B, x) : v = 0, rk B = k+1-h\}$ is fixed by $G$ for every $h\leq k+1$ and that $\Delta_h$ contains an element $(v, B, x)$ where the first $h$ columns of $B$ are composed entirely of zeroes. On such element the inequalities \ref{eq_inequalitiesfibrationone} become:
\begin{equation}
    \left\{
    \begin{array}{rl}
        \sum_{i=0}^{k+1} \alpha_i & >0  \\
         -\alpha_j& \geq 0 \hspace{10pt} j>l\\
         3 \alpha_0 + 2\sum_{i=1}^{k+1}\alpha_i - \alpha_j& \geq 0 \\
    \end{array}
    \right.
\end{equation}
which admit solution if and only if $l>0$. On the other hand, if $\rk B = k+1$ the only way to produce an unstable point is to set $v=0$, hence the unstable locus is given by $\{\rk B < k+1\}\cup\{ v = 0\}$. Summing all up, we get:
\begin{equation}
    V^{ss}_+ = \{(v,B,x)\in  V | \rk v=1, \rk B = k+1\}.
\end{equation}
Therefore, since
\begin{equation}
    V\git_+G = V^{ss}_+/G = \Pc^\vee(-1, -2).
\end{equation}
we conclude that $\operatorname{Crit}(w)\git_+G\simeq X$ by \cite[Remark 1.2]{okonekteleman}.

\subsubsection{The chamber $\tau<0$}
Here the condition $\rho_-^{-1}(g_n)\arw 0$ gives the inequality to $\sum_{i=0}^{k+1} \alpha_i < 0$. The other inequalities are unchanged, but the solution is radically different:
\begin{equation}
    V^{ss}_- = \{(v,B,x)\in  V | \rk v=1, \rk x=1,  \ker B\cap\ker x = \{0\} \}.
\end{equation}
Acting with $G$ we can reduce to the situation where $x=(1,0,\dots, 0)$. Then the stabilizer has the form

\begin{equation}\label{eq_stabilizerfiberedglsm}
    G_S = \left\{
    g\in G : g = \left(
    \begin{array}{ccccc}
        \lambda & z_{k+1} & z_{k+2} &\dots &z_{2k+1} \\
        0 & \delta & 0 &\dots & 0 \\
        0 & z_1 & m_{11} & & m_{1k}\\
        \vdots & \vdots&\vdots& & \vdots\\
        0 & z_k & m_{k1}&\dots & m_{kk}
    \end{array}
    \right)
    \right\}
\end{equation}

We observe that the action of the stabilizer on $B$ preserves linear combinations of the last $k$ columns, while the first one transforms like the image of the fiber of $t^*\Qc(-1,-2)$. Hence, the GIT quotient is 
\begin{equation}
    V\git_-G = V^{ss}_+/G = r^*\Qc^\vee(-1, -2).
\end{equation}
\subsection{The phase transition}
In order to prove that the critical locus in the second phase is isomorphic to $X_2$, we need to describe the section $s$ more explicitly. First let us describe $S\in H^0(F(1,k+1,k+2, V_{2k+2}), \Oc(1,1,1))$. In analogy with Equation \ref{eq_GITF14}, the flag variety $F(1,k+1,k+2,V_{2k+2})$ is given by the following GIT description:
\begin{equation}\label{eq_GITF134}
   F(1,k+1,k+2,V_{2k+2})\simeq(\Hom(\CC^{k+2}, V_{2k+2})\setminus Z)/H
\end{equation}
where
\begin{equation}
    H = \left\{\left(
        \begin{array}{ccc}
            \lambda & \times &\times\\
            0  & h & \times\\
            0 & 0 & \delta
        \end{array}
        \right)\right\}\subset GL(k+2), \,\,\,\,\, \lambda, \delta\in\CC^*, \,\,\, h\in GL(k).
\end{equation}
and the action is $C\simeq Cg^{-1}$ for every $g\in H$. The symbol $\times$ denotes entries corresponding to the nilpotent part, where there is no condition. Let us write $C=(v|A|u)\in\Hom(\CC^{k+2}, V_{2k+2})$ where $v, u\in \Hom(\CC, V_{2k+2})$ and $A\in\Hom(\CC^2, V_{2k+2} )$. Then, a section of $\Oc(1,1,1)$ acts in the following way:
\begin{equation}
        (v|A|u)\xrightarrow{\hspace{30pt}}\ S((v|A|u)) = S^{i j_1\dots j_{k+1} l_1\dots l_{k+2}}\,v_i\psi_{j_1\dots j_{k+1}}(v|A)\psi_{l_1\dots l_{k+2}}(v|A|u)
\end{equation}
where $\psi_{k_1,\dots k_r}$ is the minor obtained choosing the lines $k_1, \dots k_r$, hence the coordinates $\psi_{k_1,\dots k_r}$ defines a Pl\"ucker map to $\wedge^r V_{2k+2}$. As we did for Equation \ref{eq_explicitsection}, to unclutter the notation, we omit sums over repeated high and low indices, and we use square brackets to anti-symmetrize indices. We observe that
\begin{equation}\label{eq_section134}
    S(g.(v|A|u)) = \lambda^{-3}\det h^{-2} \delta^{-1} S((v|A|u))
\end{equation}
which is the correct equivariancy condition since $\Oc(1,1,1)\simeq\Oc(1)\boxtimes\Oc(1)\boxtimes\Oc(1)$. Then, the pushforwards of this section to $F(1,k+1,V_{2k+2})$ and $F(1,k+2,V_{2k+2})$ are described by the following equivariant functions: 
\begin{equation}\label{eq_section13}
        (v|A)\xrightarrow{\hspace{30pt}}\hat\sigma^r((v|A|u)) = S^{i j_1\dots j_{k+1} l_1\dots l_{k+2}}\,v_i\psi_{j_1\dots j_{k+1}}(v|A)\delta^r_{\,\,[l_{k+2}}\psi_{l_1\dots l_{k+1}]}(v|A)
\end{equation}
\begin{equation}\label{eq_section14}
        (v|B)\xrightarrow{\hspace{30pt}}\hat s^r((v|B)) = S^{i j_1\dots j_{k+1} l_1\dots l_{k+2}}\,v_i\left(\frac{\partial}{\partial B_r^{\,\,t}}\psi_{j_1\dots j_{k+1}t}(v|B)\right)\psi_{l_1\dots l_{k+2}}(v|B)
\end{equation}
where square brackets around a set of indices means totally skew-symmetric.
What is left to prove is that the quotient of ther critical locus of $w$ restricted to $V_-^{ss}$ by $G$ is isomorphic to $X_1$. Let us write the superpotential explicitly: by Equations \ref{eq_superpotential} and \ref{eq_section14} we have
\begin{equation}\label{eq_superpotentialexplicit}
        (v, B, x)\xrightarrow{\hspace{30pt}} x^r S^{i j_1\dots j_{k+1} l_1\dots l_{k+2}}\,v_i\left(\frac{\partial}{\partial B_r^{\,\,t}}\psi_{j_1\dots j_{k+1}t}(v|B)\right)\psi_{l_1\dots l_{k+2}}(v|B)
\end{equation}
As we showed before, for every $G_S$-orbit in $V_-^{ss}$ there exist a unique point such that $x=x_0:=(1,0,\dots,0)$. Let us work on such points. Define:
\begin{equation}
    \wt V = \{ (v, B) : \rk v = 1, B_r^{\,\,1}=0\,\,\forall r\leq {2k+2}\}.
\end{equation}
We are interested in the locus
\begin{equation}
    dw\cap\wt V =  \{ (v, B, x) : x=x_0, (v,B)\in\wt V, \hat s(v,B, x)=0, x\cdot d\hat s(v,B,x)=0 \}.
\end{equation}
If $(v, B)\in\wt V$ the first equation is automatically satisfied, since $\psi(v|B)$ is identically zero for lower rank matrices, and the first column of $B$ is zero. Let us now focus on the second equation defining the critical locus. By Equation \ref{eq_superpotentialexplicit}, restricted to $(\wt V, x_0)$ it becomes (up to sign):
\begin{equation}\label{eq_candidatequotientsection}
    \begin{split}
        x\cdot ds(v,B,x_0)^z|_{(v, B)\in\wt V} &= \left. S^{i j_1\dots j_{k+1} l_1\dots l_{k+2}}v_i\left(\frac{\partial}{\partial B_1^{\,\,t}}\psi_{j_1\dots j_{k+1}t}(v|B)\right)\left(\frac{\partial}{\partial B_1^{\,\,z}}\psi_{l_1\dots l_{k+2}}(v|B)\right)\right|_{(v, B)\in\wt V} \\
        &= S^{i j_1\dots j_{k+1} l_1\dots l_{k+2}}v_i\psi_{j_1\dots j_{k+1}}(v|\wt A)\delta^z_{\,\,[l_{k+2}}\psi_{l_1\dots l_{k+1}]}(v|\wt A) =: R^z(A).
    \end{split}
\end{equation}
where $\wt A\in\Hom(\CC^k, V_{2k+2})$ is the matrix resulting by removing the first (vanishing) column from $B$. This last equation coincides with \ref{eq_section13}, hence it describes the image in $H^0(F(1, k+1, V_{2k+2}), V_{2k+2}\otimes\Oc(1,2))$ of a section of $r^*\Qc^\vee(1,2)$ on $F(1,k+1,V_{2k+2})$. Summing all up, the critical locus of $w$ on $V^{ss}_-$ is a bundle over the zero locus of the $2k+2$ equations $x\cdot dw$. We observe that the $2k+2$ equations vanish exactly where the associated section of $r^*\Qc^\vee(1,2)$ vanish, hence the critical locus is a bundle over $X_1$. In fact, the contraction of the vector $R(A)$ with any column of $A$ is identically zero, hence $R(A)\in\ker(A^T)$, which allows us to conclude that $R$ is the image of a section $s$ of $r^*\Qc^\vee(1,2)$ by an injective morphism of vector bundles, hence $Z(R)\simeq Z(s)$. \\
\\
The last step is to observe that the action of the stabilizer $G_S$ described by Equation \ref{eq_dualaction} is transitive and free on $\{x=(x_1,\dots, x_{k+1})\}$. Hence, taking the quotient by $G_S$, we obtain $X_1$.\\
\\
If we choose $k=2$ we obtain a GLSM description of a pair of Calabi--Yau fibrations associated to the roof bundle of type $A^G_4$. Hence, we can state the following theorem summarizing all dualities appearing in this picture. 
\begin{theorem}\label{thm_main2_body}(Theorem \ref{thm_intro_4})
There exists a pair of derived equivalent Calabi--Yau eightfolds $X_1$, $X_2$ of Picard number two, and fibrations $f_1:X_1\arw \PP^5$ and $f_2:X_1\arw \PP^5$ such that for the general $b\in \PP^5$ the varieties $Y_1:= f_1^{-1}(b)$ and $Y_2:= f_1^{-1}(b)$ are non birational, derived equivalent Calabi--Yau threefolds. Moreover, $X_1$ and $X_2$ are isomorphic to the vacuum manifolds of two phases of a non abelian gauged linear sigma model.
\end{theorem}
\begin{proof}
Let us consider the roof bundle of type $A_4^G$ over $\PP^5$. By the discussion of Section \ref{sec_CY}, $X_1$ and $X_2$ are Calabi--Yau eightfolds. In particular, by Proposition \ref{prop_picardnumberCY}, they have Picard number two. Derived equivalence follows from Corollary \ref{cor_knowncases}. By the above, $X_1$ and $X_2$ are isomorphic to the $G$-quotients of the critical loci of $w$ in the two stability chambers $\tau<0$ and $\tau>0$. Finally, the fibers $Y_1:= f_1^{-1}(b)$ and $Y_2:= f_1^{-1}(b)$ over a general $b\in B$ are a Calabi--Yau pair associated to the roof of type $A_4^G$, hence they are non birational and derived equivalent by \cite[Theorems 3.6 and 5.7]{kr}.
\end{proof}

\appendix

\section{Derived equivalences of Calabi--Yau pairs}\label{appendix}
In this appendix we will review all known cases of roofs where a sequence of mutations to prove derived equivanelce of the asssociated Calabi--Yau pair is known. We will also check, case by case, that the following condition is satisfied, which is equivalent to say that the roof $G/P$ fulfills Assumptions \ref{itm:ass_a2} and \ref{itm:ass_a3}:

\begin{customthm}{($\dagger$)}\label{def_mutationsroof}
    Consider a mutation $(F_1, F_2)\arw(\LL_{F_1}F_2, F_1)$ of exceptional objects on $M$. We say that such mutation satisfies Condition \emph{\ref{cond_mutationsroof} } if there exist exceptional objects $E_1$, $E_2$ on $G/P$ such that for $i\in\{1;2\}$ one has $l^*E_i = F_i$ and the following vanishings hold:
    \[
        \Ext^\bullet_{G/P}(E_2, E_1) = \Ext^\bullet_{G/P}(E_1\otimes L, E_2)=0.\label{cond_mutationsroof}%\tag{$\dagger$}
    \]
    The same condition is defined for right mutations $(F_1, F_2)\arw(F_2, \RR_{F_2}F_1)$.
\end{customthm}

Condition \ref{cond_mutationsroof}  is crucial for the proof of Theorem \ref{thm_intro_3} and Theorem \ref{thm_intro_1}.

\subsection{Setup and general strategy}\label{sec_setupandgeneralstrategy}
Let us recall some notation. In the following, $G/P$ is a homogeneous roof of rank $r$ with projective bundle structures $h_i:\PP(\Ec_i)\arw G/P_i$. Let $L=h_1^*\Oc(1)\otimes h_2^*\Oc(1)$ be the Grothendieck line bundle of both the projective bundle structures and call $M\subset G/P$ the (smooth) zero locus of a general section of $L$. We will commit the abuse of notation of denoting by $L$ also the pullback of such line bundle to $M$. Let $(Y_1, Y_2)$ be the associated Calabi--Yau pair, i.e. $Y_1$ and $Y_2$ are zero loci of pushforwards of a section defining $M$ along the projective bundle maps (see Definition \ref{def_CYpairs}). One has the following diagram:

\begin{equation}\label{eq_diagramfiber}
    \begin{tikzcd}[row sep = huge, column sep = large]
           &T_1\ar[swap]{ddl}{\nu_1}\ar[hook]{r}{k_1}   & M\ar[hook,swap]{d}{l}\ar[hookleftarrow]{r}{k_2}\ar[swap]{ddl}{h_1|_M}\ar{ddr}{h_2|_M}  & T_2\ar{ddr}{\nu_2}  &   \\
           &   & G/P\ar{dl}{h_1}\ar[swap]{dr}{h_2}   &   &   \\
        Y_1\ar[hook]{r}{t_1} & G/P_1 &    & G/P_2\ar[hookleftarrow]{r}{t_2} &   Y_2    
    \end{tikzcd}
\end{equation}

where $T_i$ is the preimage of $Y_i$ under $h_i|_M$, and $\nu_i$ is the restriction of $h_i|_M$ to $T_i$ for $i\in\{1;2\}$.

\subsection{Derived equivalence for the roof of type \texorpdfstring{$C_2$}{something}}
Let $V_4$ be a vector space of dimension four. The roof of type $C_2$ is the symplectic flag variety $IF(1,2,V_4)$ with projective bundle structures respectively over $\PP(V_4)$ and $IG(2, V_4)$. Note that $IG(2,V_4)$ is a three-dimensional quadric in $\PP^4$. Both $h_1$ and $h_2$ are $\PP^1$-fibrations. By dimensional reasons and Lemma \ref{lem_CY}, the zero loci $Y_1 = Z(h_{1*}\sigma)$ and $Y_2 = Z(h_{2*}\sigma)$ are elliptic curves.\\
\\
Let us call $\Uc$ the pullback to $IF(1, 2, V_4)$ of the tautological vector bundle of $IG(2, V_4)$. By Cayley trick \cite[Proposition 2.10]{cayleytrick}, or equivalently by Orlov's blowup formula \cite[Theorem 4.3]{orlovblowup} and an application of the Serre functor of $M$, we write the following semiorthogonal decompositions:
\begin{equation}\label{eq_collectionsC2initial}
\begin{split}
        \dbcoh(M) \simeq & \langle \phi_1\dbcoh(Y_1), \Oc_M(-1,1),\Oc_M(0,1),\Oc_M(1,1),\Oc_M(2,1)\rangle \\
                  \simeq & \langle \phi_2\dbcoh(Y_2), \Oc_M(1,1), l^*\Uc^\vee(1,1), \Oc_M(1,2),\Oc_M(1,3)\rangle
 \end{split}
\end{equation}
where $\phi_i=k_{i*} \nu_i^*$ in the notation of Diagram \ref{eq_diagramfiber}. We formulate the following lemma:

\begin{lemma}\label{lem_derivedequivalenceC2}
In the setting above, there is a sequence of mutations realizing a derived equivalence $\dbcoh(Y_1)\arw\dbcoh(Y_2)$ and satisfying Condition \ref{cond_mutationsroof} .
\end{lemma}
\begin{proof}
First let us apply the Serre functor to the last four objects of each collection, obtaining:
\begin{equation}\label{eq_collectionsC2}
\begin{split}
        \dbcoh(M) \simeq & \langle \Oc_M(-2,0),\Oc_M(-1,0),\Oc_M,\Oc_M(1,0), \phi_1\dbcoh(Y_1)\rangle \\
                            \simeq & \langle \Oc_M,l^*\Uc^\vee,\Oc_M(0,1),\Oc_M(0,2), \phi_2\dbcoh(Y_2)\rangle
 \end{split}
\end{equation}
where we used the fact that $\omega_M = \Oc(-1, -1)$. For brevity, let us fix $F := IF(1,2,V_4)$. Our approach for finding the right mutations follows the approach of \cite{morimura} closely. Let us start from the first collection. We can send the first bundle to the far right, then move $\phi_1\dbcoh(Y_1)$ one step to the right, obtaining
\begin{equation}
    \begin{split}
            \dbcoh(M) \simeq & \langle \Oc_M(-1,0),\Oc_M,\Oc_M(1,0), \Oc_M(-1,1), \RR_{\Oc_M(-1,1)}\phi_1\dbcoh(Y_1)\rangle
    \end{split}
\end{equation}
We have the following short exact sequence on $F$ (and its pullback on $M$) \cite[Equation 2.2]{morimura}:
\begin{equation}
    0\arw\Oc(-1,1)\arw \Uc^\vee\arw\Oc(1,0)\arw 0
\end{equation}
All cohomology in the following is computed by Borel--Weil--Bott's theorem. %(Theorem \ref{thm_bottvectorbundles}).
First, since one has
\begin{equation}
    \Ext_F^\bullet(\Oc(1,0), \Oc(-1,1))=\Ext_M^\bullet(\Oc_M(1,0), \Oc_M(-1,1))=\CC[-1]
\end{equation}
we can mutate $\Oc_M(1,0)$ and get:
\begin{equation}
    \begin{split}
            \dbcoh(M) \simeq & \langle \Oc_M(-1,0),\Oc_M,\Oc_M(-1,1),l^* \Uc^\vee, \RR_{\Oc_M(-1,1)}\phi_1\dbcoh(Y_1)\rangle
    \end{split}
\end{equation}
and we compute the following vanishings:
\begin{equation}
    \Ext_F^\bullet(\Oc(2,1), \Oc(-1,1)) = \Ext_F^\bullet(\Oc(-1,1), \Oc(1,0))=0
\end{equation}
which are required to fulfill Condition \ref{cond_mutationsroof} . The next step is to exchange the second and the third bundles. We have:
\begin{equation}
    \Ext_F^\bullet(\Oc, \Oc(-1,1))=\Ext_M^\bullet(\Oc_M, \Oc_M(-1,1))=0
\end{equation}
hence we can move the first two to the end and send $\RR_{\Oc_M(-1,1)}\phi_1\dbcoh(Y_1)$ to the far right. Again, this mutation fulfills Condition \ref{cond_mutationsroof}  since:
\begin{equation}
    \Ext_F^\bullet(\Oc(-1, 1), \Oc) = \Ext_F^\bullet(\Oc(1,1), \Oc(-1,1))=0.
\end{equation}
We find:
\begin{equation}\label{eq_collectionC2final}
    \begin{split}
            \dbcoh(M) \simeq & \langle\Oc_M,l^*\Uc^\vee, \Oc_M(0,1),\Oc_M(0,2),\\ &\hspace{6pt}\RR_{\Oc_M(0,2)}\RR_{\Oc_M(0,1)}\RR_{\Oc_M(-1,1)}\phi_1\dbcoh(Y_1)\rangle
    \end{split}
\end{equation}
In the first four bundles we recognise $\dbcoh(IG(2, V_4))$. Hence, comparing Equation \ref{eq_collectionsC2} with Equation \ref{eq_collectionC2final} we prove our claim.
\end{proof}

\begin{remark}
    Note that the derived equivalence $\dbcoh(Y_1)\simeq\dbcoh(Y_2)$ is a consequence of the derived equivalence of local Calabi--Yau fivefolds described in \cite{morimura}: in fact, one can follow the approach of \cite{uedaflop} based on matrix factorization categories. In general, given a roof of type $G/P$ with $\PP^{r-1}$-bundle structures $h_i:G/P\arw G/P_i$, let us call $\Ec_i:= h_{i*}\Oc(1,1)$ and $Y_i=Z(h_{i*}\sigma)$, where $\sigma$ is a general section of $\Oc(1,1)$. Then, one can define by the data of a section of $\Ec_i$ a superpotential $w_i$ such that the derived category of matrix factorizations of the Landau--Ginzburg model $(\Ec_i^\vee, w_i)$ is equivalent to $\dbcoh(Y_i)$ via Kn\"orrer periodicity (for more details, see Chapter 14). Then, by \cite{uedaflop} if there is a derived equivalence $\dbcoh(\operatorname{Tot }\Ec_1^\vee)\simeq \dbcoh(\operatorname{Tot }\Ec_2^\vee)$ satisfying a $\CC^*$-equivariancy condition, it lifts to a derived equivalence of the matrix factorization categories of $(\Ec_i^\vee, w_i)$, and $\dbcoh(Y_1)\simeq\dbcoh(Y_2)$ follows from this last equivalence composed with Kn\"orrer periodicity. This gives a derived equivalence for Calabi--Yau pairs of type $A^G_4$, $C_2$ \cite{morimura} and $G_2$ \cite{uedaflop}.
\end{remark}

\subsection{Derived equivalence for the roof of type \texorpdfstring{$A^M_{n}$}{something}}
Let $V$ be  a vector space of dimension $n+1$. A roof of type $A^M_{n}$ is given by the following diagram:
\begin{equation}\label{eq_mukairoofbundlefiber}
    \begin{tikzcd}[row sep=huge, column sep = small]
        &F(1,n,V)\ar[swap]{dl}{h_1}\ar{dr}{h_2} &  \\
        G(1,V) & & G(n, V)
    \end{tikzcd}
\end{equation}
The zero loci $Y_i=Z(h_{i*}\sigma)$ are zero-dimensional. Nonetheless we discuss 
mutations, since it is necessary to prove that the roof of type $A_n^M$ satisfies Assumptions \ref{itm:ass_a2} and \ref{itm:ass_a3}. A similar result, with a different proof, has been previously found in the context of Mukai flops by \cite{kawamatadk, namikawa}, later \cite{morimura} developed a sequence of mutations to achieve the same result. The proof we will describe is similar to the one by \cite{morimura}.\\ 
\\
By Cayley trick we recover the following semiorthogonal decompositions:
\begin{equation}\label{eq_cayleytrickmukaifiber}
    \small
    \begin{split}
    \dbcoh(M)\simeq& \langle \phi_1 \dbcoh(Y_1), h_1|_M^*\dbcoh(G(1,V))\otimes  L,\dots, h_1|_M^*\dbcoh(G (1,V))\otimes  L^{\otimes(n-1)}\rangle \\
    \simeq & \langle \phi_2 \dbcoh(Y_2), h_2|_M^*\dbcoh(G(n,V))\otimes  L,\dots, h_2|_M^*\dbcoh(G (n,V))\otimes  L^{\otimes(n-1)}\rangle
    \end{split}
\end{equation}
where $\phi_i := l_*\nu_i^*$. By choosing the right twists of Beilinson's full exceptional collection for $\PP^n$ \cite{beilinson} we write:
\begin{equation}\label{eq_collectionsAm}
    \begin{split}
        \dbcoh(M)\simeq\langle \phi_1 \dbcoh(Y_1), 
                   &\Oc_M(1,1),\dots\dots\dots\dots\dots,\Oc_M(n+1,1),\\
                   &\hspace{15pt}\vdots\hspace{150pt}\vdots \\
                   &\Oc_M(n-1,n-1),\dots,\Oc_M(2n-1,n-1)\rangle\\
             \simeq\langle \phi_2 \dbcoh(Y_2), 
                   &\Oc_M(1,1-n),\dots\dots\dots\dots\dots,\Oc_M(1,1),\\
                   &\hspace{15pt}\vdots\hspace{150pt}\vdots \\
                   &\Oc_M(n-1,-1),\dots,\dots\Oc_M(n-1, n-1)\rangle
    \end{split}
\end{equation}
First, we need the following vanishing results:

\begin{lemma}\label{lem_mukaivanishing}
    For $1<m<n+1$ and $0<t<m$ one has:
    \begin{equation}
        \Ext^\bullet_M(\Oc_M(m+1, 1), \Oc_M(2+t,2+t)) = \Ext^\bullet_{F(1,n, V)}(\Oc(m+1, 1), \Oc(2+t,2+t)) = 0.
    \end{equation}
    
\end{lemma}
\begin{proof}
    We need to compute the cohomology of $\Oc(1+t-m, 1+t)$ and $\Oc_M(1+t-m, 1+t)$. Twisting the Koszul resolution for $M$ yields:
    \begin{equation}
        0\arw\Oc(t-m, t)\arw\Oc(1+t-m, 1+t)\arw\Oc_M(1+t-m,1+t)\arw 0
    \end{equation}
    Let us first compute the cohomology of $\Oc(t-m, t)$. For $t=0$ this is clearly acyclic, while for $t\neq0$ we consider the stalk of the pushforward: observe that $\Oc(a, b)$ is flat over $G(n, V)$ for every $a$ and $b$ \cite[Proposition III.9.2]{hartshorne}, and for every $l$ one has that $\dim H^l(h_2^{-1}(x),\Oc(a)|_{h_2^{-1}(x)})$ does not depend on $x$. Hence, by \cite[Page 50, Corollary 2]{mumfordabelian} for every $x\in G(n,V)$:
    \begin{equation}\label{eq_mukaipf}
        R^\bullet h_{2*}\Oc(t-m, t)_x \simeq H^\bullet(h_2^{-1}(x),\Oc(t-m)|_{h_2^{-1}(x)})
    \end{equation}
    where $h_2^{-1}(x)\simeq\PP^{n-1}$. This is identically zero for $0<t<m<n+1$. and by the Leray spectral sequence $\Oc(t-m, t)$ is acyclic. Similarly, the vanishing for the middle term  of the exact sequence follows by the totally analogous computation 
    \begin{equation}
        R^\bullet h_{2*}\Oc(1+t-m, 1+t)_x \simeq H^\bullet(h_2^{-1}(x),\Oc(1+t-m)|_{h_2^{-1}(x)})
    \end{equation}
\end{proof}

\begin{lemma}\label{lem_mukaivanishingsecondclaim}
    For $1<m<n+1$ and $1+t<m$ the following holds:
    \begin{equation}
        \begin{split}
            \Ext^\bullet_{F(1,n, V)}(\Oc(2+t, 2+t), \Oc(m+1, 1)) & = 0\\
            \Ext^\bullet_{F(1,n, V)}(\Oc(m+2, 2), \Oc(2+t,2+t)) & = 0.
        \end{split}
        \end{equation}
\end{lemma}
\begin{proof}
    The second claim is equivalent to $\Oc(t-m, t)$ being acyclic, already proven in the previous lemma. The first claim is a consequence of the semiorthogonality of
    \begin{equation}
        \begin{split}
            \dbcoh(F(1,n,V))\simeq 
                &\langle\Oc_M(0,0),\dots\dots\dots\dots\dots,\Oc_M(n,0),\\
                &\hspace{15pt}\vdots\hspace{130pt}\vdots \\
                &\Oc_M(n-1,n-1),\dots,\Oc_M(2n-1,n-1)\rangle.
        \end{split}
    \end{equation}
    Alternatively, one can formulate a direct computation as in the proof of Lemma \ref{lem_mukaivanishing}.
    %By the same argument of Equation \ref{eq_mukaipf} one has an isomorphism $R^\bullet h_{2*}\Oc(m-1,-1)_x \simeq H^\bullet(h_1^{-1}(x),\Oc(-1)|_{h_1^{-1}(x)})$, hence we get $\Ext^\bullet_{F(1,n, V)}(\Oc(1, 1), \Oc(m, 0)) = 0$. The second vanishing is a consequence of the fact that $\Oc(-m)$ is acyclic on $\PP^n$ for $1\leq m\leq n$. Alternatively, one can prove both the statement by the Borel--Weil--Bott theorem.
\end{proof}

\begin{lemma}\label{lem_mukaiderivedequivalencefibers}
    In the setting above, there is a sequence of mutations of exceptional objects of $\dbcoh(M)$ satisfying Condition \ref{cond_mutationsroof}  and realizing a derived equivalence $\dbcoh(Y_1)\simeq\dbcoh(Y_2)$.
\end{lemma}
\begin{proof}
    Let us switch to a more compact notation: hereafter $\Oc_{a,b}:=\Oc_M(a,b)$. Hence, Equation \ref{eq_cayleytrickmukaifiber} becomes:
    \begin{equation}
        \begin{split}
        \dbcoh(M)\simeq \langle \phi_1\dbcoh(Y_1),
                       &\Oc_{1,1},\dots\dots\dots\dots\dots,\Oc_{n+1,1},\\
                       &\Oc_{2,2},\dots\dots\dots\dots\dots,\Oc_{n+2,2},\\
                       &\hspace{15pt}\vdots\hspace{100pt}\vdots \\
                       &\Oc_{n-1,n-1},\dots\dots\dots,\Oc_{2n-1,n-1}\rangle.
        \end{split}
    \end{equation}
    First, let us move $\phi_1\dbcoh(Y_1)$ one step to the right, then let us send $\Oc_{1,1}$ to the end of the collection. We get:
    \begin{equation}
        \begin{split}
        \dbcoh(M)\simeq \langle \psi_1\phi_1\dbcoh(Y_1), 
                       &\Oc_{2,1}\dots\dots\dots\dots\dots\dots,\Oc_{n+1,1},\Oc_{2,2}\\
                       &\Oc_{3,2},\dots\dots\dots\dots\dots\dots,\Oc_{n+2,2},\Oc_{3,3}\\
                       & \hspace{15pt}\vdots\hspace{130pt}\vdots \\
                       &  \Oc_{n,n-1},\dots\dots\dots\dots,\Oc_{2n-1,n-1},\Oc_{n,n}\rangle
        \end{split}
    \end{equation}
    where $\psi_1:=\RR_{\Oc_{1,1}}$ and we used the fact that $\omega_M  =\Oc(-n+1, -n+1)$. By Lemma \ref{lem_mukaivanishing} we can move $\Oc_{2,2}$ leftwards until it stops at the right of $\Oc_{2,1}$, since it is orthogonal to all the bundles in between. These mutations, by Lemma \ref{lem_mukaivanishingsecondclaim}, satisfy Condition \ref{cond_mutationsroof} . We can repeat the same step on each row obtaining:
    \begin{equation}
        \begin{split}
        \dbcoh(M)\simeq \langle \psi_1 \phi_1\dbcoh(Y_1), 
                       &\Oc_{2,1},\Oc_{2,2},\Oc_{3,1},\dots\dots\dots\dots\dots\dots,\Oc_{n+1,1},\\
                       & \Oc_{3,2},\Oc_{3,3},\Oc_{4,2},\dots\dots\dots\dots\dots\dots,\Oc_{n+2,2},\\
                       & \hspace{15pt}\vdots\hspace{160pt}\vdots \\
                       & \Oc_{n,n-1},\Oc_{n+1,n-1},\Oc_{n+2,n-1},\dots\dots,\Oc_{2n-1,n-1}\rangle.
        \end{split}
    \end{equation}
    Now we mutate $\psi_1\phi_1\dbcoh(Y_1)$ two steps to the right, then we move the first two bundles to the end of the collection obtaining:
    \begin{equation}
        \begin{split}
        \dbcoh(M)\simeq \langle \psi_2 \phi_1\dbcoh(Y_1)
                       &\Oc_{3,1},\dots\dots\dots\dots\dots\dots,\Oc_{n+1,1},\Oc_{3,2},\Oc_{3,3},\\
                       &\Oc_{4,2},\dots\dots\dots\dots\dots\dots,\Oc_{n+2,2},\Oc_{4,3}, \Oc_{4,4},\\
                       & \hspace{15pt}\vdots\hspace{160pt}\vdots \\
                       &\Oc_{n+2,n-1},\dots\dots\dots,\Oc_{2n-1,n-1}, \Oc_{n+1, n}, \Oc_{n+1,n+1}\rangle
        \end{split}
    \end{equation}
    where $\psi_2 = \RR_{\langle\Oc_{2,1},\Oc_{2,2}\rangle}\circ\psi_1$. Then, on each row, using Lemma \ref{lem_mukaivanishing} we shift the last two bundles all the way left, until they stop at the right of the first bundle. Again, by Lemma \ref{lem_mukaivanishingsecondclaim} such mutations satisfy Condition \ref{cond_mutationsroof} . We find:
    \begin{equation}
        \small
        \begin{split}
        \dbcoh(M)\simeq \langle \psi_2 \phi_1\dbcoh(Y_1),
                       &\Oc_{3,1},\Oc_{3,2},\Oc_{3,3},\Oc_{4,1},\dots\dots\dots\dots\dots\dots\dots,\Oc_{n+1,1},\\
                       &\Oc_{4,2},\Oc_{4,3},\Oc_{4,4},\Oc_{5,2},\dots\dots\dots\dots\dots\dots\dots,\Oc_{n+2,2},\\
                       & \hspace{15pt}\vdots\hspace{170pt}\vdots \\
                       & \Oc_{n+1,n-1},\Oc_{n+1,n},\Oc_{n+1,n+1},\Oc_{n+2,n-1}, \dots,\Oc_{2n-1,n-1}\rangle.
        \end{split}
    \end{equation}
    This process can be iterated moving the first three bundles to the end, then on each row sending the last three bundles to the right of the first one, and repeating these steps increasing by one the number of bundles we move. We stop once we get a semiorthogonal decomposition given by $n-1$ twists of  $\langle\Oc_{1,1},\dots\Oc_{1,n+1}\rangle$ and the image of $\phi_1\dbcoh(Y_1)$ under a composition of mutations. This eventually happens after $n$ steps. We get the following collection:
    \begin{equation}
        \begin{split}
        \dbcoh(M)\simeq \langle \psi_n \phi_1\dbcoh(Y_1),
                       &\Oc_{n+1,1},\Oc_{n+1,2},\dots\dots\dots\dots,\Oc_{n+1,n+1},\\
                       &\Oc_{n+2,2},\Oc_{n+2,3},\dots\dots\dots\dots,\Oc_{n+2,n+2},\\
                       & \hspace{15pt}\vdots\hspace{130pt}\vdots \\
                       &\Oc_{2n-1,n-1},\Oc_{2n-1,n},\dots\dots,\Oc_{2n-1,2n-1}\rangle
        \end{split}
    \end{equation}
    If we twist the whole collection by $\Oc_{-n, -n}$ we obtain:
    \begin{equation}\label{eq_collectionsAmfinal}
        \begin{split}
        \dbcoh(M)\simeq \langle \Tc_{-n,-n}\circ\psi_n \phi_1\dbcoh(Y_1),
                       &\Oc_{1,1-n},\Oc_{1,2-n},\dots\dots\dots\dots\dots,\Oc_{1,1},\\
                       &\Oc_{2,2-n},\Oc_{2,3-n},\dots\dots\dots\dots\dots,\Oc_{2,2},\\
                       & \hspace{15pt}\vdots\hspace{130pt}\vdots \\
                       &\Oc_{n-1,-1},\Oc_{n-1,0},\dots\dots\dots,\Oc_{n-1,n-1}\rangle
        \end{split}
    \end{equation}
    where $\Tc_{-n,-n}$ is the twist functor. The proof is concluded by comparing Equation \ref{eq_collectionsAm} with Equation \ref{eq_collectionsAmfinal}.
\end{proof}

\subsection{Derived equivalence for the roof of type \texorpdfstring{$A_n\times A_n$}{something}}\label{sec_dequivalenceAnxAn}

Given a vector space $V_{n+1}$ of dimension $n+1$, the roof of type $A_n\times A_n$ is $\PP^n\times\PP^n$, with two (trivial) projective bundle structures over $\PP^n$.
Here the zero locus of a general section of $L$ is isomorphic to $F(1,n, V_{n+1})$. In the analogous context of standard flops, a proof of derived equivalence by mutations has been found by \cite{bondalorlovflop} and later by \cite{addingtondonovanmeachan, morimura}. We will present a similar sequence of mutations adapted to the present setting.\\
\\
Instead of using the Cayley trick as in the previous sections, we construct two semiorthogonal decompositions for $M$ using the fact that it admits two projective bundle structures ($M$ is itself a roof!) obtaining:
\begin{equation}\label{eq_AnxAnleftcollection}
    \begin{split}
        \dbcoh(M) = \langle & \Oc(0,0)\dots\dots\dots\dots\dots\dots\Oc(n,0)\\
                            & \hspace{20pt}\vdots\hspace{100pt}\vdots\\
                            & \Oc(n-1,n-1)\dots\dots\dots\Oc(2n-1,n-1)\rangle\\
    \end{split}
\end{equation}
\begin{equation}\label{eq_AnxAnrightcollection}
    \begin{split}
        \dbcoh(M) = \langle & \Oc(0,-n)\dots\dots\dots\dots\dots\Oc(0,0)\\
                            & \hspace{20pt}\vdots\hspace{100pt}\vdots\\
                            & \Oc(n-1,-1)\dots\dots\dots\Oc(n-1,n-1)\rangle
    \end{split}
\end{equation}
As we observed in Remark \ref{rem_AnxAnemptyloci}, there is no Calabi--Yau pair associated to this roof since the zero loci of $h_{1*}s$ and $h_{2*}s$ are empty, which makes the problem of derived equivalence somewhat trivial. Nonetheless, the existence of a sequence of mutations transforming the collection \ref{eq_AnxAnleftcollection} into \ref{eq_AnxAnrightcollection} is still useful for applying Theorem \ref{thm_intro_3}, therefore we formulate the following result:
\begin{lemma}\label{lem_derivedequivalenceAnxAn}
    In the setting introduced above, there is a sequence of mutations transforming the collection \ref{eq_AnxAnleftcollection} into \ref{eq_AnxAnrightcollection}, and each of these mutations satisfies Condition \ref{cond_mutationsroof} .
\end{lemma}
\begin{proof}
    The approach is nearly identical to the one we used to prove Lemma \ref{lem_mukaiderivedequivalencefibers}, which is not a surprise since the collections are very similar. In fact, \ref{eq_collectionsAm} can be formally obtained by removing one row from \ref{eq_AnxAnleftcollection} and adding the Calabi--Yau subcategory, therefore we can apply essentially the same argument.\\
    \\
    The main part of the process starts from the collection \ref{eq_AnxAnleftcollection} and it can be described by $n$ steps where the $k^{th}$ step consists in:
    \begin{itemize}
        \item[$\circ$] sending the first block of $k$ bundles to the end by means of the inverse Serre functor, they get twisted by $\Oc(n, n)$
        \item[$\circ$] Starting with the second line and proceeding down to the last, move then first $k$ bundles to the end of the previous line 
        \item[$\circ$] the last block of $k$ bundles of each line is orthogonal to the $n-k$ bundles at its left by Lemma \ref{lem_mukaivanishing}, hence it can be moved $n-k$ steps to the left.
    \end{itemize}
    After $n$ steps we obtain the following collection:
    \begin{equation}
        \begin{split}
            \dbcoh(M) = \langle & \Oc(n,0)\dots\dots\dots\dots\dots\dots\Oc(n,n)\\
                                & \hspace{20pt}\vdots\hspace{100pt}\vdots\\
                                & \Oc(2n-1,n-1)\dots\dots\dots\Oc(2n-1,2n-1)\rangle\\
        \end{split}
    \end{equation}
    and we can recover \ref{eq_AnxAnrightcollection} with a twist by $\Oc(-n, -n)$. Note that by Lemma \ref{lem_mukaivanishingsecondclaim} all mutations satisfy Condition \ref{cond_mutationsroof} , thus the assertion is proven.
\end{proof}

% Example of roofs of type A4: useless now that we have plenty of other worked examples!

\subsection{Derived equivalence for the roof of type \texorpdfstring{$A^G_{4}$}{something}}\label{sec_dequivalenceAg4}
Let $V_5$ be a vector space of dimension five. We recall the construction of the roof of type $A^G_4$:
\begin{equation}
    \begin{tikzcd}[row sep=huge, column sep = small]
        &F(2,3,V_5)\ar[swap]{dl}{h_1}\ar{dr}{h_2} &  \\
        G(2,V_5) & & G(3, V_5)
    \end{tikzcd}
\end{equation}
In the following we will call $\Uc_k$ and $\Qc_k$ respectively the tautological and the quotient bundle of $G(k,V_5)$.
We will use the minimal Lefschetz decomposition for $G(2,V_5)$ introduced in \cite{kuznetsovisotropiclines}:
\begin{equation}\label{kuznetsov25}
\small
D^b G(2,V_5) = \left<\mathcal O, \mathcal U_2^\vee,\mathcal O(1), \mathcal U_2^\vee (1),\mathcal O(2), \mathcal U_2^\vee (2),\mathcal O(3), \mathcal U_2^\vee (3),\mathcal O(4), \mathcal U_2^\vee (4)\right>
\end{equation}
The duality isomorphism between $G(2,V_5)$ and $G(3,V_5)$ exchanges $\mathcal U_2^\vee $ with $\mathcal Q_3$ and allows us to write a minimal Lefschetz exceptional collection for $G(3,V_5)$:

\begin{equation}\label{kuznetsov35}
\small
D^b G(3,V_5) = \left<\mathcal O, \mathcal Q_3,\mathcal O(1), \mathcal Q_3(1),\mathcal O(2), \mathcal Q_3(2),\mathcal O(3), \mathcal Q_3(3),\mathcal O(4), \mathcal Q_3(4)\right>.
\end{equation}

The content of this section has been developed in \cite{kr}, we will review it in order to check that Condition \ref{cond_mutationsroof} is fulfilled. For the sake of brevity, in the remainder of this section, we will omit pullbacks to $M$ while denoting exceptional objects of $\dbcoh(M)$, but we will always keep track of the variety where Exts are computed.

\begin{lemma}\label{vanishingQO}
    One has the following results:
    \begin{equation}
        \Ext_F^\bullet(\Qc_3(0,k), \Oc(1,1)) =\Ext_M^\bullet(\Qc_3(0,k), \Oc(1,1)) = \left\{ 
        \begin{array}{cc}
           \CC[-1]  & k = 2 \\
            0 & k\in\{1;3;4\}
        \end{array}
        \right.
    \end{equation}
    \begin{equation}
        \begin{split}
            \Ext_F^\bullet(\Oc(1,1),\Qc_3(0,k)) = 0 \hspace{10pt} & k\in\{1;2;3;4\}.
        \end{split}
    \end{equation}
\end{lemma}

\begin{proof}
    The first computation follows by the exact sequence on $F$:
    \begin{equation}
        0\arw\Qc_3^\vee(0,-k)\arw\Qc_3^\vee(1, 1-k)\arw\Qc^\vee_3|_M (1, 1-k)\arw 0.
    \end{equation}
    Both the first two bundles are irreducible on $F$ and their cohomology can be computed with the Borel--Weil--Bott theorem yielding the expected result.\\
    \\
    The second computation follows simply by the semiorthogonality of the full exceptional collection:
    \begin{equation}\label{eq_collectionF235}
        \begin{split}
            D^b(F) = \langle & \Oc(0,0), \Qc_3(0,0),\dots\dots\dots, \Oc(0,4), \Qc_3(0,4),\\
                             & \Oc(1,1), \Qc_3(1,1),\dots\dots\dots, \Oc(1,5), \Qc_3(1,5),\\
                             & \Oc(2,2), \Qc_3(2,2),\dots\dots\dots, \Oc(2,6), \Qc_3(2,6) \rangle
        \end{split}
    \end{equation}
    which follows from an application of \cite[Corollary 2.7]{orlovblowup}.
\end{proof}

A similar result can be obtained with the same argument:

\begin{lemma}\label{vanishingOO}
The following relation holds for $2\leq k\leq 4$:
\begin{equation}
    \begin{split}
        \Ext_F^\bullet(\Oc(0,k), \Oc(1,1)) = \Ext_M^\bullet(\Oc(0,k), \Oc(1,1)) & = 0 \\
        \Ext_F^\bullet(\Oc(2,2), \Oc(0,k)) & = 0
    \end{split}
\end{equation}

\end{lemma}

\begin{proof}
    For the first claim, we consider the sequence on $F$
    \begin{equation}
        0\arw\Oc(0, -k)\arw\Oc(1, 1-k)\arw\Oc_m(1, 1-k)\arw 0
    \end{equation}
    and the vanishing follow from applying the Borel--Weil--Bott theorem to the first two terms. The second vanishing is again a consequence of the collection \ref{eq_collectionF235}
\end{proof}

\begin{comment}
We also prove the following statement, which will be necessary to verify Condition \ref{cond_mutationsroof} :
\begin{lemma}\label{lem_vanishingOOQOsecondclaim1}
    For every integer $k$ and for non negative integers $a, b$ which satisfy $1+a\leq b \leq 4+a$ one has:
    \begin{equation}
        \begin{split}
            \Ext_F^\bullet(\Qc_3(k+1, k+1+b), \Oc(k+1, k+1+a)) & = 0\\
            \Ext_F^\bullet(\Oc(k+1, k+1+a), \Qc_3(k, k+b)) & = 0.
        \end{split}
    \end{equation}
    Moreover, for every integer $k$ and non negative integers $a, b$ which satisfy $3+a\leq b \leq 6+a$ one has:
    \begin{equation}
        \begin{split}
            \Ext^\bullet_F(\Oc(2+k, b+1+k), \Oc(2+k, 2+a+k)) &= 0\\
            \Ext^\bullet_F(\Oc(2+k, 2+a+k), \Oc(1+k, b+k)) &= 0.
        \end{split}
    \end{equation}
\end{lemma}
\begin{proof}
    All the claims are proven by semiorthogonality of the collection \ref{eq_collectionF235}
\end{proof}
\end{comment}

\noindent Another useful vanishing condition comes from the Leray spectral sequence and the Koszul resolution of $M$:

\begin{lemma}\label{vanishingpushforward}
    Let $E_1$ and $E_2$ be vector bundles on $F$ such that they are pullbacks of vector bundles on $G(2,V_5)$. Then the following relation holds:
    \begin{equation}
        \operatorname{Ext}^\bullet_F( E_1(0,1),  E_2) = \operatorname{Ext}^\bullet_M( E_1(0,1),  E_2) = 0
    \end{equation}
    Moreover, if $\Ext^\bullet_F(E_2(1,1), E_1) = \Ext^\bullet_M(E_2, E_1) = 0$ one has: 
    \begin{equation}\label{eq_secondvanishingsecondclaim}
        \Ext^\bullet_F(E_2, E_1) = 0.
    \end{equation}
    The same kind of result holds for pullbacks from $G(3,V_5)$.
\end{lemma}

\begin{proof}
    The Koszul resolution of $M$ gives the sequence:
    
    \begin{equation}
        0\arw h_1^*(E_1^\vee\otimes E_2(-1))\otimes h_2^*\Oc(-2)\arw h_1^*(E_1^\vee\otimes E_2)\otimes h_2^*\Oc(-1)\arw h_1^*(E_1^\vee\otimes E_2)\otimes h_2^*\Oc(-1)|_M\arw 0.
    \end{equation}
    
    Both the first and the second term have no cohomology because their pushforward vanishes: For instance, consider the first one:
    
    \begin{equation}
        \begin{split}
            h_{1*} h_1^*(E_1^\vee\otimes E_2(-1))\otimes h_2^*\Oc(-2) &= E_1^\vee\otimes E_2(-1)\otimes h_{1*} h_2^*\Oc(-2)
        \end{split}
    \end{equation}
    and on every stalk the sheaf  $h_{1*} h_2^*\Oc(-2)$ vanishes. The same argument holds for the second term of the sequence. The proof for the statement about pullbacks from $G(3, V_5)$ is identical.\\
    \\
    The second part of the lemma follows by applying $\Hom_F(E_2, -)$ to the Koszul resoluton of $E_1|_M$ once we observe that by the previous part we get $\Ext^\bullet_F(E_2(1,1), E_1)$=0.
\end{proof}

\begin{lemma}\label{mutationUQ}
We have the following mutations in the derived category of $M$ for every choice of the integers $a,b$ and for $k\in\{1;2\}$:
\begin{equation}
    \begin{split}
        \LL_{\mathcal O(a,b)}\mathcal U^\vee_k(a,b) &=\mathcal Q^\vee_k(a,b) \\
        \RR_{\mathcal O(a,b)}\mathcal Q^\vee_k(a,b) &=\mathcal U^\vee_k(a,b)
    \end{split}
\end{equation}
and they satisfy Condition \ref{cond_mutationsroof} .
\end{lemma}
\begin{proof}
Let us compute $\LL_{\mathcal O(a,b)}\mathcal U^\vee_k(a,b)$. The following result
\begin{equation}
\operatorname{Ext}^\bullet_M(\Oc(a,b), \Uc^\vee_k(a,b)) = V_5^\vee[0]
\end{equation}
follows from the Borel--Weil--Bott theorem, it tells us that the  mutation we are interested in is the cone of the morphism
\begin{equation}
V_5^\vee\otimes\mathcal O(a,b)\longrightarrow \Uc_k^\vee(a,b).
\end{equation}
From the dual of the universal sequence
\begin{equation}
0\longrightarrow\mathcal U\longrightarrow V_5\otimes\mathcal O\longrightarrow\mathcal Q\longrightarrow 0
\end{equation}
we see that the morphism is surjective, thus the cone yields the kernel $\Qc_k^\vee(a,b)$. The mutation $\RR_{\mathcal O(a,b)}\mathcal Q^\vee_k(a,b)$ follows from an identical argument.\\
\\
Both mutations satisfies Condition \ref{cond_mutationsroof}  because of the relations:
\begin{equation}
    \begin{split}
        \Ext^\bullet_F(\Oc, \Uc_k^\vee) &= \Ext^\bullet_F(\Uc_k^\vee(1,1), \Oc) = 0 \\
        \Ext^\bullet_F(\Qc_k^\vee, \Oc) &= \Ext^\bullet_F(\Oc(1,1), \Qc_k^\vee) \,= 0 \\
    \end{split}
\end{equation}
which can be verified by the Borel--Weil--Bott theorem.
\end{proof}
\begin{lemma}\label{getridoftilde}
In the derived category of $M$, for every $a$ and $b$ one has the following mutations, which satisfy Condition \ref{cond_mutationsroof} :
\begin{equation}
    \begin{split}
        \RR_{\mathcal O(a+1, b-1)}\mathcal Q_3(a,b) &= \mathcal Q_2(a,b)\\
        \RR_{\mathcal O(a+1, b-1)}\mathcal U_3(a,b) &= \mathcal U_2(a,b)\\
        \LL_{\mathcal O(a-1, b+1)}\mathcal Q_3^\vee(a,b) &= \mathcal Q_2^\vee(a,b)\\
        \LL_{\mathcal O(a-1, b+1)}\mathcal U_3^\vee(a,b) &= \mathcal U_2^\vee(a,b).
    \end{split}
\end{equation}
\end{lemma}
\begin{proof}
With the Borel--Weil--Bott theorem we have:
\begin{equation}
    \begin{split}
        \operatorname{Ext}^\bullet_M(\mathcal Q_3(a,b), \mathcal O(a+1,b-1)) &= \mathbb C[-1]\\
        \operatorname{Ext}^\bullet_F(\mathcal Q_3(a+1,b+1), \mathcal O(a+1,b-1)) &= 0\\
        \operatorname{Ext}^\bullet_F(\mathcal O(a+1,b-1), \mathcal Q_3(a,b)) &= 0
    \end{split}
\end{equation}
so $\RR_{\mathcal O(a+1, b-1)}\mathcal Q_3(a,b)$ is an extension and satisfies Condition \ref{cond_mutationsroof} . The relevant exact sequence is
\begin{equation}\label{sequenceOQQ}
0\longrightarrow\mathcal O(1,-1)\longrightarrow\mathcal Q_2\longrightarrow\mathcal Q_3\longrightarrow 0,
\end{equation}
which can be found computing the rank one cokernel of the injection $\mathcal U_2\xhookrightarrow{}\mathcal U_3$, comparing the universal sequences of the two Grassmannians and applying the Snake Lemma, this proves our first claim.\\
In order to verify the second one, we write the sequence involving the injection between the universal bundles, which is
\begin{equation}\label{sequenceUUO}
0\longrightarrow\mathcal U_2\longrightarrow\mathcal U_3\longrightarrow\mathcal O(1,-1)\longrightarrow 0.
\end{equation}
The related Ext, in this case, is $\mathbb C[0]$, so the mutation is the cone of the morphism $\Uc_3\longrightarrow\mathcal O(1,-1)$, yielding the desired result. Condition \ref{cond_mutationsroof}  is satisfied due to the following cohomological results:

\begin{equation}
    \begin{split}
        \operatorname{Ext}^\bullet_F(\mathcal U_3(a+1,b+1), \mathcal O(a+1,b-1)) &= 0\\
        \operatorname{Ext}^\bullet_F(\mathcal O(a+1,b-1), \mathcal U_3(a,b)) &= 0.
    \end{split}
\end{equation}
The proof for the last two mutations follow from the same arguments applied to the duals of Equations \ref{sequenceUUO} and \ref{sequenceOQQ} and similar cohomological computations.
\end{proof}

\begin{proposition}\label{prop_derivedeqAG4}
Let $Y_1$ and $Y_2$ be the zero loci of the pushforwards of a general $s\in H^0(F,\mathcal O(1,1))$. Then there is a composition of mutations satisfying Condition \ref{cond_mutationsroof} , which yields an equivalence of categories 
\begin{equation}\label{eq_collectionF235_aftermutations}
    \begin{split}
    \left<\phi_2 D^b(Y_2), D^b(G(3,V_5))\otimes\Oc(1,1), D^b(G(3,V_5))\otimes\Oc(2,2)\right>  \xrightarrow{\sim}\\ 
    \left<\psi D^b(Y_2), D^b(G(2,V_5))\otimes\Oc(1,1), D^b(G(2,V_5))\otimes\mathcal O(2,2)\right>
    \end{split}
\end{equation}
where $\psi$ is given by a composition of mutations and $\phi_2 =k_{2*}\nu_2^*$ in the notation of Diagram \ref{eq_diagramfiber}. Moreover, $Y_1$ and $Y_2$ are derived equivalent.
\end{proposition}

\begin{proof}
The idea of the proof is writing a full exceptional collection for $M$ in a way such that we can use our cohomology vanishing results to transport line bundles $\mathcal O(a+1,b-1)$ to the immediate right of $\mathcal Q_3(a,b)$, then use Lemma \ref{getridoftilde} to get rid of $\mathcal Q_3(a,b)$, thus transforming pullbacks of vector bundles on $G(2,V_5)$ to pullbacks of vector bundles on $G(3,V_5)$.

We start from:

\begin{align*}
D^b(M)= \langle & \psi_1 D^b(Y_2), \mathcal O, \mathcal Q_3,\mathcal O(0,1), \mathcal Q_3(0,1),\mathcal O(0,2), \mathcal Q_3(0,2),\mathcal O(0,3), \mathcal Q_3(0,3), \mathcal O(0,4), \mathcal Q_3(0,4), \\ 
&\mathcal O(1,1), \mathcal Q_3(1,1),\mathcal O(1,2), \mathcal Q_3(1,2),\mathcal O(1,3), \mathcal Q_3(1,3),\mathcal O(1,4), \mathcal Q_3(1,4),\mathcal O(1,5), \mathcal Q_3(1,5)\rangle
\end{align*}

which is obtained by applying the Cayley trick to the collection \ref{kuznetsov35} for $\dbcoh(G(3, V_5)$, then twisting the whole decomposition by $\Oc(-1,-1)$. We defined $\psi_1$ as $\psi_1  = \phi_2(-)\otimes\Oc(-1,-1)$.\\
\\
Our first operation is moving the first five bundles past $\psi_1 D^b(Y_2)$, then  sending them to the end: they get twisted  by the anticanonical bundle of $M$, which, with the adjunction formula, can be shown to be $\omega_M^\vee=\mathcal O(2,2)$.
\begin{align*}
D^b(M)= \langle & \psi_2 D^b(Y_2), \mathcal Q_3(0,2),\mathcal O(0,3), \mathcal Q_3(0,3),\mathcal O(0,4), \mathcal Q_3(0,4), \mathcal O(1,1), \mathcal Q_3(1,1),
\mathcal O(1,2), \mathcal Q_3(1,2),\\
&\mathcal O(1,3), \mathcal Q_3(1,3),\mathcal O(1,4), \mathcal Q_3(1,4),\mathcal O(1,5), \mathcal Q_3(1,5),\mathcal O(2,2), \mathcal Q_3(2,2),\mathcal O(2,3), \mathcal Q_3(2,3),\mathcal O(2,4)\rangle
\end{align*}
where we introduced the functor
\begin{equation}
\psi_2 = \RR_{\left<\mathcal O(0,0), \mathcal Q_3(0,0),\mathcal O(0,1), \mathcal Q_3(0,1),\mathcal O(0,2)\right>}\psi_1
\end{equation}
Applying Lemma \ref{vanishingQO}, we observe that $\mathcal O(1,1)$  can be moved next to $\mathcal Q_3(0,2)$ and these mutations satisfy Condition \ref{cond_mutationsroof}. Then we can use Lemma \ref{getridoftilde} transforming $\mathcal Q_3(0,2)$ in $\mathcal Q_2(0,2)$. This can be done twice due to the invariance of the operation up to overall twists, yielding:
\begin{align*}
D^b(M)=\langle &\psi_2 D^b (Y_2), \mathcal O(1,1),\mathcal Q_2(0,2),\mathcal O(0,3), \mathcal Q_3(0,3),\mathcal O(0,4), \mathcal Q_3(0,4), \mathcal Q_3(1,1),  \mathcal O(1,2), \mathcal Q_3(1,2),\\
& \mathcal O(1,3), \mathcal O(2,2),\mathcal Q_2(1,3),\mathcal O(1,4), \mathcal O(2,3), \mathcal Q_3(1,4),\mathcal O(1,5), \mathcal Q_3(1,5), \mathcal Q_3(2,2), \mathcal Q_3(2,3),\mathcal O(2,4)\rangle.
\end{align*}
The next step is to move $\Oc(1,2)$ one step to the left. Since $\Qc_3(1,1)\simeq\Qc_3^\vee(1,2)$ the result follows from Lemma \ref{mutationUQ}. Then, since by Lemma \ref{vanishingQO} and \ref{vanishingOO} $\Oc(1,2)$ is orthogonal to $\mathcal O(0,4)$ and $\mathcal Q_3(0,4)$ we apply Lemma \ref{getridoftilde} to transform $\mathcal Q_3(0,3)$ in $\mathcal Q_2(0,3)$. Again, all these operations can be performed twice (by invariance to overall twist) and they fulfill Condition \ref{cond_mutationsroof} :
\begin{align*}
D^b(M) = \langle & \psi_2 D^b(Y_2), \mathcal O(1,1),\mathcal Q_2(0,2),\mathcal O(0,3),\mathcal O(1,2),\mathcal Q_2(0,3),\mathcal O(0,4), \mathcal Q_3(0,4), \mathcal U_3^\vee(1,2), \mathcal Q_3(1,2),\\
& \mathcal O(1,3),\mathcal O(2,2),\mathcal Q_2(1,3),\mathcal O(1,4), \mathcal O(2,3), \mathcal Q_2(1,4),\mathcal O(1,5), \mathcal Q_3(1,5), \mathcal U_3^\vee(2,3), \mathcal Q_3(2,3),\mathcal O(2,4)\rangle.
\end{align*}
By Lemma \ref{vanishingOO} let us move $\Oc(0,3)$ one step to the right. Then, by a similar application of the Borel--Weil--Bott theorem we observe that $\Uc_3^\vee(1,2)$ is orthogonal to the two bundles at its left, hence we can move it to the immediate right of $\Qc_2(0,4)$. These operations fulfill Condition \ref{cond_mutationsroof}  due to the following vanishings:
\begin{equation}
    \begin{split}
        \Ext^\bullet_F(\Uc_3(1,2), \Qc_3(0,4)) &= 0\\
        \Ext^\bullet_F(\Uc_3(1,2), \Oc(0,4)) &= 0\\
        \Ext^\bullet_F(\Qc_3(1,5), \Uc_3(1,2)) &= 0\\
        \Ext^\bullet_F(\Oc(1,5), \Uc_3(1,2)) &= 0.
    \end{split}
\end{equation}
We can then move $\Uc_3^\vee(1,2)$ one additional step to the left using Lemma \ref{getridoftilde}. Applying the same sequence of mutations to the $\Oc(1,1)$-twist of these objects we get the following collection:
\begin{align*}
D^b(M)=\langle & \psi_2 D^b(Y_2), \mathcal O(1,1),\mathcal Q_2(0,2),\mathcal O(1,2),\mathcal U_2(0,3),\mathcal U_2^\vee(1,2),\mathcal O(0,3),\mathcal O(0,4), \mathcal Q_3(0,4),  \mathcal Q_3(1,2),\\
& \mathcal O(1,3),\mathcal O(2,2),\mathcal Q_2(1,3),\mathcal O(2,3), \mathcal U_2(1,4), \mathcal U_2^\vee(2,3),\mathcal O(1,4),\mathcal O(1,5), \mathcal Q_3(1,5),  \mathcal Q_3(2,3),\mathcal O(2,4)\rangle.
\end{align*}
Again, thanks to Lemma \ref{mutationUQ}, one has $\RR_{\mathcal O(1,3)}\Qc_3(1,2)\simeq\Uc_3^\vee(1,3)$, then we can apply Lemma \ref{getridoftilde} to transform $\mathcal Q_3(0,4)$ in $\mathcal Q_2(0,4)$. But then $\mathcal O(1,3)$ ends up next to $\mathcal O(0,4)$, which is orthogonal to it by application of Lemma \ref{vanishingOO}, so they can be exchanged. Passing through $\mathcal Q_2(0,4)$ via Lemma \ref{mutationUQ} and mutating it to $\mathcal U_2(0,4)$, $\mathcal O(0,4)$ goes right next to $\mathcal U_3^\vee(1,3)$, which is mutated to $\mathcal U_2^\vee(1,3)$ by applying Lemma \ref{getridoftilde}. Again all these mutations satisfy Condition \ref{cond_mutationsroof} .\\
\\
Once we have done the same for the $\Oc(1,1)$-twists, we have transformed all the rank 2 and rank 3 pullbacks from $G(3,V_5)$ in pullbacks from $G(2,V_5)$. %What we still need to do is to remove the twists involving powers of the hyperplane class of $G(3,V_5)$ and, consequently, recognize an exceptional collection of $G(2,V_5)$ and its twist.
Removing all the duals by the identity $\Qc_2\simeq\wedge^2\Qc_2^\vee\otimes\det(\Qc_2)$ and the analogous one for $\Uc_2$, we get the following result:
\begin{equation}\label{eq_collectionbeforereordering}
    \begin{split}
        D^b(M) = \langle & \psi_2 D^b(Y_2), \mathcal O(1,1),\mathcal Q_2(0,2),\mathcal O(1,2),\mathcal U_2(0,3),\mathcal U_2(2,2),\mathcal O(0,3),\mathcal O(1,3), \mathcal U_2(0,4),\mathcal U_2(2,3),\\
        & \mathcal O(0,4), \mathcal O(2,2),\mathcal Q_2(1,3),\mathcal O(2,3), \mathcal U_2(1,4), \mathcal U_2(3,3),\mathcal O(1,4),\mathcal O(2,4),\mathcal U_2(1,5),\mathcal U_2(3,4),\mathcal O(1,5)\rangle.
    \end{split}
\end{equation}
First we send $\mathcal O(1,1)$ to the end, then we use Lemma \ref{vanishingpushforward} and the fact that $\Oc(2,-2)$ and $\mathcal U_2^\vee(2,-2)$ are acyclic to order the bundles by their power of the second twist:
\begin{equation}\label{eq_collectionafterreordering}
    \begin{split}
        D^b(M)= \langle & \psi_3 D^b(Y_2), \mathcal Q_2(0,2),\mathcal O(1,2),\mathcal U_2(2,2),\mathcal O(2,2),\mathcal U_2(0,3),\mathcal O(0,3),\mathcal O(1,3), \mathcal U_2(2,3), \mathcal Q_2(1,3), \\
        & \mathcal O(2,3), \mathcal U_2(3,3),\mathcal O(3,3)\mathcal U_2(0,4),\mathcal O(0,4), \mathcal U_2(1,4), \mathcal O(1,4),\mathcal O(2,4),\mathcal U_2(3,4),\mathcal U_2(1,5),\mathcal O(1,5)\rangle
    \end{split}
\end{equation}

where we defined
\begin{equation}
    \psi_3 = \RR_{\mathcal O(1,1)}\psi_2.
\end{equation}
Note that each time we exchanged objects $E_1(a,b), E_2(c,d)$ because $d-b = -1$ in order to get the collection \ref{eq_collectionafterreordering}, such mutations satisfy Condition \ref{cond_mutationsroof}  by Lemma \ref{vanishingpushforward} because $E_2(c+1,d+1)$ is semiorthogonal to $E_1(a,b)$.  This semiorthogonality can be easily checked in the collection \ref{eq_collectionbeforereordering} (using the Serre functor if necessary).\\
\\
Now we send the last 10 objects to the beginning and reorder again the collection with respect to the second twist, obtaining the following:
\begin{align*}
 D^b(M)=\langle & \psi_4 D^b(Y_2), \mathcal U_2(1,1),\mathcal O(1,1),\mathcal U_2(-2,2),\mathcal O(-2,2),\mathcal U_2(-1,2),\mathcal O(-1,2),\mathcal O(0,2), \mathcal U_2(1,2),\mathcal Q_2(0,2),\\
 & \mathcal O(1,2),\mathcal U_2(2,2),\mathcal O(2,2), \mathcal U_2(-1,3),\mathcal O(-1,3), \mathcal U_2(0,3),\mathcal O(0,3),\mathcal O(1,3),\mathcal U_2(2,3),\mathcal Q_2(1,3),\mathcal O(2,3)\rangle,
\end{align*}
where
\begin{equation}
\psi_4 = \LL_{\left<\mathcal U_2(1,1),\mathcal O(1,1),\mathcal U_2(-2,2),\mathcal O(-2,2),\mathcal U_2(-1,2),\mathcal O(-1,2),\mathcal O(0,2),\mathcal U_2(1,2),\mathcal U_2(-1,3),\mathcal O(-1,3)\right>}\psi_3
\end{equation}
Now we observe that $\mathcal Q_2(0,2)$ is orthogonal to $\mathcal U_2(1,2)$, so they can be exchanged: this allows to mutate $\mathcal Q_2(0,2)$ to $\mathcal U_2(0,2)$ sending it one step to the left (this mutation satisfies Condition \ref{cond_mutationsroof}  by a simple Borel--Weil--Bott computation). After doing the same thing with $\mathcal O(1,1)$--twists of these bundles, the last steps are sending the first two bundles to the end and twisting everything by $\mathcal O(-1,-1)$. We get:
\begin{equation}\label{eq_finalcollectionG25}
    \begin{split}
        D^b(M)=\langle & \psi_5 D^b(Y_2), \mathcal U_2(-3,1),\mathcal O(-3,1),\mathcal U_2(-2,1),\mathcal O(-2,1),\mathcal U_2(-1,1),\mathcal O(-1,1), \mathcal U_2(0,1), \mathcal O(0,1),\mathcal U_2(1,1),\\
        & \mathcal O(1,1), \mathcal U_2(-2,2),\mathcal O(-2,2),\mathcal U_2(-1,2), \mathcal O(-1,2), \mathcal U_2(0,2),\mathcal O(0,2),\mathcal U_2(1,2),\mathcal O(1,2), \mathcal U_2(2,2),\mathcal O(2,2)\rangle
    \end{split}
\end{equation}
where we defined the functor
\begin{equation}
\psi_5 = \mathcal T(-1,-1)\RR_{\left<\mathcal U_2(1,1),\mathcal O(1,1)\right>}\psi_4
\end{equation}
where $\mathcal T(-1,-1)$ is the twist with $\mathcal O(-1,-1)$.\\

Observe that, by the isomorphism $\Uc_2\simeq\Uc_2^\vee(-1)$ and the fact that $\omega_{G(2, V_5)}\simeq\Oc(-5)$, one has:
\begin{equation}
    \begin{split}
        \dbcoh(G(2, V_5)) = \langle \mathcal U_2(-4),\mathcal O(-4),\mathcal U_2(-3),\mathcal O(-3),\mathcal U_2(-2,0),\mathcal O(-2), \mathcal U_2(-1), \mathcal O(-1),\mathcal U_2(0),\mathcal O(0)\rangle
    \end{split}
\end{equation}
Hence the collection \ref{eq_finalcollectionG25} has the form
\begin{equation}
    \dbcoh(M) = \langle\psi_5\dbcoh(Y_2), \dbcoh(G(2, V_5))\otimes\Oc(1,1), \dbcoh(G(2, V_5))\otimes\Oc(2,2)\rangle.
\end{equation}
The proof is completed once we compare this last decomposition with the following one, obtained by Cayley trick:
\begin{equation}
    \dbcoh(M) = \langle\phi_1\dbcoh(Y_1), \dbcoh(G(2, V_5))\otimes\Oc(1,1), \dbcoh(G(2, V_5))\otimes\Oc(2,2)\rangle.
\end{equation}
\end{proof}

\subsection{Derived equivalence for the roof of type \texorpdfstring{$G_2$}{something}}
The roof of type $G_2$ is the complete flag $F$ of type $G_2$, which admits projections $h_1$ and $h_2$ to the two $G_2$-Grassmannians $G/P_1$ and $G/P_2$ of dimension five, where the former is a smooth quadric.\\
\\
The main result of this section is due to Kuznetsov in the paper \cite{kuznetsovimou}, where by a sequence of mutations on a suitable semiorthogonal decomposition of $M$, a derived equivalence for the associated Calabi--Yau pair has been given. What we summarize here is the same argument, with some minor variation to make the result compatible with Condition \ref{cond_mutationsroof}, and therefore to Theorems \ref{thm_derivedequivalencefibrations} and \ref{thm_derivedequivalencekequivalence}. First, one has the following semiorthogonal decompositions for the quadric \cite{kapranov}:
\begin{equation}
    \dbcoh(G/P_1)= \langle \Oc, \Sc^\vee, \Oc(1), \Oc(2), \Oc(3), \Oc(4), \rangle
\end{equation}
where $\Sc$ is the spinor bundle. By Serre functor, once we observe that $\omega_{G/P_1} = \Oc(-5)$, we write:
\begin{equation}
    \dbcoh(G/P_1)= \langle \Oc(-3), \Oc(-2), \Oc(-1), \Oc, \Sc^\vee, \Oc(1) \rangle.
\end{equation}
One has $\Ext^\bullet_{G/P_1}(\Oc, \Sc^\vee) \simeq\CC^8[0]$ by the Borel--Weil--Bott theorem, and in light of the following exact sequence \cite[Theorem 2.8]{ottaviani}:

\begin{equation}
    0\arw\Sc\arw\Oc^{\oplus 8}\arw\Sc^\vee\arw 0
\end{equation}

we compute $\LL_{\Oc}\Sc^\vee \simeq\Sc$, hence the following decomposition \cite[Equation 6]{kuznetsovimou}:

\begin{equation}
    \dbcoh(G/P_1)= \langle \Oc(-3), \Oc(-2), \Oc(-1), \Sc, \Oc, \Oc(1) \rangle.
\end{equation}

On the other hand, for the second $G_2$-Grassmannian there is the full exceptional collection \cite[Section 6.4]{kuznetsovhyperplane}:

\begin{equation}
    \dbcoh(G/P_2)= \langle \Oc, \Uc^\vee, \Oc(1), \Uc^\vee(1), \Oc(2), \Uc^\vee(2) \rangle
\end{equation}

where $\Uc$ is the tautological bundle. By Serre functor (since $\omega_{G/P_2}\simeq \Oc(-3)$) and by the isomorphism $\Uc^\vee\simeq\Uc(1)$ we write:

\begin{equation}\label{eq_kuznetsovscollectionforG/P1}
    \dbcoh(G/P_2)= \langle \Oc(-1), \Uc, \Oc, \Uc^\vee, \Oc(1), \Uc^\vee(1) \rangle
\end{equation}

By Cayley trick one has the following semiorthogonal decompositions:

\begin{equation}\label{eq_G2withthequadric}
    \begin{split}
        \dbcoh(M) = \langle \phi_1\dbcoh(Y_1), & \Oc_M(-2,1), \Oc_M(-1,1), \Oc_M(0,1), l^*\Sc(1,1), \Oc_M(1,1), \Oc_M(2,1)\rangle.
    \end{split}
\end{equation}
\begin{equation}\label{eq_G2startingpointforme}
    \begin{split}
        \dbcoh(M) = \langle \phi_2\dbcoh(Y_2), & \Oc_M(1,0), l^*\Uc(1,1), \Oc_M(1,1), l^*\Uc^\vee(1,1), \Oc_M(1,2), l^*\Uc^\vee(1,2)\rangle
    \end{split}
\end{equation}
For later convenience, let us also write the following exceptional collections for $F$, which are an application of \cite[Corollary 2.7]{orlovblowup} to the two projective bundle structures of $F$:
\begin{equation}\label{eq_G2projbundlecollections}
    \begin{split}
        \dbcoh(F) = \langle & \Oc(-3,0), \Oc(-2,0), \Oc(-1,0), \Sc(0,0), \Oc(0,0), \Oc(1,0), \Oc(-2,1), \Oc(-1,1), \Oc(0,1), \Sc(1,1), \Oc(1,1), \Oc(2,1)\rangle\\
                  = \langle & \Oc(0,0), \Uc^\vee(0,0), \Oc(0,1), \Uc^\vee(0,1), \Oc(0,2), \Uc^\vee(0,2), \Oc(1,1), \Uc^\vee(1,1), \Oc(1,2), \Uc^\vee(1,2), \Oc(1,3), \Uc^\vee(1,3)\rangle.
    \end{split}
\end{equation}
We are now ready to formulate the following result, as a corollary to \cite[Theorem 5]{kuznetsovimou}

\begin{corollary}\label{cor_derivedequivalenceG2}
    In the setting above, there is a sequence of mutations of exceptional objects of $\dbcoh(M)$ satisfying Condition \ref{cond_mutationsroof}  and giving:
    \begin{equation}%\label{eq_G2withthequadric}
        \begin{split}
            \langle \phi_2\dbcoh(Y_2), & \Oc_M(1,0), l^*\Uc(1,1), \Oc_M(1,1), l^*\Uc^\vee(1,1), \Oc_M(1,2), l^*\Uc^\vee(1,2)\rangle\\
            &= \langle \psi\dbcoh(Y_2), \Oc_M(-2,1), \Oc_M(-1,1), \Oc_M(0,1), l^*\Sc(1,1), \Oc_M(1,1), \Oc_M(2,1)\rangle.
        \end{split}
    \end{equation}
\end{corollary}

\begin{proof}
    Since $\omega_M = \Oc_M(-1,-1)$, by applying the Serre functor to the last six objects of the collection \ref{eq_G2startingpointforme} we find:
    \begin{equation}\label{eq_G2startingpointforkuznetsov}
        \begin{split}
            \dbcoh(M) = \langle \Oc_M(0,-1), l^*\Uc(0,0), \Oc_M(0,0), l^*\Uc^\vee(0,0), \Oc_M(0,1), l^*\Uc^\vee(0,1), \phi_2\dbcoh(Y_2)\rangle
        \end{split}
    \end{equation}
    which can also be derived directly by applying Orlov's blowup formula \cite[Theorem 4.3]{orlovblowup} to the collection \ref{eq_kuznetsovscollectionforG/P1}, since $M$ is the blowup of $G/P_1$ in $Y_1$, as it is explained in \cite{kuznetsovimou}.\\
    The first step is to send $\Oc_M(0,1)$ and $l^*\Uc^\vee(0,1)$ to the beginning by Serre functor, we get:
    \begin{equation}
        \begin{split}
            \dbcoh(M) = \langle \Oc_M(-1,0), l^*\Uc^\vee(-1,0), \Oc_M(0,-1), l^*\Uc(0,0), \Oc_M(0,0), l^*\Uc^\vee(0,0), \psi_1\dbcoh(Y_2)\rangle
        \end{split}
    \end{equation}
    where we defined $\psi_1:= \LL_{\langle\Oc_M(0,1), l^*\Uc^\vee(0,1)\rangle}\phi_2$. The next step is to move $\Oc_M(0,-1)$ to the beginning of the collection. By the Koszul resolution of $M$ and the results of \cite[Lemma 1, Corollary 2]{kuznetsovimou} one finds $\Ext^\bullet_F(\Oc(-1,0), \Oc(0,-1)) = \Ext^\bullet_M(\Oc_M(-1,0), \Oc_M(0,-1)) = 0$ and $\Ext^\bullet_F(\Uc^\vee(-1,0), \Oc(0,-1)) = \Ext^\bullet_M(l^*\Uc^\vee(-1,0), \Oc_M(0,-1)) = 0$. Therefore we have $\LL_{\langle \Oc_M(-1,0), l^*\Uc^\vee(-1,0)\rangle}\Oc_M(0,-1) \simeq \Oc_M(0,-1)$, and by semiorthogonality of \ref{eq_G2projbundlecollections} these mutations satisfy Condition \ref{cond_mutationsroof} .\\
    We get the following collection:
    \begin{equation}
        \begin{split}
            \dbcoh(M) = \langle \Oc_M(0,-1), \Oc_M(-1,0), l^*\Uc^\vee(-1,0), l^*\Uc(0,0), \Oc_M(0,0), l^*\Uc^\vee(0,0), \psi_1\dbcoh(Y_2)\rangle
        \end{split}
    \end{equation}
    Again by applying \cite[Lemma 1, Corollary 2]{kuznetsovimou} we have the vanishings $\Ext^\bullet_F(\Uc^\vee(0,1), \Uc(0,0)) = 0$ and \linebreak $\Ext^\bullet_M(l^*\Uc^\vee(-1,0), l^*\Uc(0,0)) = \CC[-1]$, while by semiorthogonality of the collections of Equation \ref{eq_G2projbundlecollections} also \linebreak $\Ext^\bullet_F \Uc(0,0), (\Uc^\vee(-1,0)) = 0$ holds, hence the mutation $\LL_{l^*\Uc^\vee(-1,0)}l^*\Uc(0,0)$ satisfies Condition \ref{cond_mutationsroof} . The result is a rank four extension which is isomorphic to $\Sc$ \cite[Lemma 4]{kuznetsovimou}. We get:
    \begin{equation}
        \begin{split}
            \dbcoh(M) = \langle \Oc_M(0,-1), \Oc_M(-1,0), \Sc(0,0), l^*\Uc^\vee(-1,0), \Oc_M(0,0), l^*\Uc^\vee(0,0), \psi_1\dbcoh(Y_2)\rangle
        \end{split}
    \end{equation}
    The next operation is to move the first bundle to the end by means of the inverse Serre functor, and then move $\psi_1\dbcoh(Y_2)$ one step to the right. We find:
    \begin{equation}
        \begin{split}
            \dbcoh(M) = \langle \Oc_M(-1,0), \Sc(0,0),  l^*\Uc^\vee(-1,0), \Oc_M(0,0), l^*\Uc^\vee(0,0), \Oc_M(1,0), \psi_2\dbcoh(Y_2)\rangle
        \end{split}
    \end{equation}
    where we defined $\psi_2 = \RR_{\Oc_M(1,0)}\psi_1$.\\
    By the results of \cite[Lemma 1, Corollary 2]{kuznetsovimou} it follows that  $\Ext_F^\bullet(\Uc^\vee(0,1), \Oc) = 0$ and $\Ext_M^\bullet(l^*\Uc^\vee(-1,0), \Oc_M) = \CC[0]$, while one finds $\Ext_F^\bullet(\Oc, \Uc^\vee(-1,0), ) = 0$ as a consequence of the semiorthogonality of \ref{eq_G2projbundlecollections}. Hence, we can define a mutation $\RR_{l^*\Uc^\vee(-1,0)}\Oc_M$ which satisfies Condition \ref{cond_mutationsroof} . We can compute such mutations by means of the following short exact sequence \cite[Equation 5]{kuznetsovimou}:
    \begin{equation}
        0\arw \Oc(-1,1)\arw \Uc^\vee \arw\Oc(1,0)\arw 0.
    \end{equation}
    The result is $\RR_{l^*\Uc^\vee(-1,0)}\Oc_M \simeq \Oc_M(-2,1)$. In the same way we find $\RR_{l^*\Uc^\vee(0,0)}\Oc_M(1,0) \simeq \Oc_M(-1,1)$ and the semiorthogonal decomposition becomes:
    \begin{equation}
        \begin{split}
            \dbcoh(M) = \langle \Oc_M(-1,0), \Sc(0,0),  \Oc_M(0,0), \Oc_M(-2,1), \Oc_M(1,0), \Oc_M(-1,1), \psi_2\dbcoh(Y_2)\rangle
        \end{split}
    \end{equation}
    Observe that by the vanishing of the cohomology of $\Oc_M(3, -1)$, $\Oc(2, -2)$ and $\Oc(-3,0)$ we see that the mutation obtained by exchanging $\Oc_M(-2,1)$ with $\Oc_M(1,0)$ satisfies Condition \ref{cond_mutationsroof} . After applying such operation, let us move $\psi_2\dbcoh(Y_2)$ two steps to the left, then send the block $\Oc_M(-1,0), \Sc(0,0),  \Oc_M(0,0), \Oc_M(1,0)$ to the end. Summing all up, we obtain:
    \begin{equation}
        \begin{split}
            \dbcoh(M) = \langle \psi_3\dbcoh(Y_2), & \Oc_M(-2,1), \Oc_M(-1,1), \Oc_M(0,1), \hspace{15pt} \Sc(1,1), \Oc_M(1,1), \Oc_M(2,1)\rangle
        \end{split}
    \end{equation}
    where we introduced the functor $\psi_3 = \RR_{\langle\Oc_M(-2,1), \Oc_M(-1,1)\rangle}\psi_2$. The proof is concluded by setting $\psi_3 = \psi$.
\end{proof}

%%%%%%%%%%%%%%%%%%%%%%%% BIBLIOGRAPHY %%%%%%%%%%%%%%%%%%%%%%%%%%%%%%%%%%%%%%%%%%%%%%%%%% 

%%%%%%%%%%%%%%%%%%%%%%%%%%%%%%%%%%%%%%%%%%%%%%%%%%%%%%%%%%%%%%%%%%%%%%%%%%%%%%%%%%%%%%%%  

\end{document}